\documentclass{article}
\usepackage{tikz}
\usetikzlibrary{shapes,arrows,chains}
\pdfoutput=1
\usepackage[utf8]{inputenc}
\usepackage{array}
\usepackage{amsmath, amsthm, amssymb,xcolor, amsfonts}
\usepackage{float}
\usepackage{subcaption}
\usepackage{soul}
\usepackage{thmtools}
\usepackage{enumitem}

\def\ve#1{\mathchoice{\mbox{\boldmath$\displaystyle\bf#1$}}
{\mbox{\boldmath$\textstyle\bf#1$}}
{\mbox{\boldmath$\scriptstyle\bf#1$}}
{\mbox{\boldmath$\scriptscriptstyle\bf#1$}}}

\newcommand\veb{{\ve b}}

\newcommand\ved{{\ve d}}

\newcommand\veg{{\ve g}}
\newcommand\veh{{\ve h}}

\newcommand\ves{{\ve s}}

\newcommand\veu{{\ve u}}
\newcommand\vev{{\ve v}}

\newcommand\vex{{\ve x}}
\newcommand\veX{{\ve X}}
\newcommand\vey{{\ve y}}
\newcommand\vez{{\ve z}}
\newcommand\veo{{\ve 0}}

\newcommand{\mh}{\mathcal{H}}

\newtheorem{theorem}{Theorem}[section]
\newtheorem{lemma}[theorem]{Lemma}
\newtheorem{corollary}[theorem]{Corollary}

\newtheorem{question}[theorem]{Question}

\newtheorem*{theorem*}{Theorem}
\newtheorem{definition}{Definition}
\newtheorem{proposition}[theorem]{Proposition}

\newcommand*{\myproofname}{Proof}

\newcommand{\Z}{{\mathbb Z}}

\newcommand{\R}{\mathbb{R}}
\newcommand{\Q}{\mathbb{Q}}
\newcommand{\supp}{\text{supp}}

\usepackage{bm}
\definecolor{cbpink}{HTML}{DC267F}
\definecolor{cbblue}{HTML}{648FFF}
\definecolor{cborange}{HTML}{FE6100}

\usepackage[margin=1.2in]{geometry}

\date{}
\begin{document}

\title{On Circuit Imbalance and 0/1 Circuits for Coloring and  Forest Problems\thanks{The work of Borgwardt and Lee was supported by the Air Force Office of Scientific Research under award number FA9550-24-1-0240. The work of Lee was additionally supported by the Office of Naval Research under award number N00014-24-1-2694.}}

\author{Steffen Borgwardt\thanks{\texttt{steffen.borgwardt@ucdenver.edu}; University of Colorado Denver, Department of Mathematical and Statistical Sciences} \and Nicholas Crawford\thanks{{\texttt{ncrawford@losmedanos.edu}; {Los Medanos College}}} \and Sean Kafer\thanks{\texttt{spkafer@ilstu.edu}; Illinois State University, Department of Mathematics} \and Jon Lee\thanks{\texttt{jonxlee@umich.edu}; University of Michigan} \and Angela Morrison\thanks{\texttt{angela.morrison1@ucalgary.ca}; University of Calgary, Department of Mathematics and Statistics}} 

\maketitle
\begin{abstract}
Circuits are fundamental objects in linear programming and oriented matroid theory, representing the elementary difference vectors of a polyhedron. A recent concept introduced by Dadush, Huiberts, Natura, and V\'egh, the \emph{circuit imbalance}, serves as a complexity measure relevant to iteration bounds for circuit-based augmentation and circuit diameters, and the interpretability of circuits in terms of the underlying application. In this paper, we analyze linear programming formulations of relaxed combinatorial optimization problems to prove two contrasting types of results related to the circuit imbalance. 

First, we identify simple and polynomially-sized constraint structures, in particular arising in graph-theoretic problems, that lead to an exponential circuit imbalance. These constructions show that, in quite general situations, working with the entire set of circuits poses significant challenges for an application of circuit augmentation or the study of circuit diameters, even with a compact description of the underlying polytope. 

Second, through a case study of two classic graph-theoretic problems with exponential imbalance\textemdash the vertex graph coloring problem and the maximum weight forest problem\textemdash we exhibit the existence of sets of highly interpretable circuits within the set of \emph{0/1 circuits} of (best-case) imbalance $1$. 
We prove that a restriction of circuit walks to these sets suffices to not only guarantee reachability of the integral extreme-points, but leads to a linear bound to walk between integral extreme-points for the fractional coloring polytope and a constant circuit diameter bound for the forest polytope, respectively. We conclude with a brief discussion of more general questions arising naturally from the use of 0/1 circuits in reachability and diameter studies. 
\end{abstract}

\\\\
\noindent\textbf{MSCcodes:}
05C70, 52B05, 52B12, 90C05, 90C27

\section{Introduction} 
\label{sec:intro}
Given a polyhedron $P=\{\vex:A\vex=\veb,B\vex\leq\ved\}$ with rational data, the \textit{circuits} or \textit{elementary vectors} of $P$ are a classic topic introduced in the works of several authors in the 1960s and 70s  \cite{c-64, g-75, r-69, t-65}. Formally, they are nonzero vectors $\veg$ in $\ker(A)$ such that $B\veg$ is support-minimal. They are a generalization of the edge directions of $P$, and serve as an inclusion-minimal test set for optimality certificates for any linear program (LP) over $P$. In the past decade, research related to circuits has focused on \textit{circuit augmentation schemes} for solving LPs \cite{bv-17,dhl-15,dks-19}, generalizations of the Simplex method for solving LPs, wherein one can move in the direction of any circuit (taking a maximal-length step) instead of just along edges, and the \textit{circuit diameter} of polyhedra \cite{bfh-14,bgl-23,bgl-24,kps-19}, which generalizes the combinatorial diameter in the same way and gives a lower bound on the number of iterations necessary to solve an LP via circuit augmentation algorithms, as well as a lower bound on the combinatorial diameter itself. Circuits have also been used indirectly as an analytical tool to bound the performance of interior point methods \cite{dkno-24}. Research in the area takes several forms and has often focused on problems in combinatorial optimization. In fact, many algorithms from combinatorial optimization turn out to be efficient implementations of circuit augmentation tailored to the application.

It is well known that for many structured problems, there exist circuits that are readily interpreted in terms of the underlying application. This comes from a combination of two properties: first, circuits have an inclusion-minimal support; for standard-form problems, this typically means that they only have few nonzero entries. And second, it greatly matters how complicated the nonzero entries of a circuit are. This can be formally described by the so-called \textit{circuit imbalance measure} \cite{dhnv-23}, which is the maximum absolute ratio of nonzero components of any circuit (Definition \ref{def:imba}). Whereas a small support is guaranteed by definition, a low circuit imbalance is not. In fact, it can be exponentially large in terms of the natural parameters of the underlying problem (for example, the number of vertices in the graph defining the problem). A high circuit imbalance is directly opposed to the ease of a formal analysis and study of diameters and circuit pivot rules. Without additional information, determining the circuit diameter of a polyhedron is known to be strongly NP-hard \cite{ns-25}. Likewise, determining the circuit distance between given vertices is NP-hard \cite{ach-21, bgkls-24, cs-23, ikkko-22}. In fact, there does not exist a polynomial-time constant-factor approximation for a shortest (monotone) circuit walk between given vertices, 
assuming P $\neq$ NP  \cite{cs-23}. Hardness and inapproximability for shortest monotone walks hold even for simple polytopes and in fixed dimension \cite{bns-25}. 

This leads to two promising directions of research: the study of problems where  circuit imbalance is known to be low; or a restriction of the analysis to those circuits of low imbalance to gain valuable insights on more complicated problems. Circuit augmentation and diameters are best understood for polyhedra with low circuit imbalance measure. A prime example lies in network-flows applications, which have totally-unimodular constraint matrices, in turn guaranteeing that circuits only have nonzero entries $\pm 1$. For example, the characterization of circuits in \cite{bm-23,gdl-14,env-21} are a key to showing that many graph-theoretic algorithms are, in fact, circuit augmentation schemes with respect to conceptually simple pivot rules. However, only a small number of concrete LPs from combinatorial optimization are known to have sub-exponential circuit imbalance measure.  These include the fractional matching polytope, the fractional stable set polytope, and problems with totally-unimodular constraint matrices, such as the network-flows problems mentioned above \cite{dks-19, kps-19, o-10}. 

In this work, we are interested in furthering the understanding of problems of high circuit imbalance. As we will see, the suggested approach of restricting the analysis to circuits of low imbalance is a powerful tool to obtain new insight.  

A typical strategy for connecting classical algorithms to circuit augmentation schemes, or for bounding the circuit diameter, builds on a characterization of the set of circuits of the polyhedron at hand. To this end, one attempts to devise a complete or partial description of the circuits in terms of the underlying combinatorial problem \cite{bfh-14, bm-23, bv-19a, 
env-21, kps-19}. For example, the circuit directions for the clustering problem described in \cite{bv-19a} can be thought of as swapping items between clusters. As one of our contributions, we add to this area by establishing a relation between the circuits of coloring polytopes and classic Kempe dynamics.

A key challenge lies in the fact that for many polyhedra in combinatorial optimization, the set of circuits is \textit{not} fully characterized, and it seems difficult to do so. This challenge is directly linked to the circuit imbalance and, relatedly, the fact that many (likely, most) of the circuits therefore admit no clear combinatorial interpretation. 
For example, a circuit $\veg$ defined over a set $E$ of size $n$ wherein one component of $\veg$ is $2^n$ times larger than another is unlikely to admit a natural interpretation in terms of the underlying set.  This is in contrast to, say, a circuit $\veg\in\{0,\pm 1\}^n$, which we call a \textit{0/1 circuit}. Such circuits are readily interpreted as adding and removing elements from some subset $M$ of $E$ (e.g., adding and removing edges from a given matching). As one of our contributions, in Section \ref{sec:imbal} we demonstrate that in a wide range of graph-theoretic problems, one already arrives at an exponential circuit imbalance just by using a small set (i.e., one whose size is polynomial in the size of the underlying combinatorial problem) of constraints with a simple structure: given a graph $G=(V,E)$, we consider various potential constraint matrices $A$ whose rows are in $\{0,1\}^E$.  We consider a variety of natural choices of families $\mathcal{H}\subseteq 2^E$ to determine the supports of the rows of $A$ (i.e., of the constraints).  In all of the cases we consider, we are able to show that any polyhedron $\{\vex\in\R^n:A\vex\leq b,\vex\geq\veo\}$ has circuit imbalance at least exponential in $|E|$. The fact that we do this for matrices $A$ with a number of rows polynomial in its number of columns (i.e. $|E|$) is significant.  It is known \cite{deprs-26} that a matrix with a number of rows exponential in its number of columns have circuit imbalance exponential in its number of columns. We find these observations to support the general hypothesis that it is rare for polyhedra arising from combinatorial optimization to have sub-exponential circuit-imbalance, even\textemdash as is the case, for example, with many integer program (IP) relaxations\textemdash when they have only a polynomial number of constraints.

In Sections \ref{sec:col} and \ref{sec:tree}, we then study two classic families of polytopes, corresponding to fractional vertex coloring and graphic matroids (or the search for a maximum-weight forest). While both of them are readily seen to have exponential circuit imbalance, 
we are able to gain valuable insight from a study of their 0/1 circuits, leading to a generalization of classic concepts of combinatorial reconfiguration and constant circuit diameter bounds, respectively. 

A restriction of the analysis to 0/1 circuits has already proven to be a successful approach in the literature for previous analytical results. 
Specifically, it was shown that the 0/1 circuits suffice to find extremely short paths (of just 1 or 2 steps) between the integral solutions of the well-known traveling salesman problem (TSP) \cite{kps-19}. Another use of 0/1 circuits appears in integer programming, where they serve as a key ingredient for proximity bounds for programs with nearly totally-unimodular matrices and bounded subdeterminants \cite{afjkswy-25}. 
Further, many combinatorial algorithms which are, as mentioned above, in essence circuit augmentation algorithms utilize only 0/1 circuits, even when the associated polyhedron has other circuits. 
And finally, a restriction to 0/1 circuits gives a natural connection to the area of combinatorial reconfiguration: given two integral extreme-point solutions of an LP, one can ask if it is possible to transform one solution into another via a set of valid ``moves", where each intermediary solution is required to also be integral. In a yes instance, one then asks how many steps it takes to transform any solution into any other solution. Here, our set of valid moves is precisely the circuits of the given LP.  When we consider moving between integral solutions for problems whose integral solutions are all in $\{0,1\}^n$, this necessitates that only 0/1 circuits are used, even when others exist.
Although combinatorial reconfiguration is widely studied, and there is circumstantial value in moving only between integral solutions (e.g., interpretability or feasibility with respect to an original application domain), related concepts like general circuit walks and the circuit diameter do not necessarily translate to results in the vein of combinatorial reconfiguration due to their ability to move to non-integral solutions.

These questions can be seen as a variation of the study of circuit diameters \cite{bdfm-18,bsy-18,kps-19}, where we seek to move between any two extreme-point solutions (i.e., not only integral ones), and in each iteration we are allowed to take a maximal-length step in the direction of any circuit.   
For many LPs from combinatorial optimization, all extreme-point solutions are in $\{0,1\}^n$, and all solutions in $\{0,1\}^n$ are extreme-point solutions.  In this setting, the number of moves in the transformations we consider gives an upper bound on the circuit diameter.

On the other hand, in view of circuit augmentation, a restriction of the analysis to 0/1 circuits or any subset of the set of circuits would have some natural limitations. First, generic methods for solving LPs via circuit augmentation may utilize circuits which are not 0/1, even where there is a monotone walk just using 0/1 circuits. In fact, many theoretical results bounding the running time of generic circuit augmentation algorithms on general LPs \textit{require} that the full set of circuits be available as augmentation directions \cite{dknv22, dhl-15,dks-19}. Relatedly, in combinatorial optimization we are typically interested in moving only between the integral solutions of an integer program. However, circuit augmentation — whether along general circuits or even along 0/1 circuits — may pass through non-integral solutions of the corresponding LP relaxation. There are several interesting future research directions coming out of these challenges. We point out some of them in our final remarks in Section \ref{sec:concl}.

Next, in Section \ref{sec:prelim}, we provide some necessary background and definitions on circuits and circuit imbalance measure. We then summarize our contributions and provide an outline of the paper in Section \ref{sec:contributions}.

\subsection{Preliminaries}\label{sec:prelim}
We now provide some formal terminology and then outline our contributions. Given a set $X$, a vector $\vex\in\R^X$, and $i\in X$, we use $\vex(i)$ to denote the component of $\vex$ indexed by $i$.

\begin{definition}
    The characteristic vector of a subset $T$ of a set $S$ is the vector $\ve{X}(T)\in\{0,1\}^S$ such that $\veX(s)=1$ if $s \in T$ and $\veX(s)=0$ if $s \notin T$ . 
\end{definition}

 We follow \cite{bfh-14,bv-17,dhl-15,env-21,r-69} for some background on the theory of circuits and circuit augmentation. Recall that a set of \textit{circuits} for a polyhedron can be defined as follows. 
 
 \begin{definition}[Circuits]\label{def:circuits}
Given matrices $A\in\Q^{m_A\times n}$ and $B\in\Q^{m_B\times n}$, the {\bf set of circuits} corresponding to $A$ and $B$, denoted $\mathcal{C}(A,B)$, consists of those $\veg \in ker(A)\setminus \{\veo\}$, normalized to coprime integer components, for which $B\veg$ is support minimal over the set of $\{B\vex:\vex\in ker(A) \setminus \{\veo\}\}$.  When $A$ is empty (in which case $\ker(A)=\R^n$), we denote the corresponding set of circuits by $\mathcal{C}(0,B)$.
\end{definition}

In order for a normalization to coprime integer components to be well-defined, we require that $P$ have rational data.  Throughout, we will assume rational data at all times.
Given a polyhedron $P$ described by the system of equalities and inequalities $P=\{\vex\in\R^n:A\vex = \veb, B\vex\leq \ved\}$, we say that the circuits $\mathcal{C}(A,B)$ are the circuits of this formulation of $P$. In the studies of circuit augmentation and diameters over a polyhedron $P$, it often is necessary to restrict to the feasible circuit directions at some point $\vex$ in $P$, i.e., those directions in which a nonzero step starting at $\vex$ remains in $P$.

\begin{definition}
    Given a point $\vex$ in a polyhedron $P=\{\vex:A\vex=\veb,B\vex\leq \ved\}$, we say a {\bf circuit $\veg\in\mathcal{C}(A,B)$ is feasible at $\vex$} if there exists $\epsilon>0$ such that $\vex+\epsilon\veg\in P$.
\end{definition}

This leads to the notion of a {\em circuit walk} where one moves between solutions in $P$ using feasible circuits.
\begin{definition}\label{defn:cw}
    Let $P = \{\vex\in \mathbb{R}^m: A\vex = \veb, B\vex \leq \ved\}$ be a polyhedron. Starting at an extreme-point $\vev$ of $P$, we call a sequence of $\vev = \vex_0,...,\vex_k$ a {\bf circuit walk} of length $k$ if for $i = 0,...,k-1$:
 \begin{enumerate}
     \item $\vex_i \in P$
     \item $\vex_{i+1} = \vex_i +\epsilon_i \veg_i$ for some $\veg_i \in \mathcal{C}(A,B)$ and $\epsilon_i > 0$ and
     \item $\vex_i+\epsilon \veg_i$ is infeasible for all $\epsilon > \epsilon_i$.
 \end{enumerate}
\end{definition}
\noindent This definition assumes that feasible circuits are used for steps of maximal length, and that the walk begins at an extreme-point $\vev = \vex_0$. We do not require that the walk terminate at an extreme-point $\vex_k$.

In general, different formulations of the exact same polyhedron $P$ can give rise to different sets of circuits.  However, as shown in \cite{k-22}, all so-called \textit{minimal} descriptions of a polyhedron give rise to the same set of circuits.  Here, a minimal description is defined as follows.

\begin{definition}[See e.g.~\cite{ccps-97}]\label{defn:minimal}
Given a polyhedron $P$ (considered as a set), we say that $\{\vex\in\R^n:A\vex = \veb, B\vex\leq \ved\}$ is a \textbf{minimal description of $P$} if it satisfies the following:

\begin{enumerate}
	\item $P = \{\vex\in\R^n:A\vex = \veb, B\vex\leq \ved\}$
	\item No inequality of $\{\vex\in\R^n:A\vex = \veb, B\vex\leq \ved\}$ can be made an equality without changing the set. 
	\item No inequality or equality of $\{\vex\in\R^n:A\vex = \veb, B\vex\leq \ved\}$ can be omitted without changing the set.
\end{enumerate}
\end{definition}

All valid formulations of $P$ must contain the circuits of a minimal description of $P$. Thus, when $\{\vex\in\R^n:A\vex = \veb, B\vex\leq \ved\}$ is a minimal description of a polyhedron $P$, we can meaningfully say that $\mathcal{C}(A,B)$ are \textit{the circuits of $P$}. 
We can define the circuits of $P$ equivalently as the set of all directions that can be obtained as edges of any polyhedron of the form $P'=\{\vex\in\R^n:A\vex = \veb', B\vex\leq \ved'\}$ as the right-hand sides $\veb'$ and $\ved'$ vary \cite{g-75}. As such, the circuits of a polyhedron include its edge directions. This latter fact holds even when one is given a non-minimal description.  That is, for \textit{any} description $P=\{\vex\in\R^n:A\vex = \veb, B\vex\leq \ved\}$ of a polyhedron $P$, (appropriate normalizations of) the edge directions of $P$ are in $\mathcal{C}(A,B)$. We now formally define the circuit imbalance measure, a measure on how complex the set of circuits can be. The definition below contrasts somewhat with definitions used elsewhere in the literature \cite{dhnv-23,dknv22,env-21}.  After stating our definition, we will discuss its relationship to the common definition and our reason for diverging from it.  In the below definition $|J|$ refers to the element-wise absolute value of the matrix.

\begin{definition}\label{def:imba}
    Given matrices $A\in \R^{m_A\times n}$ and $B\in \R^{m_B\times n}$ where $B$ contains as a row submatrix a matrix $J$ such that $|J|=I_{n\times n}$, the \textbf{circuit imbalance measure} of $A$ and $B$ is defined to be
    \[
    \kappa(A,B) = \max \left\{\left|\frac{\veg(i)}{\veg(j)}\right|:\veg\in\mathcal{C}(A,B)\right\}.
    \]
\end{definition}

The stipulation that $B$ contains the identity matrix $I_{n\times n}$ (up to row negations) is a technical point that allows our definition of the circuit imbalance to comport in part with the traditional one.  Here it corresponds to the inclusion of variable bounds (typically, non-negativity constraints) among the inequality constraints.  This is common among LPs coming from combinatorial optimization, which are the focus of study in this work. 

In prior treatments of the topic, the circuit imbalance measure is considered over LPs in standard equality form, in which case the circuits are support-minimal elements of the kernel of the equality constraint matrix $A$.  There, the circuit imbalance measure is defined in the same way---the maximum absolute ratio of any two entries of any circuit. There are two main ways that the general form LPs considered here could be transformed to the standard equality form setting (described below), at which point the standard definition could be applied.  In either case, Definition \ref{def:imba} gives a lower bound on the imbalance measure of the transformed, standard equality form LP.  This means that, when we show that a formulation admits exponential imbalance under our definition, the same holds for the standard definition when applied to either of the below constructions.

First, one could put it into standard equality form in the classic way (i.e., adding slack variables, splitting free variables, etc.). In addition to potentially growing the set of circuits \cite{bv-22,f-14}, this causes the entries of the circuits to include the slack variables introduced on the inequality constraints, and those entries would need to be taken into consideration when determining the circuit imbalance.  Viewed in this light, our definition essentially ignores the slacks (and the new circuits introduced by the construction).  This only makes our measure of the circuit imbalance smaller, and as such, our definition of the circuit imbalance offers a lower bound on the traditional definition applied to this construction.

Second, one could consider the construction described at the end of \cite{dknv22}.  When the standard definition of circuit imbalance is applied under that transformation, the circuit imbalance in terms of the original LP would instead be the maximum possible value of $\left|\frac{(B\veg)(i)}{(B\veg)(j)}\right|$ over all circuits $\veg$.  Note that, analogous to the first construction, $B\veg$ can be interpreted as the change in slacks when augmenting along the circuit $\veg$.  Since the identity matrix is a row submatrix of $B$, we have that the maximum value of $\left|\frac{\veg(i)}{\veg(j)}\right|$ is a lower bound on the above.  Then in this case too, our definition of the circuit imbalance offers a lower abound on the traditional definition applied to this construction.

We opt for this alternative definition because, in the context of combinatorial problems and combinatorial reconfiguration, it is usually most natural to represent the problem with a general form LP (especially, one with inequality constraints), but in our context we care predominantly about the relative entries of the circuits themselves, and not necessarily the relative changes in slack.  This is due to considerations already mentioned, like interpretability of the augmentation and the desire to move between integral solutions representing combinatorial objects.

 \subsection{Contributions}\label{sec:contributions}

We now have the basic terminology needed to describe our contributions. As discussed previously, these are of two types: we begin in Section \ref{sec:imbal} with a study of the circuit imbalance measure for generic graph-theoretic constraint matrices.  There we show that many simple, natural sets of constraints of size only polynomial in the size of an underlying graph produce exponential circuit imbalance. Then, in Sections \ref{sec:col} and \ref{sec:tree}, we study two specific, classic problems where we are able to devise new insight from their 0/1 circuits, despite those problems having an exponential imbalance.

To describe the contributions of Section \ref{sec:imbal} more formally, consider a function $\mh$ which maps any simple, loop-free graph $G=(V,E)$ to $\mh(G)\subseteq2^E$.  The function $\mh$ describes how to construct a constraint matrix associated with $G$ in the following way.
Let $B'(G,\mh)$ be the matrix whose rows are given by $\ve{X}(H)$ for all $H\in\mh(G)$, where $\ve{X}(H)$ is the characteristic vector of $H$.  Let 
\[
B(G,\mh)=
\begin{bmatrix}
B'(G,\mh)\\
I
\end{bmatrix},
\]
and let $\kappa(\mh)=\max_G\{\kappa(0,B(G,\mh)\}$.
Note that this a very general framing that captures many well-studied constraint matrices from combinatorial optimization (for appropriate choices of $\mh$).
It is known, for example, that if
\[
\mh(G) = \{\delta(v):v\in V\},
\]
then $B$ is the inequality constraint matrix of the standard formulation of the fractional matching polytope, and the circuit imbalance measure of that matrix is at most 2 (see, e.g., \cite{dks-19}).

In Section \ref{sec:imbal}, we consider 
the following natural question: 

\begin{question}\label{q:sub-exponential}
Do natural choices for $\mh$ which output simple, polynomially sized sets typically lead to $\kappa(\mh)$ which grows exponentially or sub-exponentially in the size of the input graph?
\end{question}

Questions along this line have been studied before. For example, in Section 4 of \cite{h-21} it is shown that the multicommodity flow LP has exponential circuit imbalance per the definition of circuit imbalance used here. 
A sub-exponential circuit imbalance comes with several advantages. In particular, known bounds on circuit diameters become stronger \cite{dknv22} and, more generally, a partial (or complete) characterization of the set of circuits becomes more attainable. However, we will show that many families of constraints (both structurally-simple and small in size) already lead to exponential circuit imbalance. 

Polyhedra associated to combinatorial optimization problems often have a number of rows which is exponential in their number of columns (i.e., its number of constraints is exponential in the number of edges of $G$).  It was shown recently in \cite{deprs-26} that any such matrix admits a circuit imbalance which is exponential in the number of columns, albeit for the different but related definition of circuit imbalance discussed earlier.  We show that in fact, many natural choices of $\mh$ which produce constraint matrices with only a polynomial number of rows still produce exponential circuit imbalance in the worst case.  This suggests that combinatorial problems are somehow prone to producing exponential circuit imbalance for reasons that are stronger than the general lower bound given by \cite{deprs-26}.  We propose that functions $\mh$ that give low circuit imbalance for \textit{all} graphs may be rare, even when $\mh(G)$ is guaranteed to be polynomial in $|V|$.

Consider for example the set 
\[
\mh(G) = \{\delta(v):v\in V\}\cup\{E[S]:S\subseteq V, |S|\text{ odd}\},
\]
where, for $S\subseteq V$ we let $E[S]=\{uv\in E:u,v\in S\}$. In this case, $B(G,\mh)$ is the inequality constraint matrix of the standard formulation of the matching polytope, and indeed $B(G,\mh)$ has exponentially many rows.

However, in Theorem \ref{thm:induce_imbalance} we show that the inclusion of just a polynomially sized subset of $\{E[S]:S\subseteq V,|S|\text{ odd}\}$ is already enough to guarantee at least exponential circuit imbalance. In the same theorem, we show the same holds if $\mh(G)$ contains all paths of length 4 or if $\mh(G)$ contains all sets of size 3.
We strengthen these last two conclusions with a new construction in Theorem \ref{thm:path_imbalance} and show that even if $\mh(G)$ only contains all paths of length 3, then $\kappa(\mh)$ is at least exponential.  Finally, in Theorems \ref{thm:cycle_imbalance} and \ref{thm:star_imbalance}, we get at least exponential circuit imbalance when $\mh(G)$ contains all cycles of length 4 or all star subgraphs of size at most 4, respectively.  

Our motivation to study the 0/1 circuits of polyhedra from combinatorial optimization is bolstered by these observations.  As noted earlier, when such a polyhedron admits exponential imbalance, it is difficult to ascribe any combinatorial interpretation to the full set of circuits, and indeed it is likely difficult to be able to even characterize the full set of circuits. In cases where the circuit imbalance is exponential, we may seek to characterize the 0/1 circuits, determine whether the 0/1 circuits suffice to move between integral solutions (those that correspond to combinatorial objects), determine how many augmentations along 0/1 circuits it takes to move from any 0/1 solution to another, etc.  In short, for these combinatorial problems where the full set of circuits is difficult to use, we may seek to do with the 0/1 circuits what one might typically do with the full set.
 
As  case studies, we do this for the classic families of polytopes studied in Sections \ref{sec:col} and \ref{sec:tree} where exponential circuit imbalance is observed (Lemma \ref{lem:unbounded_imbalance} and Corollary \ref{cor:rank_imbalance}, respectively).  There, we focus on studying the 0/1 circuits of the fractional coloring polytope and forest polytope, respectively.
Among other things, we prove that a restriction of circuit walks to the 0/1 circuits suffices to not only guarantee reachability of the integral extreme-points, but leads to a linear bound to walk between integral extreme-points for the fractional coloring polytope and a constant circuit diameter bound for the forest polytope, respectively. 

In Section \ref{sec:col}, we study the relationship between the 0/1 circuits of the fractional coloring polytope and classic Kempe dynamics. We consider the following well-known representation of the fractional coloring problem.
\begin{definition}
Let $G = (V, E)$ be a graph and $\tau = \{ t_1, \ldots, t_r \}$ a set of colors. The {\bf fractional coloring problem} has the following formulation:
\begin{align*}
    \vex \in \R ^{|V| \times r}: \sum_{i \in K} \vex({v, i}) = &  \;\; 1  \;\;\ \forall v \in V \\
    \vex({v,i}) + \vex({u, i}) \leq & \;\; 1  \;\;\ \forall i \in \tau, \, uv \in E \\ 
    \vex({v,i}) \geq &  \;\; 0  \;\;\ \forall v \in V, \forall i\in\tau
\end{align*}
\end{definition}

Using this formulation, in Theorem \ref{thm:colcir} we are able to characterize the circumstances under which (the extreme-points corresponding to) two proper colorings differ by a circuit. A consequence of Theorem \ref{thm:colcir} is that the so-called
Kempe chains of the underlying reconfiguration graph correspond to some but not all of the 0/1 circuits of the fractional coloring polytope.  This correspondence allows us to reinterpret Kempe feasibility in terms of circuit walks. Here, and in the literature, Kempe feasibility or Kempe equivalence refers to the ability to transform one proper coloring into another by a sequence of Kempe swaps, i.e., an exchange of two colors on a minimal number of vertices subject to the coloring remaining proper.

In the polyhedral setting, this is analogous to reachability between extreme-points of the integer coloring polytope via edge walks, though the correspondence is not perfect: it is well known that Kempe swaps alone do not suffice to move between any pair of colorings (i.e., reachability can fail); for instance, in the case of the 3-colorings of the triangular prism. 

In fact, in Theorem \ref{thm:colcir}, we \textit{fully} characterize the 0/1 circuits of the fractional coloring polytope.  In particular, we show that the set of 0/1 circuits of the fractional coloring polytope includes vectors corresponding to a natural \textit{generalization} of Kempe swaps wherein they can involve more than just two colors.
We then show in Theorem \ref{thm:coloring-walks} that, although Kempe swaps do not suffice to achieve reachability between all proper colorings, the generalized Kempe swaps corresponding to these 0/1 circuits \textit{do} suffice to achieve reachability.   That is to say, all integral extreme-points of the fractional coloring polytope (those that correspond to proper colorings) are reachable from each other via circuit walks using only the 0/1 circuits. 

In this way, the 0/1 circuits of the LP relaxation of the coloring IP (Definition \ref{def:col}) serve a useful function which is not effectively served by edge walks in either the coloring polytope (the convex hull of the coloring IP) or the fractional coloring polytope.  Although edge-walks in the coloring polytope suffice to traverse the proper colorings of a graph, any description of the coloring polytope has a number of constraints which is at least exponential in the size of the graph \cite{cmmz-02}, and to our knowledge there is no known complete facet description of this polytope.  On the other hand, the fractional coloring polytope has a much smaller description---one which is polynomial in the size of the graph whenever the number of available colors is polynomial in the size of the graph---but edge walks in this polytope may move to fractional vertices which do not correspond to proper colorings. 

In this respect, the 0/1 circuits of the LP relaxation prove to be a very valuable class of objects in this case.  They both allow for traversal of the 0/1 solutions (those that correspond to the combinatorial objects we care about, i.e., proper colorings), but can be obtained from the small IP description rather than the larger description of its convex hull. 

On the other hand, in contrast to some combinatorial polytopes (like TSP and Matching discussed earlier), we show that the 0/1 circuits of the fractional coloring polytope are not enough to guarantee the existence of \textit{constant}-length circuit walks between extreme-points corresponding to proper colorings (if those walks are required to visit only integral extreme-points, i.e., proper colorings).  In particular, we show in Theorem \ref{thm:long_proper_walk} that for a graph on $n$ vertices, $n$ steps can be necessary (and the fact that $n$ steps suffices is trivial).
Finally, in Section \ref{sec:kempe_imbal}, we provide an explicit construction showing that the circuit imbalance of the coloring polytope is at least exponential, further supporting the restriction of our attention to 0/1 circuits.  

In Section \ref{sec:tree} we study the polytope formed from a subset of the constraints that appear in problems such as the maximum weight forest problem, a polytope associated with a graphic matroid. A major difference between this polytope and the coloring polytope discussed in Section \ref{sec:col} is that it has an exponential number of constraints, which leads to some interesting differences in our approach and results.

\begin{definition}
    Let $G = (V,E)$ be a simple graph. The {\bf maximum weight forest problem} has a feasible set that corresponds to the extreme-points of the polytope given by the following system of constraints:
    \begin{equation*}
\begin{array}{rcrcllr}
   \vex \in \R^{|E|}:  &&\sum_{e \in E(G[U])}\vex(e) &\leq& |U|-1 &\quad \forall \emptyset \neq U \subseteq V&\\
    && \vex(e) &\geq& 0 &\quad \forall e \in E,& \tag{MWF}
\end{array}
\end{equation*}
where $G[U]$ is the subgraph of $G$ induced by $U \subseteq V$.
\end{definition}

\noindent In the constraint matrix of (\ref{MWF}), there are duplicated rows in that the matrix rows for the constraints $\vex(e) \geq 0$ appear also in the rank inequality constraints ($\sum_{e \in E(G[U])}\vex(e) \leq |U|-1$) for sets $U=\{u,v\}$ with $uv \in E$. For the characterization of circuits or non-circuits, it suffices to consider the simpler system
 \begin{equation*}
\begin{array}{rcrclcr}
    &&\sum_{e \in E(G[U])}\vex(e) &\leq& |U|-1 &\quad \forall \emptyset \neq U \subseteq V.& \tag{Rank}
\end{array}
\end{equation*}

\noindent First, we note that the circuit imbalance of (\ref{Rank}) is exponential; this follows from the general lower bound of \cite{deprs-26}, but moreover it follows from our Theorem \ref{thm:induce_imbalance} that restricting the above constraints to the polynomially sized subset given by $U\subseteq V$ with $|U|=5$ is already enough to guarantee exponential imbalance.   In combination with the exponential number of constraints,  a full and interpretable characterization of the whole set of  circuits of (\ref{Rank}) is unlikely, and we dedicate the section to studying the 0/1 circuits of (\ref{Rank}) and walks along them within (\ref{MWF}).  

In Section \ref{sec:simplerank}, we begin by devising a number of properties that 0/1 circuits or non-circuits must satisfy. First, we show in Lemma \ref{lem:unitvector} that standard unit vectors are circuits of (\ref{Rank}) and that no other \emph{uniform-sign} 0/1 vectors (i.e., 0/1 vectors
where all nonzero entries have the same sign) are circuits. This allows us to focus our discussion on \emph{mixed-sign vectors} (and mixed-sign edge sets) which have at least one positive and one negative entry. As a key tool to this end, we introduce the idea of \emph{balanced vertex sets}, where the values of induced edges ``cancel out,'' and \emph{imbalanced vertex sets}, where they do not.

 Recall that, to check whether a given 0/1 vector is a circuit, we have to check whether there exists a vector of smaller, i.e., strictly included support with respect to (\ref{Rank}). 
 We show that if a single edge could be dropped to arrive at such a vector, then that edge is not in a balanced set, and vice versa (Lemma \ref{lem:singleedgedrop}). 
 This implies that all edges in the support of a circuit lie in \emph{some} balanced set. As the number of vertex subsets is exponential in the size of $V$, the property of an edge \emph{not} being in a balanced set is very restrictive. In Lemma \ref{lem:discon_circs}, this allows us to prove that many mixed-sign 0/1 vectors whose supports are disconnected edge sets are circuits.
 We conclude this first subsection by proving that ``alternating" (and connected) graph structures (reminiscent of alternating paths and cycles in the setting of matchings) correspond to 0/1 circuits of (\ref{Rank}), as do natural generalizations of these structures; see Lemmas \ref{lem:alts}, \ref{lem:root_cyc}, and \ref{lem:pseudo}). These alternating structures are of interest due to their connection to classical combinatorial algorithms; see for example \cite{bm-23}. 

In Section \ref{sec:mixedsign}, we prove a set of more general properties that characterize mixed-sign 0/1 vectors as \emph{non}-circuits. Notably, we show that for 0/1 non-circuits, either it is possible to drop a single edge to obtain a smaller support (Lemma \ref{lem:singleedge} and Corollary \ref{cor:unitvectorcontainment}), or the support of the non-circuit is either a single mixed-sign component of low diameter (at most $10$), or it contains a unique mixed-sign component of diameter at most $5$ (Theorems \ref{thm:diameter} and \ref{thm:notacircuit}) which itself corresponds to a non-circuit. This implies that, outside of special cases, 0/1 vectors with a mixed-sign component of diameter greater than $10$ must be circuits. 

While we do not have, and did not expect, a full characterization of 0/1 circuits for (\ref{Rank}), the insight gained allows us to prove surprisingly strong circuit diameter bounds for (\ref{MWF}) in Section \ref{sec:rankddiameters}. For general graphs, we show a constant upper bound of $9$ (Theorem \ref{thm:upperbound9}), which improves to $7$ for complete graphs (Theorem \ref{thm:upperbound7}). We conclude by proving a lower bound of $3$ (and for most graphs, $4$) in Theorem \ref{thm:low_diam_bnd}. These results, where diameters range between $3$ and $9$ (or in special cases, between $4$ and $7$), contrast with diameter upper bounds of $1$ or $2$ in the literature for problems such as TSP and matching \cite{kps-19}.

Finally, in Section \ref{sec:concl} we conclude with some remarks and discuss some natural directions of future work related to the use of 0/1 circuits in reachability and diameter studies.

\section{Exponential Circuit Imbalance from Polynomially Sized Graph Problems}
\label{sec:imbal}

In this section we explore a variety of seemingly-natural sets of constraints defined over the edges of a graph whose inclusion in a linear system guarantees it to have at least exponential circuit imbalance.  The sets of constraints we consider all have sizes which are polynomial in the size of the underlying graph (which here corresponds to the number of variables), and as such the fact that they produce exponential circuit imbalance is not already implied by the general lower bound on the circuit diameter given by \cite{deprs-26}.

Our first theorem (Theorem \ref{thm:induce_imbalance}) shows that even a small collection (i.e., polynomially sized) of relatively ``simple" constraints is enough to guarantee exponential circuit imbalance.  Many LPs from combinatorial optimization have a number of constraints which are exponential in the number of variables. Since, in general, the number of circuits grows exponentially in the number of constraints, it is not surprising to find that such LPs often also have exponential circuit imbalance measure.  This holds, for example, for the Matching LP. A particularly surprising consequence of Theorem \ref{thm:induce_imbalance} is that only a polynomially sized subset of the constraints of the Matching LP (i.e., those described by point 1. in the theorem) are necessary to already give exponential circuit imbalance.  We go on to show analogous results for other structurally simple, polynomially sized sets of constraints.  In our view, this underscores a general difficulty of characterizing \textit{all} circuits of LPs from combinatorial optimization, and motivates a focus on a restricted attention to just the 0/1 circuits (or, in general, a subset of circuits which is interpretable, characterizable, and sufficient to be useful).

Before we state and prove Theorem \ref{thm:induce_imbalance}, we recall a useful notation and fact of which we will make use in the remainder of this section.
We will use Landau notation to describe the circuit imbalance. Recall the following definition. 

\begin{definition}
    Let $f,g : \Z^+ \rightarrow \R^+$.We say that $f(n)$ is $\Omega(g(n))$ (or $f(n) \in \Omega(g(n)))$ if there exists a real constant $c > 0$ and an integer constant $n_0 \geq 1$ such that $f(n) \geq c \cdot g(n)$ for every integer $n \geq n_0$.
\end{definition}
The following simple observation on a linear system with four variables will be used repeatedly in our proofs.

\begin{equation}\label{eq:all_half}
    \text{If }b+c=a, b+d=a,\text{ and }c+d=a, \text{ then }b=c=d=\frac{1}{2}a.
\end{equation}

We now state Theorem \ref{thm:induce_imbalance} Recall that $\mh$ denotes the supports of the rows of an inequality constraint matrix, as defined in Section \ref{sec:contributions}.
\begin{theorem}\label{thm:induce_imbalance}
Let $\mh$ be a function which maps any simple, loop-free graph $G=(V,E)$ to $\mh(G)\subseteq 2^E$. We have that $\kappa(\mh) \in \Omega(2^{|E|})$ if 
\begin{enumerate}
    \item for all $G$, $\mh(G) \supseteq \{E[U]:U\subseteq V, |U|=5\}$, 
    \item for all $G$, $\mh(G)\supseteq\{H:H \ \text{is the edge set of a path}, |H|=4\}$, or
    \item for all $G$, $\mh(G) \supseteq\{H:|H|=3\}$, 
\end{enumerate}
\end{theorem}
\begin{proof}
    We first prove conclusion 1 explicitly.  As we will see later, points 2 and 3 follow from the fact that such collections of edge sets give rise to the same relationships between the same variables.
    
    Consider the graph $G_k$ with vertex set $V=\{w_0\}\cup\{t_i,u_i,v_i,w_i:1\leq i\leq k\}$ and edge set $E=\{t_iu_i, t_iv_i, v_iw_i, u_iw_i:1\leq i\leq k\}\cup \{w_it_{i+1}:0\leq i\leq k-1\}$. 

    Let $\veg$ be such that $\veg(w_it_{i+1})=(-2)^{-i}$ for all $0\leq i\leq k-1$, $\veg(t_iu_i)=\veg(t_iv_i)=0$ for all $1\leq i\leq k$, and $\veg(u_iw_i)=\veg(v_iw_i)=(-2)^{-i}$. See Figure \ref{fig:induce_imbalance} for such a construction of $\veg$. 

    There exists a circuit $\veh$ with $\supp(B\veh)\subseteq\supp(B\veg)$.  We will show that in any such circuit $\veh$, $|\veh(w_it_{i+1})| = 2|\veh(w_{i+1}t_{i+2})|$ for all $0\leq i\leq k-2$.  We have that for all $1\leq i\leq k-1$,
    \begin{align}
    \veg(E[\{w_{i-1},t_i,u_i,v_i,w_i\}])=0,\label{eq:sum1}&\\
    \veg(E[\{w_{i-1},t_i,u_i,w_i,t_{i+1}\}])=0,\label{eq:sum2}&\text{ and}\\
    \veg(E[\{w_{i-1},t_i,v_i,w_i,t_{i+1}\}])=0.\label{eq:sum3}&
    \end{align}
    Since $\supp(B\veh)\subseteq\supp(B\veg)$, we also have that for all $1\leq i\leq k-1$,
     \begin{align*}
    \veh(E[\{w_{i-1},t_i,u_i,v_i,w_i\}])=0,&\\
    \veh(E[\{w_{i-1},t_i,u_i,w_i,t_{i+1}\}])=0,&\text{ and}\\
    \veh(E[\{w_{i-1},t_i,v_i,w_i,t_{i+1}\}])=0.&
    \end{align*}

    If for all $1\leq i\leq k-1$  we let $a=-\veh(w_{i-1}t_i)$, the above implies that
    
     \begin{alignat*}{6}
    \veh(u_iw_i) &+& \veh(v_iw_i) & &              &= a,& \\
                 & & \veh(v_iw_i) &+& \veh(w_it_{i+1}) &= a,& \text{ and}\\
    \veh(u_iw_i) & &              &+& \veh(w_it_{i+1}) &= a.
    \end{alignat*}

    By (\ref{eq:all_half}), $\veh(u_iw_i)=\veh(v_iw_i)=\veh(w_it_{i+1})=\frac{1}{2}a$.  It follows that $|\veh(w_{i-1}t_i)|=2|\veh(w_it_{i+1})|$ for all $1\leq i\leq k-1$.  Thus, $\kappa(\mh)\in \Omega(2^{|E|})$, as desired.

    Finally, we observe that in the above construction, the exact same relations between $\veh(u_iw_i), \veh(v_iw_i)$, and $\veh(w_it_{i+1})$ can be achieved if $\mh \supseteq \{H: H\text{ is the edge set}$ $\text{of a path}, |H|=4\}$ or if $\mh \supseteq\{H:|H|=3\}$, giving the desired results. In particular, consider the edge sets $E[\{w_{i-1},t_i,u_i,v_i,w_i\}]$, $E[\{w_{i-1},t_i,u_i,w_i,t_{i+1}\}]$, and $E[\{w_{i-1},t_i, v_i, w_i, t_{i+1}\}]$ utilized in equations (\ref{eq:sum1}), (\ref{eq:sum2}), and (\ref{eq:sum3}), respectively.  Denote by $S$ the set 
    \[E[\{w_{i-1},t_i,u_i,v_i,w_i\}]\cap \supp(\veg) = \{w_{i-1}t_i,u_iw_i,w_iv_i\}.\] 
    
    If we let $P$ denote the path $\{w_{i-1}t_i,t_iu_i,u_iw_i,w_iv_i\}$, then we have that $P\cap \supp(\veg)=S$ and $|P|=4$.  Thus, the equation (\ref{eq:sum1}) holds in the case that $\mh\supseteq\{H:H \ \text{is the edge set of a path}, |H|=4\}$.  Analogous paths exist for equations (\ref{eq:sum2}) and (\ref{eq:sum3}).  Thus, in this case we infer the same relationships between the same variables, as desired.

    Likewise, the set $S$ defined above is of size 3, and of course $S\cap\supp(\veg)=S$.  Thus, the equation (\ref{eq:sum1}) holds in the case that $\mh \supseteq\{H:|H|=3\}$.  Analogous inferences hold for (\ref{eq:sum2}) and (\ref{eq:sum3}).  Thus, in this case we again infer the same relationships between the same variables, as desired.
\end{proof}
\begin{figure}[htb]
        \centering
        \begin{tikzpicture}[scale = 0.9]
            \tikzset{vertex/.style = {shape=circle,draw,minimum size=2em}}
            \node[vertex] (w0) at (0,2) {$w_0$};
            
            \node[vertex] (t1) at (2,0) {$t_1$};
            \node[vertex] (u1) at (2,2) {$u_1$};
            \node[vertex] (v1) at (4,0) {$v_1$};
            \node[vertex] (w1) at (4,2) {$w_1$};

            \node[vertex] (t2) at (6,0) {$t_2$};
            \node[vertex] (u2) at (6,2) {$u_2$};
            \node[vertex] (v2) at (8,0) {$v_2$};
            \node[vertex] (w2) at (8,2) {$w_2$};

            \node[vertex] (t3) at (10,0) {$t_3$};

             \node at (11,1) {\ldots};

            \draw[thick] (w0) to node[midway, fill=white] {$1$} (t1) {};

            \draw[thick,dashed] (t1) to (u1) {};
            \draw[thick,dashed] (t1) to (v1) {};
            \draw[thick] (v1) to node[midway, fill=white] {$-\frac{1}{2}$} (w1) {};
            \draw[thick] (u1) to node[midway, fill=white] {$-\frac{1}{2}$} (w1) {};
            \draw[thick] (w1) to node[midway, fill=white] {$-\frac{1}{2}$} (t2) {};

            \draw[thick,dashed] (t2) to (u2) {};
            \draw[thick,dashed] (t2) to (v2) {};
            \draw[thick] (v2) to node[midway, fill=white] {$\frac{1}{4}$} (w2) {};
            \draw[thick] (u2) to node[midway, fill=white] {$\frac{1}{4}$} (w2) {};
            \draw[thick] (w2) to node[midway, fill=white] {$\frac{1}{4}$} (t3) {};

            \end{tikzpicture}
            \caption{The construction of the vector $\veg$ in the proof of Theorem \ref{thm:induce_imbalance}. Each edge is labeled with its corresponding value in $\veg$, where dashed lines indicate a value of 0.}
            \label{fig:induce_imbalance}
    \end{figure}
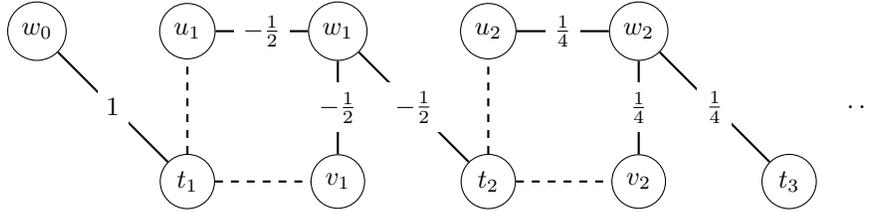

In fact, we can show that we get at least exponential circuit imbalance if we consider only paths of size 3.

\begin{theorem}\label{thm:path_imbalance}
    Let $\mh$ be a function which maps any simple, loop-free graph $G=(V,E)$ to $\mh(G)\subseteq2^E$.  We have that $\kappa(\mh)\in\Omega(2^{|E|})$ whenever, for all $G$, $\mh(G)\supseteq\{H:\text{$H$ is the edge set of a path}, |H|=3\}$.
\end{theorem}
\begin{proof}
    Consider the graph $G_k$ with vertex set $V=\{u_0\}\cup\{t_i,u_i,v_i,\\ w_i:1\leq i\leq k\}$ and edge set $E=\{u_{i-1}t_i,t_iu_i,u_iv_i,v_iw_i\}$.
    Let $\veg$ be such that $\veg(u_it_{i+1})=(-2)^{-i}$ for all $0\leq i\leq k-1$, $\veg(t_iu_i)=\veg(u_iv_i)=(-2)^{-i}$ for all $1\leq i\leq k$, and $\veg(v_iw_i)=(-2)^{-(i-1)}$ for all $1\leq i\leq k$. See Figure \ref{fig:path_imbalance} for such a construction of $\veg$.

    There exists a circuit $\veh$ with $\supp(B\veh)\subseteq \supp(B\veg)$.  We will show that for any such circuit, $|\veh(u_{i-1}t_i)| = 2|\veh(u_it_{i+1})|$ for all $1\leq i\leq k-1$.  We have that for all $1\leq i \leq k-1$

     \begin{align*}
    \veg(\{u_{i-1}t_i,t_iu_i,u_iv_i\})=0,&\\
    \veg(\{t_iu_i,u_iv_i,v_iw_i\})=0,&\\
    \veg(\{w_iv_i,v_iu_i,u_it_{i+1}\})=0,&\text{ and}\\
    \veg(\{u_{i-1}t_i,t_iu_i,u_it_{i+1}\}\})=0.&
    \end{align*}

    Since $\supp(B\veh)\subseteq\supp(B\veg)$, we also have that for all $1\leq i\leq k-1$

    \begin{align}
    \veh(\{u_{i-1}t_i,t_iu_i,u_iv_i\})=0,&\label{e1}\\
    \veh(\{t_iu_i,u_iv_i,v_iw_i\})=0,&\label{e2}\\
    \veh(\{w_iv_i,v_iu_i,u_it_{i+1}\})=0,&\text{ and}\label{e3}\\
    \veh(\{u_{i-1}t_i,t_iu_i,u_it_{i+1}\})=0.\label{e4}&
    \end{align}

    It follows from equations (\ref{e1}) and (\ref{e2}) that for all $1\leq i\leq k$
    \[
    \veh(u_{i-1}t_i) = - (\veh(t_iu_i) + \veh(u_iv_i)) = \veh(v_iw_i)
    \]
    Then, if for all $1\leq i\leq k-1$ we let $a=-\veh(u_{i-1}t_i)=-\veh(v_iw_i)$, it follows from equations (\ref{e2}), (\ref{e3}), and (\ref{e4}) that
    \begin{alignat*}{6}
    \veh(t_iu_i) &+& \veh(u_iv_i) & &              &= a,& \\
                 & & \veh(u_iv_i) &+& \veh(u_it_{i+1}) &= a,& \text{ and}\\
    \veh(t_iu_i) & &              &+& \veh(u_it_{i+1}) &= a.
    \end{alignat*}
    By (\ref{eq:all_half}), $\veh(t_iu_i)=\veh(u_iv_i)=\veh(u_it_{i+1})=\frac{1}{2}a$.  It follows that for all $1\leq i\leq k-1$, we have $|\veh(u_{i-1}t_i)|=2|\veh(u_it_{i+1})|$.  Thus, $\kappa(\mh)\geq\Omega(2^{|E|})$, as desired.
\end{proof}
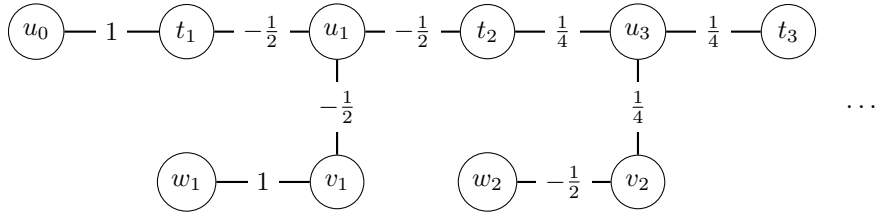
\begin{figure}[htb]
        \centering
        \begin{tikzpicture}[scale = 0.9]
            \tikzset{vertex/.style = {shape=circle,draw,minimum size=2em}}
            \node[vertex] (u0) at (0,2) {$u_0$};
            \node[vertex] (v1) at (2,0) {$w_1$};
            \node[vertex] (u1) at (2,2) {$t_1$};
            \node[vertex] (v2) at (4,0) {$v_1$};
            \node[vertex] (u2) at (4,2) {$u_1$};
            \node[vertex] (v3) at (6,0) {$w_2$};
            \node[vertex] (u3) at (6,2) {$t_2$};
            \node[vertex] (v4) at (8,0) {$v_2$};
            \node[vertex] (u4) at (8,2) {$u_2$};
            \node[vertex] (u5) at (10,2) {$t_3$};

             \node at (11,1) {\ldots};

             \draw[thick] (u0) to node[midway, fill=white] {$1$} (u1) {};
             
             \draw[thick] (u2) to node[midway, fill=white] {$-\frac{1}{2}$} (v2) {};
             \draw[thick] (v1) to node[midway, fill=white] {$1$} (v2) {};
             \draw[thick] (u1) to node[midway, fill=white] {$-\frac{1}{2}$} (u2) {};
             \draw[thick] (u2) to node[midway, fill=white] {$-\frac{1}{2}$} (u3) {};

             \draw[thick] (u4) to node[midway, fill=white] {$\frac{1}{4}$} (v4) {};
             \draw[thick] (v3) to node[midway, fill=white] {$-\frac{1}{2}$} (v4) {};
             \draw[thick] (u3) to node[midway, fill=white] {$\frac{1}{4}$} (u4) {};

             \draw[thick] (u4) to node[midway, fill=white] {$\frac{1}{4}$} (u5) {};

            \end{tikzpicture}
            \caption{The construction of the vector $\veg$ in the proof of Theorem \ref{thm:path_imbalance}. Each edge is labeled with its corresponding value in $\veg$.}
            \label{fig:path_imbalance}
    \end{figure}
Given that the circuit imbalance measure is unbounded for edge sets of short paths, a natural question is whether this holds true for edge sets of short cycles. We see this is true in the next result. Though the construction (depicted in Figure \ref{fig:cycle_imbalance}) is different, the proof is of a very similar nature to those of Theorems \ref{thm:induce_imbalance} and \ref{thm:path_imbalance}, and so it is postponed to Appendix \ref{app:imbal}.
\begin{restatable}{theorem}{cycleimbal}\label{thm:cycle_imbalance}Let $\mh$ be a function which maps any simple, loop-free graph $G=(V,E)$ to $\mh(G)\subseteq2^E$.  We have that $\kappa(\mh) \in \Omega(2^{|E|})$ whenever, for all $G$,
$\mh(G)\supseteq\{H:\text{$H$ is the edge set of a cycle}, |H|\leq 4\}$.   
\end{restatable}

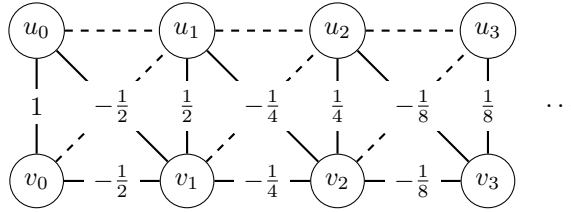
\begin{figure}[htb]
        \centering
        \begin{tikzpicture}
            \tikzset{vertex/.style = {shape=circle,draw,minimum size=2em}}
            \node[vertex] (u0) at (0,0) {$u_0$};
            \node[vertex] (v0) at (0,2) {$v_0$};
            \node[vertex] (u1) at (2,0) {$u_1$};
            \node[vertex] (v1) at (2,2) {$v_1$};
            \node[vertex] (u2) at (4,0) {$u_2$};
            \node[vertex] (v2) at (4,2) {$v_2$};
            \node[vertex] (u3) at (6,0) {$u_3$};
            \node[vertex] (v3) at (6,2) {$v_3$};

             \node at (7,1) {\ldots};

             \draw[thick,dashed] (v0) to (v1) {};
             \draw[thick,dashed] (u0) to (v1) {};

             \draw[thick] (v0) to node[midway, fill=white] {$1$} (u0) {};
             \draw[thick] (v0) to node[midway, fill=white] {$-\frac{1}{2}$} (u1) {};
             \draw[thick] (u0) to node[midway, fill=white] {$-\frac{1}{2}$} (u1) {};

             \draw[thick,dashed] (v1) to (v2) {};
             \draw[thick,dashed] (u1) to (v2) {};

             \draw[thick] (v1) to node[midway, fill=white] {$-\frac{1}{2}$} (u1) {};
             \draw[thick] (v1) to node[midway, fill=white] {$\frac{1}{4}$} (u2) {};
             \draw[thick] (u1) to node[midway, fill=white] {$\frac{1}{4}$} (u2) {};

             \draw[thick,dashed] (v2) to (v3) {};
             \draw[thick,dashed] (u2) to (v3) {};

             \draw[thick] (v2) to node[midway, fill=white] {$\frac{1}{4}$} (u2) {};
             \draw[thick] (v2) to node[midway, fill=white] {$-\frac{1}{8}$} (u3) {};
             \draw[thick] (u2) to node[midway, fill=white] {$-\frac{1}{8}$} (u3) {};

             \draw[thick] (v3) to node[midway, fill=white] {$-\frac{1}{8}$} (u3) {};
            
            \end{tikzpicture}
            \caption{The construction of the vector $\veg$ in the proof of Theorem \ref{thm:cycle_imbalance}. Each edge is labeled with its corresponding value in $\veg$, where dashed lines indicate a value of 0.}
            \label{fig:cycle_imbalance}
    \end{figure}

Finally, we note that $\mh(G) = \{\delta(v):v\in V\}$---the collection where $B(G,\mh(G))$ is the inequality constraint matrix of the fractional matching polytope---can be equivalently expressed as 
\[\mh(G) = \{H:H\text{ is the edge set of a maximal star subgraph}\},\]
where $H$ is the edge set of a \textit{maximal} star subgraph if there does not exist $e\in E$ such that $H\cup\{e\}$ is also the edge set of a star subgraph of $G$. As noted in Section \ref{sec:contributions}, this collection $\mh$ gives constant circuit imbalance \cite{dks-19}.  Since the collection of all \textit{maximal} star subgraphs gives constant circuit imbalance, it is natural to wonder whether the collection of all star subgraphs gives sub-exponential circuit imbalance.  We now show that this is not the case already for the collection of all star subgraphs of size at most 4.  As before, the construction is depicted in Figure \ref{fig:star_imbalance}, and the proof is postponed to Appendix \ref{app:imbal}.

\begin{restatable}{theorem}{starimbal}\label{thm:star_imbalance}Let $\mh$ be a function which maps any simple, loop-free graph $G=(V,E)$ to $\mh(G)\subseteq2^E$.  We have that $\kappa(\mh) \in \Omega(2^{|E|})$ whenever, for all $G$, $\mh(G)\supseteq\{H:\text{$H$ is the edge set of a star subgraph}, 2\leq|H|\leq 4\}$.   
\end{restatable}

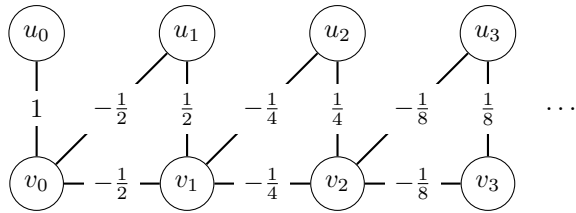
\begin{figure}[htb]
        \centering
        \begin{tikzpicture}
            \tikzset{vertex/.style = {shape=circle,draw,minimum size=2em}}
            \node[vertex] (v0) at (0,0) {$v_0$};
            \node[vertex] (u0) at (0,2) {$u_0$};
            \node[vertex] (v1) at (2,0) {$v_1$};
            \node[vertex] (u1) at (2,2) {$u_1$};
            \node[vertex] (v2) at (4,0) {$v_2$};
            \node[vertex] (u2) at (4,2) {$u_2$};
            \node[vertex] (v3) at (6,0) {$v_3$};
            \node[vertex] (u3) at (6,2) {$u_3$};

             \node at (7,1) {\ldots};

             \draw[thick] (u0) to node[midway, fill=white] {$1$} (v0) {};
             \draw[thick] (v0) to node[midway, fill=white] {$-\frac{1}{2}$} (u1) {};
             \draw[thick] (v0) to node[midway, fill=white] {$-\frac{1}{2}$} (v1) {};

             \draw[thick] (u1) to node[midway, fill=white] {$\frac{1}{2}$} (v1) {};
             \draw[thick] (v1) to node[midway, fill=white] {$-\frac{1}{4}$} (u2) {};
             \draw[thick] (v1) to node[midway, fill=white] {$-\frac{1}{4}$} (v2) {};

             \draw[thick] (u2) to node[midway, fill=white] {$\frac{1}{4}$} (v2) {};
             \draw[thick] (v2) to node[midway, fill=white] {$-\frac{1}{8}$} (u3) {};
             \draw[thick] (v2) to node[midway, fill=white] {$-\frac{1}{8}$} (v3) {};

             \draw[thick] (u3) to node[midway, fill=white] {$\frac{1}{8}$} (v3) {};
            
            \end{tikzpicture}
            \caption{The construction of the vector $\veg$ in the proof of Theorem \ref{thm:star_imbalance}. Each edge is labeled with its corresponding value in $\veg$.}
            \label{fig:star_imbalance}
    \end{figure}

\section{The Coloring Problem}
\label{sec:col}
In this section, we focus on the problem of vertex-coloring a connected graph with $n$ vertices.  The focus on connected graphs is natural. When coloring a disconnected graph, each connected component is given an independent subproblem, and so within each of these subproblems it is necessary to understand the connected case. Specifically, we are interested in analyzing the circuits of the \textit{fractional} coloring LP, i.e., the linear programming relaxation of the coloring problem (see Definition \ref{def:fractional_coloring_lp}).  The number of constraints of the associated \emph{fractional coloring polytope} is polynomial in its number of variables. Moreover, when the number of available colors is polynomial in the size of $G$ (which is typical), the number of constraints is also polynomial in the size of $G$. As such, in this case it does not follow from the general lower bound of \cite{deprs-26} that its circuit imbalance is exponential. 

As we will show in Section \ref{sec:kempe_imbal}, this LP indeed has at least exponential circuit imbalance measure.  In keeping with the broad themes and objectives of this work, we therefore explore the degree to which the 0/1 circuits alone are sufficiently useful in settings where one might ordinarily utilize the full set of circuits.  What we find is that the 0/1 circuits of the fractional coloring LP suffice to allow one to move between any pair of \textit{proper} colorings (i.e., between the integer solutions of the LP).  

As we show in Theorem \ref{thm:colcir} and Corollary \ref{cor:kemp_equiv}, the 0/1 circuits that move from a proper coloring to a proper coloring can be interpreted as a very natural generalization of so-called Kempe swaps.  Kempe swaps are a classical notion in the theory of graph coloring which give a natural and, in some sense, minimal way to reconfigure a coloring.  However, it is known that in some circumstances, Kempe swaps do not suffice to allow one to reconfigure \textit{any} coloring to \textit{any other} coloring.  In contrast, a consequence of our results is that these generalized Kempe swaps which come from 0/1 circuits \textit{do} suffice for this purpose.  

Moreover, we show that this is \textit{not} because the 0/1 circuits constitute so large a portion of $\{0,\pm1\}^n$ as to make it trivial to move between two colorings.  To explain this by way of contrast, it is known that the 0/1 circuits of the Matching and TSP LPs contain so many vectors from $\{0,\pm 1\}^n$ that almost all pairs of matchings (resp. TSP tours) are one circuit step apart \cite{kps-19}.  In Theorem \ref{thm:long_proper_walk} we show that reconfiguring two proper colorings via these generalized Kempe swaps may require up to $n$ steps.

Before proceeding to the results, we will first define all the necessary concepts and terminology.

\subsection{Kempe Dynamics}
\label{sec:kempe_dy}

We begin with the graph-theoretic definitions needed to study the polytopes in this section and Section \ref{sec:tree}. Let $G = (V,E)$ be a graph with the vertex set $V$, the edge set  $E$, and the minimum degree $\delta(G)$. Let $S\subseteq V$ be any subset of vertices of G. The \emph{induced subgraph}, $G[S]$ is the graph whose vertex set is $S$ and whose edge set consists of all the edges in $E$ that have both endpoints in $S$. An \emph{independent set} in a graph $G$, is a set $I\subseteq V$ such that no two vertices of $I$ are adjacent in $G$. 

A coloring of the vertices of a graph is an assignment of colors to the vertices of a graph. We say that a coloring of the vertices of $G$ is \emph{proper} if any two adjacent vertices receive different colors. A graph $G$ is \emph{t-colorable} if it has a proper coloring using at most \emph{t} colors. Given a set $\tau=\{1,2, \ldots t\}$, we say that the graph is $\tau$-colored if we color the vertices using all the colors in the set $\tau$. We denote a $\tau$-colored graph as $G(\tau)$. Note that this notation does not, in itself, specify the coloring, and each color class forms an independent set. The \emph{chromatic number} of a graph $G$, denoted $\chi(G)$, is the smallest integer $t$ such that $G$ is $t$-colorable. We formally define this as follows.  
\begin{definition}
    Let $G=(V,E)$ be a graph and let $t$ be a positive integer. A \textbf{$t$-coloring} of $G$ is a function $\phi:V\rightarrow\{1,2,\ldots t\}$ such that if $v_i$ and $v_j$ are adjacent then $\phi(v_i) \neq \phi(v_j)$. The numbers $1,2,\ldots, t $ are called the colors of the coloring $\phi$. 
\end{definition}

\begin{figure}[!ht]
\centering
\begin{tikzpicture}
\tikzstyle{every node}=[font=\LARGE]

\draw (-3.75,17.25) node[circle, draw, minimum size=0.5cm] (v1) {} node[left=6pt] {$1$};
\draw (-1.25,17.25) node[circle, draw, minimum size=0.5cm] (v2) {} node[right=6pt] {$2$};
\draw (-3.75,14.75) node[circle, draw, minimum size=0.5cm] (v3) {} node[left=6pt] {$3$};
\draw (-1.25,14.75) node[circle, draw, minimum size=0.5cm] (v4) {} node[right=6pt] {$1$};

\draw (2.5,17.25) node[circle, draw, minimum size=0.5cm] (v5) {} node[left=6pt]  {$1$};
\draw (5,17.25) node[circle, draw, minimum size=0.5cm] (v6) {} node[right=6pt]  {$3$};
\draw (2.5,14.75) node[circle, draw, minimum size=0.5cm] (v7) {} node[left=6pt]  {$3$};
\draw (5,14.75) node[circle, draw, minimum size=0.5cm] (v8) {} node[right=6pt] {$1$};

\draw (v3) -- (v4);
\draw (v2) -- (v4);
\draw (v1) -- (v2);
\draw (v1) -- (v3);

\draw (v5) -- (v6);
\draw (v6) -- (v8);
\draw (v7) -- (v8);
\draw (v5) -- (v7);

\end{tikzpicture}
\caption{Left: Proper coloring of the cycle on four vertices using colors \{1,2,3\}. Right: Proper coloring using the minimum number of colors of the cycle on four vertices using colors \{1,3\}.} 
\end{figure}
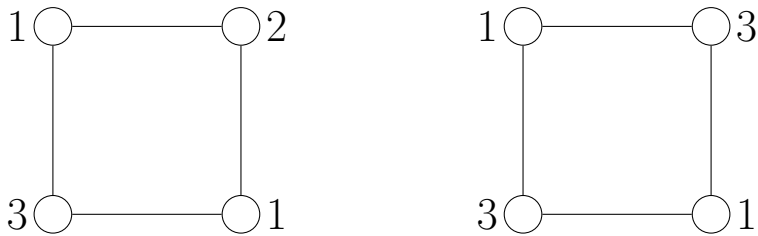

 We follow the standard notation found in \cite{l-16} for the following definitions. The \emph{coloring problem} is defined as determining whether it is possible to color the vertices of a graph with a limited number of colors, say $t$, such that the coloring is proper. The $t$-colorability problem for $t\geq 3$ is an NP-complete problem \cite{k-72}, and as a result determining the chromatic number of a graph is NP-complete. Traditionally, the coloring problem is modeled as the integer program below.

\begin{definition} \label{def:col}
Let $G = (V, E)$ be a graph and $\tau = \{ t_1, \ldots, t_r \}$ a set of colors. The following system is the integer program formulation for the coloring problem. 
    \begin{align*}
    \vex \in \R ^{|V| \times r} &: \sum_{i \in \tau} \vex({v, i}) = 1 \  \forall  v \in V \\
    & \hspace{-.1cm} \vex({v,i}) + \vex({u, i}) \leq 1 \ \forall  i \in \tau, \, uv \in E \\
    & \hspace{.26cm} \vex({v,i}) \in \{0,1\} \ \forall v \in V, \forall i\in\tau
\end{align*}
\end{definition}

The first set of constraints ensures that each vertex in our graph receives a color. The second set of constraints ensures that if two vertices $u$ and $v$ are connected via an edge they must receive a different color. We define the coloring characteristic vector to index which vertices have been assigned which color.

\begin{definition}
Let $G = (V, E)$ be a graph and $\tau = \{ t_1, \ldots, t_r \}$ a set of colors. Let $G(\tau)$ be a vertex $\tau$-colored graph. 
We define the \textbf{coloring characteristic vector} of $G(\tau)$ to be \textbf{$\ve{X}(G(\tau))$} $ \in \R^{|V| \times |\tau|}$ where $\ve{X}(G(\tau))_{j,t} = 1$ if vertex $j$ is colored with color $t$ and $0$ otherwise. 
\end{definition}
\noindent The entries of this vector are exactly the values of the variables $\vex(v,i)$. Relaxing the integrality constraints in the above integer programming formulation allows us to investigate the circuits associated with the coloring problem. In this relaxation, extreme-point solutions of the feasible region are able to have fractional values.  This corresponds to assigning a single vertex a ``fractional mix" of the available colors (provided this assignment still adheres to the constraints).  
This gives rise to the following definition of the fractional chromatic number. 
\begin{definition}\label{def:fractional_chromatic}
     Let $I(G)$ denote the set of all independent sets of vertices of a graph G, and let $I(G,u)$ denote the independent sets of $G$ that contain the vertex $u$. A \textbf{fractional coloring} of $G$ is a non-negative real function $f$ on $I(G)$ such that for any vertex $u$ of $G$,
     \begin{align*}
         \sum_{S \in I(G,u)}f(S)\geq1. 	
     \end{align*}
The sum of the values of $f$ is called the fractional coloring's weight, and the minimum possible weight of a fractional coloring is called the fractional chromatic number $\chi_f(G)$. 
\end{definition}

The following LP relaxation of the coloring IP is the formulation of the fractional coloring problem which we will consider.  In contrast to Definition \ref{def:fractional_chromatic}, which assigns weights to independent sets (which are understood to be sets of vertices who receive a common fraction of a common color), it directly assigns fractions of colors to vertices one at a time.

\begin{definition}\label{def:fractional_coloring_lp}
Let $G = (V, E)$ be a graph and $\tau = \{ t_1, \ldots, t_r \}$ a set of colors. The {\bf fractional coloring problem} has the following formulation:
\begin{align*}
    \vex \in \R ^{|V| \times r}: \sum_{i \in \tau} \vex({v, i}) = &  \;\; 1  \;\;\ \forall v \in V \\
    \vex({v,i}) + \vex({u, i}) \leq & \;\; 1  \;\;\ \forall i \in \tau, \, uv \in E \\
    \vex({v,i}) \geq &  \;\; 0  \;\;\ \forall v \in V, \forall i\in\tau
\end{align*}

\end{definition}

The polytope associated with this continuous relaxation is the \emph{fractional coloring polytope}. This is in contrast to the \emph{coloring polytope} which is defined as the convex hull of the feasible solutions to the IP formulation found in Definition \ref{def:col}. 
We note that, although the fractional coloring polytope has fractional extreme-points, we will still be predominantly concerned with the use of circuits to move between integral extreme-points.  That is, in this context, the circuits are being considered to the extent that they allow us to reconfigure proper $t$-colorings into other proper $t$-colorings.

In graph-theoretic terms, we say that two $t$-colorings, $\phi_1$ and $\phi_2$, are \emph{Kempe equivalent} (or simply equivalent) if there exists a sequence of so-called Kempe swaps (defined below) which transform the first coloring into the second coloring. By the definition of a Kempe swap, all intermediate colorings are proper. Equivalently, we can define the reconfiguration graph $H$ whose vertex set is the set of $t$-colorings of $G$, and where two colorings are adjacent if they differ by a single Kempe swap. Two colorings $\phi_1$ and $\phi_2$ are hence equivalent if their corresponding vertices lie in the same connected component of $H$. Below, we formally define these terms. 

\begin{definition}
    Let $G=(V,E)$ be a graph and $\phi:V\rightarrow \tau$ be a proper coloring where $|\tau|\geq 2$, and for each $x\in \tau$, let $V^x=\{v\in V:\phi(v)=x\}$. Given distinct colors $a,b\in \tau$, an \textbf{$(a,b)$-Kempe chain} of $(G,\phi)$ is a connected component of $G[V^a\cup V^b]$.  We say a subgraph $C$ of $G$ is a Kempe chain of $(G,\phi)$ if it is an $(a,b)$-Kempe chain for some colors $a$ and $b$.  When $\phi$ is clear from context, we may just refer to $(a,b)$-Kempe chains and Kempe chains of $G$.
\end{definition}

Given an $(a,b)$-Kempe chain, we can swap the colors $a$ and $b$ preforming a so-called Kempe swap. 

\begin{definition}
Given $G=(V,E)$, a proper coloring $\phi:V\rightarrow S$ with $|S|\geq 2$, distinct colors $a,b\in S$, and an $(a,b)$-Kempe chain $C$ of $(G,\phi)$, we say a coloring $\phi':V\rightarrow S$ is obtained from $\phi$ via a \textbf{Kempe swap} on $C$ if $\phi'(v)=\phi(v)$ for all $v\notin C$, $\phi'(v)=a$ if $v\in C$ and $\phi(v)=b$, and $\phi'(v)=b$ if $v\in C$ and $\phi(v)=a$. 
\end{definition}

The following classical notion of Kempe equivalence will play a crucial role in defining the circuits of the fractional coloring polytope.
We provide an example of two colorings that are Kempe equivalent in Figure \ref{fig:ex}.
\begin{definition}
    Two colorings are \textbf{Kempe equivalent} if we can transform one into the other through a sequence of Kempe swaps. If all pairs of $t$-colorings are Kempe equivalent, then we say the graph is \textbf{$t$-Kempe mixing}.
\end{definition}
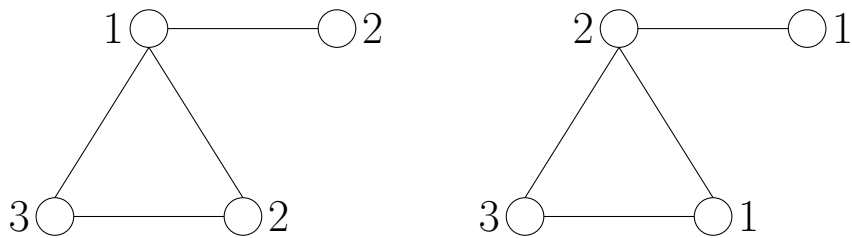
\begin{figure}[!ht]
\centering
\begin{tikzpicture}[scale = 0.9]
\tikzstyle{every node}=[font=\LARGE]
\draw (11.25,14.75) node[circle, draw, minimum size=0.5cm] (v1) {} node[left=6pt] {$1$};
\draw (10,12.25) node[circle, draw, minimum size=0.5cm] (v2) {} node[left=6pt] {$3$};
\draw (12.5,12.25) node[circle, draw, minimum size=0.5cm] (v3) {} node[right=6pt] {$2$};
\draw (13.75,14.75) node[circle, draw, minimum size=0.5cm] (v4) {} node[right=6pt] {$2$};
\draw  (10,12.5) -- (11.25,14.5);
\draw  (11.25,14.5) -- (12.5,12.5);
\draw  (10.25,12.25) -- (12.25,12.25);
\draw  (11.5,14.75) -- (13.5,14.75);
\draw (17.5,14.75) node[circle, draw, minimum size=0.5cm] (v1) {} node[left=6pt] {$2$};
\draw (16.25,12.25) node[circle, draw, minimum size=0.5cm] (v2) {} node[left=6pt] {$3$};
\draw (18.75,12.25) node[circle, draw, minimum size=0.5cm] (v3) {} node[right=6pt] {$1$};
\draw (20,14.75) node[circle, draw, minimum size=0.5cm] (v4) {} node[right=6pt] {$1$};
\draw  (17.75,14.75) -- (19.75,14.75);
\draw  (16.5,12.25) -- (18.5,12.25);
\draw  (17.5,14.5) -- (18.75,12.5);
\draw  (17.5,14.5) -- (16.25,12.5);
\end{tikzpicture}
\caption{Two Kempe equivalent colorings that differ by a Kempe swap using colors 1 and 2.}\label{fig:ex}
\end{figure}

An important note here is that Kempe chains almost always suffice to achieve reachability between any pair of colorings (i.e., graphs are usually $t$-Kempe mixing) as long as the graph in question has a low enough minimum degree in relation to the number of colors available. Combining several previous results of other authors, we get that two colorings are Kempe equivalent if there are enough colors related to the maximum degree of your graph.

\begin{theorem}[see \cite{b-19}, \cite{v-81}, \cite{f-17}]
    Let $3 \leq \Delta \leq k$. If $G$ is a connected graph with maximum degree $\Delta$, then all k-colorings of $G$ are Kempe-equivalent, unless $G$ is the triangular prism. 
\end{theorem} 

Due to this condition, current research on Kempe chains is concerned with the so-called $t$-Kempe diameter of any two $t$-colorings \cite{bhi-20,c-23}. 
\begin{definition}
    The \textbf{$t$-Kempe-diameter} of a graph $G$ is defined to be the maximum length of the shortest transformation between any two $t$-colorings of $G$. If a graph is not $t$-Kempe mixing, then we say the diameter is $\infty$.
\end{definition}
An important observation here is that Kempe swaps correspond to edges of the coloring polytope (though there are edges that do not correspond to Kempe swaps). Since each vertex in the reconfiguration graph will also correspond to an extreme-point of the coloring polytope, and since each Kempe swap corresponds to movement along a circuit, the Kempe diameter gives an upper bound on the circuit diameter. 

We now define a natural generalization of Kempe chains which allow for more than two colors in each chain. As we will show, swaps along these generalized Kempe chains will correspond to 0/1 circuits of the fractional coloring polytope which are feasible at extreme-points corresponding to proper colorings (Proposition \ref{prop:kem}).

\begin{definition}
    Let $G=(V,E)$ be a graph and $\phi:V\rightarrow \tau$ be a proper coloring where $|\tau|\geq k$, and for each $x\in \tau$, let $V^x=\{v\in V:\phi(v)=x\}$. Given distinct colors $a_1,\ldots,a_k\in S$, an \textbf{$(a_1,\ldots,a_k)$-Kempe chain} of $(G,\phi)$ is a connected component of $G[V^{a_1}\cup \cdots \cup V^{a_k}]$.  We say a subgraph $C$ of $G$ is a generalized Kempe chain of $(G,\phi)$ if it is an $(a_1,\ldots,a_k)$-Kempe chain of $(G,\phi)$ for some colors $a_1,\ldots,a_k$.  When $\phi$ is clear from context, we may just refer to $(a_1,\ldots,a_k)$-Kempe chains as generalized Kempe chains of $G$.
\end{definition}
We define a similar notion for a generalized Kempe swap.
\begin{definition}
Given $G=(V,E)$, a proper coloring $\phi:V\rightarrow \tau$ with $|\tau|\geq k$, distinct colors $a_1,\ldots,a_k\in \tau$, and an $(a_1,\ldots,a_k)$-Kempe chain $C$ of $(G,\phi)$, we say a proper $\tau$-coloring $\phi':V\rightarrow \tau$ is obtained from $\phi$ via a \textbf{generalized Kempe swap} on $C$ if 
\begin{enumerate}
\item $\phi'(v)=\phi(v)$ for all $v\notin C$,
\item $\phi'(v)\neq \phi(v)$ if $v\in C$, and
\item the subgraph of $C$ given by the edges $\{uv: \phi'(u)=\phi(v)\}$ is connected.
\end{enumerate}
\end{definition}

This final condition implies that if more than one vertex changes color, for any vertex that changes color, either it changes to the color of one of its neighbors or vice versa.
Note that in the above definition, if $k=2$ we recover exactly the definition of an $(a,b)$- Kempe chain. We can now define a correspondence from the coloring characteristic vector to a Kempe chain.

\begin{definition}
    Given $\tau$-colored graphs $G(\tau)_1$ and $G(\tau)_2$ with colorings  $\phi_1$ and $\phi_2$ such that $\phi_2$ is obtained from $\phi_1$ via a generalized Kempe swap, we say that this Kempe swap \textbf{corresponds} to the vector $\ves = \ve{X}(G(\tau)_2)-\ve{X}(G(\tau)_1)$.
\end{definition}

Next, we show the circumstances under which two characteristic vectors of proper colorings differ by a circuit.

\begin{theorem} \label{thm:colcir}
Let $G = (V, E)$ be a graph and $G(\tau)_1$ and $G(\tau)_2$ be two different proper $\tau$-colorings of $G$. The vector $\ves := \ve{X}(G(\tau)_2) - \ve{X}(G(\tau)_1)$ is a circuit of the fractional coloring polytope if and only if the subgraph $G(\ves)$ given by the vertices 
\[
    V(\ves) = \{ v : \ves({v, i}) \neq 0 \text{ for some } i \in \tau \} 
\]
and by the edges
\[
    E(\ves) = \{ uv \in E : \ves({u, i}) + \ves({v, i}) = 0, \ves({u, i}) \neq 0 \text{ for some } i \in \tau \}  
\]
is connected.
\end{theorem}

\begin{proof}
We first handle the case where $V(\ves)=\{v\}$ for some vertex $v$.  In this case, $G(\tau)_1$ and $G(\tau)_2$ only differ at $v$.  Without loss of generality, say $v$ has color 1 in $G(\tau)_1$ and color 2 in $G(\tau)_2$.  If $\ves$ is not a circuit, there exists a circuit $\ves'$ whose support is contained in the support of $\ves$, i.e., $\ves'({u,k})=0$ for all $(u,k)\notin\{(v,1),(v,2)\}$.  Moreover, since $\ves({v,1})+\ves({v,2})=0$, we have that $\ves'({v,1})+\ves'({v,2})=0$.  Since $\ves'\neq 0$, it follows that $\ves'$ is a scalar multiple of $\ves$.

Now, suppose that $G(\ves)$ is disconnected. Let $G'(\ves)$ be a connected component of $G(\ves)$. Let $\ves'$ be the vector given by $\ves'({v, k}) = \ves({v, k})$ if $v$ is a vertex of $G'(\ves)$ and zero otherwise. It is easy to see that the support of $\ves'$ is strictly contained in the support of $\ves$ and $\ves'({u, k}) + \ves'({v, k}) = 0$ whenever $\ves({u, k}) + \ves({v, k}) = 0$. Thus, $\ves$ is not a circuit.

Now, suppose that $G(\ves)$ is connected. If $\ves$ is not a circuit, then there exists $\ves' \neq \veo$ whose support is contained in the support of $\ves$ such that 
\begin{enumerate}[label = (\arabic*)]
    \item $\sum_{i\in\tau}\ves'(v,i)=0$ for all $i$,
    \item $\ves'({v, k}) = 0$ whenever $\ves({v, k}) = 0$, and
    \item $\ves'({u, k}) + \ves'({v, k}) = 0$ whenever $\ves({u, k}) + \ves({v, k}) = 0$.
\end{enumerate}

We claim the following: If $v\in V(\ves)$ with $\ves(v,j)$ and $\ves(v,k)$ not equal to zero, but $\ves'(v,j)=0$, then $\ves'(v,i)=0$ for all $i$, and for all $u$ with $uv\in E(\ves)$, $\ves'(u,i)=0$ for all $i$.

To prove this claim, let $u$ and $v$ be as described.  Note that since $\ves$ is the difference of characteristic vectors of two proper colorings, there are exactly two colors $j$ and $k$ such that $\ves(v,j)\neq 0$ and $\ves(v,k)\neq 0$.  This combined with points (1) and (2) above implies that $\ves'(v,j)=-\ves'(v,k)$, and in particular, if either of these quantities is zero, they both are.  This proves the first conclusion of our claim.  Since $uv\in E(\ves)$, by point (3) above there is some color $i\in\tau$ such that $\ves'(u,i)+\ves'(v,i)=0$.  Then in fact, $\ves'(u,i)=0$, and as above, $\ves'(u,h)=0$ for the other color $h\in\tau$ such that $\ves(u,h)\neq 0$.  This proves the claim.

This immediately implies that if $\ves({v,k}) \neq 0$, then $\ves'({v, k})\neq 0$.  Suppose for the sake of a contradiction that there exist vertices $v$ and colors $k$ such that $|\ves({v,k})|>\ves'({v,k})=0$.  Since $G(\ves)$ is connected and contains more than one vertex, it follows from our claim above that for \textit{all} $u\in V$ and $i\in\tau$, $\ves'(v,i)=0$, contradicting that $\ves'\neq\veo$.

Now, since the support of $\ves'$ is equal to the support of $\ves$ then there exist some $u, v, \text{and } k$ such that $\ves({u, k}) + \ves({v, k}) \neq 0$ and $\ves'({u, k}) + \ves'({v, k}) = 0$. Since $\ves$ is the difference of two vectors in $\{ 0, 1 \}^{V\times\tau}$, it follows that $\ves({u, k}) + \ves({v, k}) \in \{ 0, \pm 1, \pm 2 \}$. By assumption, $\ves({u, k}) + \ves({v, k}) \neq 0$, and if $\ves({u, k}) + \ves({v, k}) = \pm 1$ then exactly one of $\ves({u, k})$ or $\ves({v, k})$ is zero, and thus exactly one of $\ves'({u,k})$ or $\ves'({v,k})$ is 0. Since $\ves'({u, k}) + \ves'({v, k}) = 0$, we have that $\ves'({u, k}) = \ves'({v, k}) = 0$, contradicting that the support of $\ves$ is equal to the support of $\ves'$. Finally, $\ves({u, k}) + \ves({v, k}) \neq \pm 2$ because this would imply that  in one of the colorings $G(\tau)_1$ or $G(\tau)_2$, $u, v$ are adjacent vertices with the same color, a contradiction. Thus, no such $\ves'$ exists, and therefore $\ves$ is a circuit.
\end{proof}

\begin{corollary}\label{cor:kemp_equiv}
If two proper $\tau$-colored graphs satisfy the hypotheses of Theorem \ref{thm:colcir}, then the difference of their characteristic vectors corresponds to a generalized Kempe swap.
For any colors $a_1,\ldots,a_k$, any $(a_1,\ldots,a_k)$-Kempe swap corresponds to a circuit of the fractional coloring polytope.
\end{corollary}

Although we have shown that the difference of  characteristic vectors of two \textit{proper} colorings is a circuit if and only if those colorings differ by a generalized Kempe swap (which are themselves only defined on proper colorings), we observe that there are nevertheless 0/1 circuits of the coloring polytope which do not correspond to Kempe swaps.  In particular, there exist improper colorings whose characteristic vectors differ by a circuit.   Figure \ref{fig:improper} exhibits two improper colorings of $K_3$, and it can be easily verified that the difference of their characteristic vectors is indeed a circuit of the coloring polytope.

\begin{figure}[!ht]
\centering
\begin{tikzpicture}
\tikzstyle{every node}=[font=\LARGE]
\tikzset{graph node/.style={shape=circle,draw=black,inner sep=0pt, minimum size=.4cm}}
\draw (6,13) node[circle, draw, minimum size=0.5cm] (v1) {} node[left=6pt] {$1$};
\draw (7.5,13) node[circle, draw, minimum size=0.5cm] (v2) {} node[right=6pt] {$1$};
\draw (6.75,14.25) node[circle, draw, minimum size=0.5cm] (v3) {} node[left=6pt] {$2$};

\draw (v1)--(v2)--(v3)--(v1);

\draw (12,13) node[circle, draw, minimum size=0.5cm] (u1) {} node[left=6pt] {$2$};
\draw (13.5,13) node[circle, draw, minimum size=0.5cm] (u2) {} node[right=6pt] {$2$};
\draw (12.75,14.25) node[circle, draw, minimum size=0.5cm] (u3) {} node[left=6pt] {$1$};

\draw (u1)--(u2)--(u3)--(u1);

\end{tikzpicture}
\caption{Two improper colorings of $K_3$ whose characteristic vectors differ by a circuit.}\label{fig:improper}
\end{figure}
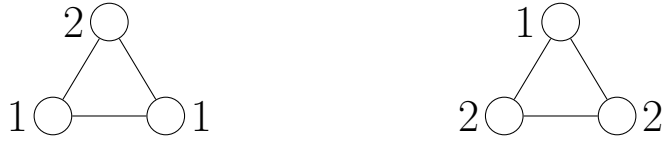

As mentioned in Section \ref{sec:intro}, 0/1 circuits are a useful tool in circuit augmentation, and so it is valuable to have a complete characterization of the set of 0/1 circuits.  Recall also that we are primarily concerned with 0/1 circuits that allow us to move between \textit{integral} extreme-point solutions of the fractional coloring polytope (i.e., those circuits which are differences of characteristic vectors of proper colorings) Although the example above shows that Theorem \ref{thm:colcir} does not describe all 0/1 circuits, it does succeed in describing all of the 0/1 circuits that we care about.

Next, we show that, at an extreme-point of the fractional coloring polytope which corresponds to a proper coloring, all feasible 0/1 circuits correspond to generalized Kempe swaps.  Note that this is not already implied by the previous results.  In general, at an extreme-point solution in $\{0,1\}^n$, it is possible that moving along a feasible 0/1 circuit causes one to arrive at a non-integral (or even non-extreme-point) solution.   
We must first define the following, and then provide a formal statement and proof.

\begin{definition}
    Given a graph $G$ and a set of colors $\tau$, we say that a circuit $\ves$ of the fractional coloring polytope uses colors $\tau'\subset\tau$ if for all vertices $v$ of $G$ and all colors $j\in\tau\setminus\tau'$, $\ves(v,j)=0$.
\end{definition}

\begin{proposition}\label{prop:kem}
    Let $G(\tau)$ be a $\tau$-colored graph with corresponding coloring $\phi$ where $G(\tau)$ corresponds to an extreme-point solution $\ve{X}(G(\tau))$ of the fractional coloring polytope. If $\ves$ is a 0/1 circuit which is feasible at $\ve{X}(G(\tau))$ and which uses $\tau'=\{a_1,\ldots,a_k\}$, then $\ves$ corresponds to an $(a_1,\ldots,a_k)$-Kempe swap at $\phi$.
\end{proposition}

\begin{proof}
        We will define a proper $\tau$-coloring $\phi'$ such that the resulting $\tau$-colored graph $G(\tau)'$ satisfies $\ves = \ve{X}(G(\tau)')-\ve{X}(G(\tau))$.     
        Let $G(\ves), V(\ves),$ and $E(\ves)$ be defined as in Theorem \ref{thm:colcir}. Since $\ves$ is feasible at $\ve{X}(G(\tau))$ and 0/1, we have that
        \begin{itemize}
            \item for all $v\in V(\ves)$, $\sum_{i\in \tau} \ves(v,i)=0$,  
            \item $\ves(v,\phi(v))=-1$ for all $v \in V(\ves)$, and
            \item $\ves(v,i)\geq 0$ for all $v\in V(\ves)$ and $i\neq \phi(v)$.
        \end{itemize}
        We therefore have that for all $v\in V(\ves)$, there is exactly one color $i^v$ such that $\ves(v,i^v)=1$, and for all colors $j\in\tau\setminus\{\phi(v),i^v\}$, $\ves(v,j)=0$.  We define $\phi'(v)=i^v$ for all $v\in V(\ves)$ and $\phi'(v)=\phi(v)$ otherwise. We let $G(\tau)'$ be the $\tau$-colored graph colored by $\phi'$.
        Then $\ves = \ve{X}(G(\tau)')-\ve{X}(G(\tau))$.  Since $\ves$ is a circuit, it follows from Theorem \ref{thm:colcir} that $G(\ves)$ is connected, and thus by Corollary \ref{cor:kemp_equiv}, $\ves$ corresponds to an $(a_1,\ldots,a_k)$-Kempe swap at $\phi$, as desired.
\end{proof} 

We note that we can indeed find 0/1 circuits that do not correspond to ordinary Kempe swaps (i.e., those using only two colors) Figure \ref{fig:kempe_coun} depicts such an example.  The depicted colorings differ by a generalized Kempe swap, and therefore the corresponding extreme-points of the fractional coloring polytope differ by a circuit.  However, it can be readily seen that they do not differ by a single Kempe swap.

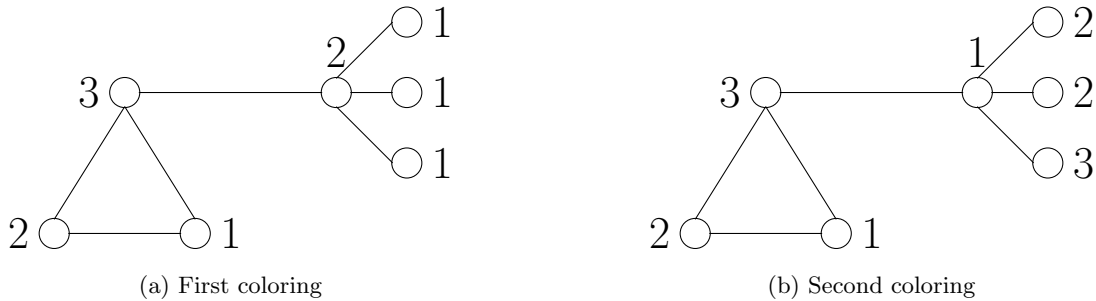
\begin{figure}[!ht]
\centering
\begin{subfigure}{0.45\textwidth}
    \centering
    \begin{tikzpicture}[scale = 0.47]
        \tikzstyle{every node}=[font=\LARGE]
        \draw  (10,12.5) -- (11.25,14.5);
        \draw  (11.25,14.5) -- (12.5,12.5);
        \draw  (10.25,12.25) -- (12.25,12.25);
        \draw  (11.5,14.75) -- (14.75,14.75);
        \draw  (15,15) -- (16,16);
        \draw  (15.25,14.75) -- (16,14.75);
        \draw  (15,14.5) -- (16,13.5);
        \draw (11.25,14.75) node[circle, draw, minimum size=0.4cm, fill = white] (v1) {} node[left=6pt] {$3$};
        \draw  (10,12.25) node[circle, draw, minimum size=0.4cm,fill = white] (v1) {} node[left=6pt] {$2$};
        \draw  (12.5,12.25) node[circle, draw, minimum size=0.4cm,fill = white] (v1) {} node[right=6pt] {$1$};
        \draw (15,14.75) node[circle, draw, minimum size=0.4cm,fill = white] (v1) {} node[above=6pt] {$2$};
        \draw (16.25,16) node[circle, draw, minimum size=0.4cm,fill = white] (v1) {} node[right=6pt] {$1$};
        \draw (16.25,14.75) node[circle, draw, minimum size=0.4cm,fill = white] (v1) {} node[right=6pt] {$1$};
        \draw (16.25,13.5) node[circle, draw, minimum size=0.4cm,fill = white] (v1) {} node[right=6pt] {$1$};
        
    \end{tikzpicture}
    \caption{First coloring}
\end{subfigure}
\hfill
\begin{subfigure}{0.45\textwidth}
    \centering
    \begin{tikzpicture}[scale = 0.47]
        \tikzstyle{every node}=[font=\LARGE]
        \draw  (10,12.5) -- (11.25,14.5);
        \draw  (11.25,14.5) -- (12.5,12.5);
        \draw  (10.25,12.25) -- (12.25,12.25);
        \draw  (11.5,14.75) -- (14.75,14.75);
        \draw  (15,15) -- (16,16);
        \draw  (15.25,14.75) -- (16,14.75);
        \draw  (15,14.5) -- (16,13.5);
        \draw (11.25,14.75) node[circle, draw, minimum size=0.4cm, fill = white] (v1) {} node[left=6pt] {$3$};
        \draw  (10,12.25) node[circle, draw, minimum size=0.4cm,fill = white] (v1) {} node[left=6pt] {$2$};
        \draw (12.5,12.25) node[circle, draw, minimum size=0.4cm,fill = white] (v1) {} node[right=6pt] {$1$};
        \draw (15,14.75) node[circle, draw, minimum size=0.4cm,fill = white] (v1) {} node[above=6pt] {$1$};
        \draw (16.25,16) node[circle, draw, minimum size=0.4cm,fill = white] (v1) {} node[right=6pt] {$2$};
        \draw(16.25,14.75) node[circle, draw, minimum size=0.4cm,fill = white] (v1) {} node[right=6pt] {$2$};
        \draw (16.25,13.5) node[circle, draw, minimum size=0.4cm,fill = white] (v1) {} node[right=6pt] {$4$};
        
    \end{tikzpicture}
    \caption{Second coloring}
\end{subfigure}

\caption{Two colorings whose corresponding extreme-points differ by a circuit, but are not a single (ordinary) Kempe swap apart.}
\label{fig:kempe_coun}
\end{figure}

We now show that, although colorings are not always equivalent via Kempe swaps, they are equivalent via generalized Kempe swaps.  As a consequence, we get that the set of 0/1 circuits of the fractional coloring LP suffice to traverse the proper colorings of a graph.

\begin{theorem}\label{thm:coloring-walks}
    Let $G$ be a $k$-colorable graph. Every $k$-coloring of $G$ is equivalent via generalized Kempe swaps. 
\end{theorem}

\begin{proof}
    Let $\tau=\{1,\ldots,k\}$ and let $\phi$ and $\phi'$ be two $\tau$-colorings of $G$.  Let $C_1,\ldots, C_q$ be the connected components of the subgraph $H=(V',E')$ of $G$ given by the vertices $V'=\{v\in V: \phi(v)\neq\phi'(v)\}$ and the edges $E'=\{uv:u,v\in V', \phi(u)=\phi'(v)\}$.  By definition, each component $C_i$ is a $(1,\ldots,k)$-Kempe chain of $(G,\phi)$.  Let $\phi_0=\phi$ and for $i=1,\ldots q$, define $\phi_i$ as the coloring obtained from $\phi_{i-1}$ as follows:
    \begin{itemize}
        \item $\phi_i(v)=\phi_{i-1}(v)$ for all $v\notin C_i$, and
        \item $\phi_i(v)=\phi'(v)$ if $v\in C_i$.
    \end{itemize}
    We show that if $\phi_{i-1}$ is a proper coloring, then $\phi_i$ is a proper coloring.  Suppose not for the sake of a contradiction.  Then in $\phi_i$ there exist vertices $u$ and $v$ adjacent in $G$ such that $\phi_i(u)=\phi_i(v)$.  Since $\phi_{i-1}$ is a proper coloring, we have that, without loss of generality, $\phi_{i-1}(u)=\phi_{i}(v)=j$. We have that  $\phi_i(v)=\phi'(v)\neq\phi'(u)$, so $j=\phi_{i-1}(u)=\phi(u)$.  Then by definition, $uv\in E'$, and thus $u\in C_i$, contradicting that $\phi_{i-1}(u)=\phi_i(u)$.

    Clearly, $C_{i+1},\ldots C_q$ are also $(1,\ldots k)$-Kempe chains of $(G,\phi_i)$.
    Since $\phi_i$ is a proper coloring, we have that $\phi_i$ is obtained from $\phi_{i-1}$ via a generalized Kempe swap on $C_i$, by definition.  Finally, we note that $\phi_q=\phi'$ by definition.  Thus, $\phi$ and $\phi'$ are equivalent via generalized Kempe swaps, as desired.
\end{proof}

\begin{corollary}\label{cor:col_0_1_sufficient}
    Given a graph $G$ and a finite set $\tau$, let $\vex_1=\veX(G(\tau)_1)$ and $\vex_2=\veX(G(\tau)_2)$ be the characteristic vectors of two proper $\tau$-colorings of $G$.  Then there exists a circuit walk in the fractional coloring polytope from $\vex_1$ to $\vex_2$ using only 0/1 circuits and such that all points visited in the walk are characteristic vectors of proper colorings of $G$.
\end{corollary}

As in the setting of ordinary Kempe swaps, we can use the results proved here to analyze the Kempe diameter when generalized Kempe swaps (especially in those cases where ordinary Kempe swaps do not achieve reachability). In \cite{bhi-20}, the authors explore subclasses of perfect graphs and give bounds on the Kempe diameter. For example, they show that bipartite graphs and cographs have Kempe diameter at most $3n/2$ and $2n\log(n)$ respectively. Although generalized Kempe swaps are far more permissive than ordinary Kempe swaps, we show that there is still a linear lower bound on the Kempe diameter even when using generalized Kempe swaps.

\begin{definition}
    Given a graph $G$, a color set $\tau$, and two proper $\tau$-colorings $G(\tau)_1$ and $G(\tau)_2$, we say a {\bf circuit walk from $\ve{X}(G(\tau)_1)$ to $\ve{X}(G(\tau)_2)$ in the fractional coloring polytope is a proper walk} if each point visited in the walk is the characteristic vector of a proper $\tau$-coloring of $G$.
\end{definition}

Note that, by definition, the circuits in a proper walk are necessarily 0/1 circuits, and thus correspond to generalized Kempe swaps.

\begin{theorem}\label{thm:long_proper_walk}
    For all $n$, there exists a graph $G$ on $n$ vertices, a color set $\tau$, and proper $\tau$-colorings $G(\tau)_1$ and $G(\tau)_2$ such that a shortest proper walk from $\ve{X}(G(\tau)_1)$ to $\ve{X}(G(\tau)_2)$ has length $n$. 
\end{theorem}

\begin{proof}
    Let $G=K_n$, the complete graph on $n$ vertices with vertex set $\{v_1,\ldots,v_n\}$.  Let $\tau\supseteq\{c_1,\ldots,c_n,d_1,\ldots,d_n\}$, let $G(\tau)_1$ be the coloring that colors the vertex $v_i$ with the color $c_i$, and let $G(\tau)_2$ be the coloring that colors the vertex $v_i$ with the color $d_i$.  Clearly, $n$ 0/1 circuit steps suffice to move from $\ve{X}(G(\tau)_1)$ to $\ve{X}(G(\tau)_2)$ since the color of each vertex can be changed one at a time.

    Next, we will show as an intermediate step that for any $\tau$-colorings $G(\tau)'$ and $G(\tau)''$, if $\ve{X}(G(\tau)'')$ is a single circuit step away from $\ve{X}(G(\tau)')$, then there is at most one color that appears in $G(\tau)''$ but not in $G(\tau)'$. By Theorem \ref{thm:colcir}, if we let $\ves = \ve{X}(G(\tau)'')-\ve{X}(G(\tau)')$, then $(V(\ves),E(\ves))$ is connected.  We now consider the digraph $(V(\ves),A(\ves))$ where
    \[
    A(\ves) = \{ (u,v): \ves(u,i)+\ves(v,i) = 0, \ves(u,i)< 0 \text{ for some }i\in\tau\}
    \]
    That is, an arc is directed from $u$ to $v$ if $v$'s color in $G(\tau)''$ is the same as $u$'s color in $G(\tau)'$.

    Next we note that if a color appears in $G(\tau)''$ but not in $G(\tau)'$, then this corresponds to a source in $(V(\ves),A(\ves))$, i.e., a vertex not incident to any in-arcs.  Now, assume for the sake of a contradiction that $(V(\ves),A(\ves))$ has at least two sources $w_1$ and $w_2$.  Consider any simple path $P$ from $w_1$ to $w_2$ in $(V(\ves),E(\ves))$. The edge of $P$ incident to $w_i$ is an out-arc in $(V(\ves),A(\ves))$ since $w_i$ is a source.  This means that there is at least one vertex $u$ in $P$ whose incident edges in $P$ are both in-arcs in $(V(\ves),A(\ves))$. By definition, this means that both neighbors of $u$ in $P$ have the same color in $G(\tau)'$.  However, since $G$ is the complete graph, this contradicts that $G(\tau)'$ is a proper $\tau$-coloring of $G$.

    Thus, any circuit step between two proper $\tau$-colorings of $G$ can introduce at most one new color.  It follows that $\ve{X}(G(\tau)_2)$ cannot be reached from $\ve{X}(G(\tau)_1)$ in fewer than $n$ circuit steps if we require those steps to move between the proper $\tau$ colorings of $G$.
\end{proof}

We note that the length of a shortest proper walk between two proper colorings depends heavily on the number of available colors.  Theorem \ref{thm:long_proper_walk} relies on the fact that the number of colors available is twice the number of colors necessary to color the graph.  As we show below, if we have only $n$ available colors to color $K_n$, then for any two proper colorings, there is a proper walk between them of length at most 2.

\begin{proposition}
    Let $G=K_n$ and $|\tau|=n$.  Then for any proper $\tau$-colorings $G(\tau)_1$ and $G(\tau)_2$, there exists a proper walk of length at most two between $\ve{X}(G(\tau)_1)$ and $\ve{X}(G(\tau)_2)$.
\end{proposition}
\begin{proof}
    We may assume that $\ves = \ve{X}(G(\tau)_1)-\ve{X}(G(\tau)_2)$ is not a circuit. We reuse the notation $A(\ves)$ introduced in the proof of Theorem \ref{thm:long_proper_walk}.  Since all $n$ colors appear in both $G(\tau)_1$ and $G(\tau)_2$, we have that $(V(\ves),A(\ves))$ is the disjoint union of $t$ directed cycles for some $t$.  Since $\ves$ is not a circuit, we have $t\geq 2$.  Assume without loss of generality that these cycles have vertex sets $V_1,\ldots,V_t$ where $V_1 = \{v_1,\ldots,v_{n_1}\}$ and for $2\leq i\leq t$, $V_i = \{v_{n_{i-1}+1},\ldots, v_{n_i}\}$.  Further assume without loss of generality that the cycle $V_i$ has arcs $\{(v_j,v_{j+1}): n_{i-1}+1 \leq j \leq v_{n_i} - 1\} \cup \{(v_{n_i},v_{n_{i-1}+1}\}$.  Note then that for $k> n_t$, the vertex $v_k$ has the same color in $G(\tau)_1$ and $G(\tau)_2$.

    Now, let $G(\tau)'$ be obtained from $G(\tau)_1$ in the following way.  For $k> n_t$, $v_k$ receives the same color in $G(\tau)'$ as in $G(\tau)_1$.  For $2\leq k\leq n_t$, in $G(\tau)'$ $v_k$ receives the color of $v_{k-1}$ in $G(\tau_1)$, and in $G(\tau)'$ $v_1$ receives the color of $v_{n_t}$ in $G(\tau_1)$.  Then if $\veg_1 = \ve{X}(G(\tau)') - \ve{X}(G(\tau)_1)$, it is not hard to see that $(V(\veg_1),E(\veg_1))$ is a single cycle, and so $\veg_1$ is a circuit by Theorem \ref{thm:colcir}.  Likewise, it is not hard to see that if $\veg_2 = \ve{X}(G(\tau)_2)-\ve{X}(G(\tau)')$ then $V(\veg_2)=\{v_1,v_{n_1+1},v_{n_2+1},\ldots, v_{n_{t-1}+1}\}$, and $(V(\veg_2),E(\veg_2))$ is a cycle.  Then $\veg_2$ is also a circuit by Theorem \ref{thm:colcir}.  Therefore, there exists a proper walk between $\ve{X}(G(\tau)_1)$ and $\ve{X}(G(\tau)_2)$ of length at most two, as desired.
\end{proof}

In the next section, we give an explicit example that the circuits of the fractional coloring polytope have unbounded circuit imbalance measure.

\subsection{Exponential Circuit Imbalance}
\label{sec:kempe_imbal}

We show that the fractional coloring polytope has circuit imbalance which grows at least exponentially in the size of the underlying graph. In particular, we show this is true even when restricted to a constant number of colors.  As noted before, exponential circuit imbalance does not follow from the general lower bound of \cite{deprs-26} when the number of available colors is polynomial in the size of the graph. 
\begin{lemma}\label{lem:unbounded_imbalance}
    Given a graph $G=(V,E)$, let $A$ and $B$ refer to the equality and inequality constraint matrices (respectively) of the fractional coloring polytope.  Then $\kappa(A,B)\in \Omega(2^{|V|})$, even for a constant number of colors.
\end{lemma}
\begin{proof}
    Let $t\geq 3$ be odd and let $G=(V,E)$ where $V=\bigcup_{i = 0}^t\{u_i, v_i, w_i\}$ and $E=\{u_tv_t, v_tw_t, w_tu_t\}\cup \bigcup_{i=0}^{t-1}\{u_iv_i,v_iw_i,w_iu_i,v_iu_{i+1}\}$.  Let $K=\{a,b,c,d\}$ be the set of colors.  Let $\veg$ (see Figure \ref{fig:unbounded_imbalance}) be such that for all even $i\in\{0,\ldots,t\}$,

\begin{align*}
    \veg(u_i,a)= \frac{1}{2^i}
    &&\veg(u_i,b)= -\frac{1}{2^{i+1}}
    &&\veg(u_i,c)=-\frac{1}{2^{i+1}}
    &&\veg(u_i,d)=0\\
    \veg(v_i,a)= 0
    &&\veg(v_i,b)= \frac{1}{2^{i+1}}
    &&\veg(v_i,c)= -\frac{1}{2^{i+1}}   
    &&\veg(v_i,d)= 0\\
    \veg(w_i,a)= 0
    &&\veg(w_i,b)= -\frac{1}{2^{i+1}}
    &&\veg(w_i,c)= \frac{1}{2^{i+1}}
    &&\veg(w_i,d)= 0\\
    \veg(u_{i+1},a)= -\frac{1}{2^{i+2}}
    &&\veg(u_{i+1},b)= 0
    &&\veg(u_{i+1},c)= \frac{1}{2^{i+1}}
    &&\veg(u_{i+1},d)= -\frac{1}{2^{i+2}}\\
    \veg(v_{i+1},a)= -\frac{1}{2^{i+2}}
    &&\veg(v_{i+1},b)= 0
    &&\veg(v_{i+1},c)= 0
    &&\veg(v_{i+1},d)= \frac{1}{2^{i+2}}\\
    \veg(w_{i+1},a)= \frac{1}{2^{i+2}}
    &&\veg(w_{i+1},b)= 0
    &&\veg(w_{i+1},c)= 0
    &&\veg(w_{i+1},d)= -\frac{1}{2^{i+2}}.
\end{align*}

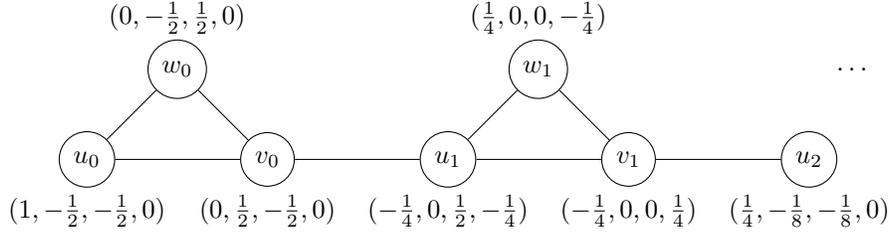
\begin{figure}[!ht]
	\begin{center}
			\begin{tikzpicture}[scale=1.05]	
				\tikzset{vertex/.style = {shape=circle,draw,minimum size=2em}}
                \tikzset{
                empty/.style={shape=circle,draw=none,inner sep=0pt, minimum size=4pt}
			}
				\tikzset{
                dot edge/.style={black, dotted , line width=.4mm}
			}		
                \node[vertex] (u0) at (0,0){$u_0$};
                \node[vertex] (v0) at (2,0){$v_0$};
                \node[vertex] (w0) at (1,1){$w_0$};

                \node[empty, label={[label distance=1]below:$(1,-\frac{1}{2},-\frac{1}{2},0)$}] (lu0) at (0,-.2){};
                \node[empty, label={[label distance=1]below:$(0,\frac{1}{2},-\frac{1}{2},0)$}] (lv0) at (2,-.2){};
                \node[empty, label={[label distance=1]above:$(0,-\frac{1}{2},\frac{1}{2},0)$}] (lw0) at (1,1.2){};

                \node[vertex] (u1) at (4,0){$u_1$};
                \node[vertex] (v1) at (6,0){$v_1$};
                \node[vertex] (w1) at (5,1){$w_1$};

                \node[empty, label={[label distance=1]below:$(-\frac{1}{4},0,\frac{1}{2},-\frac{1}{4})$}] (lu1) at (4,-.2){};
                \node[empty, label={[label distance=1]below:$(-\frac{1}{4},0,0,\frac{1}{4})$}] (lv1) at (6,-.2){};
                \node[empty, label={[label distance=1]above:$(\frac{1}{4},0,0,-\frac{1}{4})$}] (lw1) at (5,1.2){};

                \node[vertex] (u2) at (8,0){$u_2$};
                \node at (8.5,1) {\ldots};

                \node[empty, label={[label distance=1]below:$(\frac{1}{4},-\frac{1}{8},-\frac{1}{8},0)$}] (lu1) at (8,-.2){};

				\draw (u0)--(v0);
				\draw (u0)--(w0);
				\draw (v0)--(w0);

                \draw (v0)--(u1);

                \draw (u1)--(v1);
				\draw (u1)--(w1);
				\draw (v1)--(w1);

                \draw (v1)--(u2);

                \node[empty] (e1) at (8.5,0){};
                \node[empty] (e2) at (8.5,0.5){};
                
			\end{tikzpicture}
	\end{center}
	\caption{The construction of the vector $\veg$ in the proof of Lemma \ref{lem:unbounded_imbalance}. Each vertex $z_i$ is labeled by $(\veg(z_i,a),\veg(z_i,b),\veg(z_i,c),\veg(z_i,d))$, the values of the vector $\veg$ corresponding to that vertex.}
	\label{fig:unbounded_imbalance}
\end{figure}

Clearly, $\veg\in\ker(A)$.  Note that $\veg(u_i,a) = 4\veg(u_{i+2},a)$ for all even $i\in\{0,\ldots t-3\}$.  If $\veg$ is a circuit, the result follows.
If $\veg$ is not a circuit, there exists a circuit $\veh$ with $\supp(B\veh)\subseteq\supp(B\veg)$. We will show that in any such circuit $\veh$,  $\veh(u_i,a) = 4\veh(u_{i+2},a)$. 

The fact that $\supp(B\veh)\subseteq\supp(B\veg)$ implies that for any row $i$ of $B$ such that $B_i\veg = 0$, we have that $B_i\veh = 0$.  Then, since the identity matrix is a submatrix of $B$, we have that for all vertices $v$ and colors $f$ with $\veg(v,f)=0$, $\veh(v,f)=0$. 
Likewise, for all edges $uv$ and colors $f$ with $\veg(u,f) = -\veg(v,f)$, $\veh(u,f) = -\veh(v,f)$.
Finally, the fact that $\veh\in\ker(A)$ implies that for all vertices $z$,
\[
\veh(z,a) + \veh(z,b) + \veh(z,c) + \veh(z,d) = 0.
\]
By the above, we obtain the following equalities for even values of $i$: 
\begin{align*}
\veh(u_i,a) 
&= -\veh(u_i,b)-\veh(u_i,c)\\ 
&= \veh(v_i,b)+\veh(w_i,c)\\
&= -\veh(v_i,c) - \veh(v_i,c)
= -2\veh(v_i,c)
= 2\veh(u_{i+1},c)\\
&= 2(-\veh(u_{i+1},a)-\veh(u_{i+1},d))\\
&= 2(\veh(w_{i+1},a)+\veh(v_{i+1},d))\\
&= 2(-\veh(v_{i+1},a) +\veh(v_{i+1},d))
= 4(-\veh(v_{i+1},a))
= 4\veh(u_{i+2},a),
\end{align*}
as desired. 
\end{proof}

We note that although the graph $G$ in the construction of Lemma \ref{lem:unbounded_imbalance} does not itself require 4 colors to be properly colored, the circuit $\veh$ constructed in the proof would still be a circuit even if $G$ were merely a \textit{subgraph} of the graph to be colored.  As such, this proof does not rely on the fact that it uses more colors than are necessary.

\section{Forests and the Rank Inequalities}
\label{sec:tree}

In this section, we devise some insight into the 0/1 circuits arising from the constraint system of the so-called rank inequalities of over a simple graph $G$. Recall that for a graph $G = (V,E)$ with vertex set $V$ and edge set $E$ (or $V(G)$ and $E(G)$ when the graph $G$ is not clear from the context), this system can be stated as follows. 

\begin{equation*}
\begin{array}{rcrclcr}
    &&\sum_{e \in E(G[U])}\vex(e) &\leq& |U|-1 &\quad \text{for all } \emptyset \neq U \subseteq V.& \tag{Rank}
\end{array}\label{Rank}
\end{equation*}

As discussed in Section \ref{sec:intro}, these inequalities appear in the standard formulation of the feasible region of the maximum weight forest problem:

\begin{definition}
    Let $G = (V,E)$ be a simple graph. The {\bf maximum weight forest problem} has a feasible set that corresponds to the extreme-points of the polytope given by the following system of constraints:
    \begin{equation*}
\begin{array}{rcrcllr}
   \vex \in \R^{|E|}:  &&\sum_{e \in E(G[U])}\vex(e) &\leq& |U|-1 &\quad \forall \emptyset \neq U \subseteq V&\\
    && \vex(e) &\geq& 0 &\quad \forall e \in E,& \tag{MWF}
\end{array}\label{MWF}
\end{equation*}
where $G[U]$ is the subgraph of $G$ induced by $U \subseteq V$.
\end{definition}

Here $\vex(e)$ denotes the entry in vector $\vex \in \R^{|E|}$ corresponding to edge $e \in E$. We will use the notation $\vex_{e}$ to denote the standard unit vector where only $\vex(e)$ is nonzero. In the formulation (\ref{MWF}), these \emph{rank inequality constraints} from (\ref{Rank}) enforce that at an integral extreme-point, the support of $\vex$ contains no cycles. In matroid theory, these rank inequalities appear as constraints in the independent matroid polytope for graphic matroids; 
see \cite{oxley-06} for further details. Here, they ensure that the solutions to (\ref{MWF}) correspond to forests, i.e., independent sets of the graphic matroid associated with $G$. 

In Section \ref{sec:rankddiameters} we will explore the circuit diameter of (\ref{MWF}).  However, we note that the matrix rows for the constraints $\vex(e) \geq 0$ appear also in the rank inequality constraints for sets $U=\{u,v\}$ with $uv \in E$. Thus, for the purpose of characterizing circuits or non-circuits of (\ref{MWF}), it suffices to consider only the constraints from (\ref{Rank}).  In Sections \ref{sec:simplerank} and \ref{sec:mixedsign}, we characterize certain 0/1 circuits of (\ref{Rank}). We also determine some quite restrictive conditions that must be satisfied in order for a 0/1 vector to \textit{not} be a circuit.  These restrictive conditions imply that, in some sense, `most' 0/1 vectors are circuits, thus making a complete, concise characterization of the 0/1 circuits difficult.

It follows immediately both from the general lower bound of \cite{deprs-26} and from Theorem \ref{thm:induce_imbalance} that (\ref{Rank}), as well as any system of constraints containing the constraints of (\ref{Rank}), has at least exponential circuit imbalance.

\begin{corollary}\label{cor:rank_imbalance}
    There exists a graph $G$ such that, if $B$ is the inequality constraint matrix for (\ref{Rank}) or (\ref{MWF}) over $G$, then  $\kappa(0, B)\in \Omega(2^{|E|})$. 
\end{corollary}

We already obtain two known types of 0/1 circuits of (\ref{Rank}) from the connection to matroids (and polymatroids): standard unit vectors and vectors that are the difference of two standard unit vectors are circuits of (\ref{Rank}) \cite{t-84}. They correspond to the addition of an edge to a forest, or the simultaneous addition of and removal of an edge from a forest,  respectively.

For our discussion, we specify types of vectors using the following language. If a vector $\veg$ has at least one positive entry and at least one negative entry, we call $\veg$ a \emph{mixed-sign vector}. In particular, a mixed-sign 0/1 vector has at least one entry $1$ and at least one entry $-1$. Otherwise (i.e., if all entries have the same sign), we call $\veg$ a \emph{uniform-sign vector}.

Comparable to the differences of matchings in \cite{kps-19}, the 0/1 vectors we consider often arise from the differences of forests in $G$, since the extreme-points of (\ref{MWF}) correspond precisely to the forests of $G$. Recall that $\ve{X}(\cdot) \in \R^{|E|}$ denotes the characteristic vector for a collection of edges in $G$. 
For a graph $G$ with forests $F_1$ and $F_2$ in $G$ let $\veg = \ve{X}(F_1) - \ve{X}(F_2)$. Then the entries $\veg(e)$ of $\veg$ are defined as:
\[\veg(e) = 
\begin{cases}
    1 & e \in E(F_1) \text{ and } e \notin E(F_2)\\
    -1 &  e \notin E(F_1) \text{ and } e \in E(F_2)\\
    0 & \text{else}.
\end{cases}
\]

As in Corollary \ref{cor:rank_imbalance}, we let $B$ denote the inequality constraint matrix of (\ref{Rank}) in all that follows.
Using this notation, circuits $\veg$ of (\ref{Rank}) exhibit support-minimality with respect to $B$.
Recall that a vector is support-minimal with respect to a matrix $B$ if there does not exists a nonzero $\vey \neq \veg \in \R^n$ such that $\supp{(B\vey)} \subsetneq \supp{(B\veg)}$. That is, $\veg$ is \textit{not} a circuit if and only if there exists a $\vey \neq \ve0$ such that $\sum_{e\in E(G[U])} \vey(e) = 0$ for all $U$ where $\sum_{e\in E(G[U])} \veg(e) = 0$, and there exists some $U'\subseteq V$ such that $\sum_{e\in E(G[U'])} \vey(e) = 0$ while $\sum_{e\in E(G[U'])} \veg(e) \neq 0$.

\subsection{Simple Uniform and Mixed-Sign Vectors}\label{sec:simplerank}

We begin with a partial characterization of the 0/1 circuits of (\ref{Rank}). In particular, we prove that a variety of differences of two forests, i.e., $\ve{X}(F_1) - \ve{X}(F_2)$, are circuits. As we will see, the fact that (\ref{Rank}) has an exponential number of constraints poses a major challenge in this task, but it also gives us some tools to this end.

We begin with a simple observation: the only uniform-sign vectors that are circuits are standard unit vectors. 

\begin{lemma}\label{lem:unitvector}
Let $\veg$ be a uniform-sign vector. Then $\veg$ is a circuit if and only if it is a standard unit vector.
\end{lemma}
\begin{proof}
    Let $\veg_{e_1}$ be a standard unit vector with nonzero entry for $e_1\in E$, and note that it has coprime integer components. 
    Then $\supp(B\veg_{e_1})$ contains all sets $U\subseteq V$ with $e_1\in E(G[U])$. 
    Consider any vector $\vey\neq\veo$ with $\supp(B\vey)\subseteq\supp(B\veg_{e_1})$. Note that this implies $\supp(\vey)\subseteq\supp(\veg_{e_1})=\{e_1\}$. Then for all $U$, $\sum_{e \in E(G[U])}\vey(e) = 0$ if an only if $\sum_{e \in E(G[U])}\veg_{e_1}(e) = 0$. Thus, $\supp(B\vey)=\supp(B\veg_{e_1})$, and so $\veg_{e_1}$ is a circuit.
    
    Conversely, let $\veg$ be a uniform-sign vector with at least two nonzero entries and let $\vey$ be a standard unit vector corresponding to one of these nonzero entries, denoted $e_2$. 
    For any $U\subseteq V$, $\sum_{e \in E(G[U])}\vey(e) \neq 0$ if and only if $e_2\in E(G[U])$. Further, for any $U\subseteq V$ with $e_2\in E(G[U])$, we also have $\sum_{e \in E(G[U])}\veg(e) \neq 0$, as $\veg$ is a uniform-sign vector. Thus $\supp(B\vey)\subseteq \supp(B\veg)$ and, in fact, $\supp(B\vey)\subsetneq \supp(B\veg)$, as $\supp(\vey) \subsetneq \supp(\veg)$.  This shows that $\veg$ is not a circuit.
\end{proof}

Due to Lemma \ref{lem:unitvector}, it suffices to consider mixed-sign vectors $\veg$ for the remainder of this section. 
We call both $U\subseteq V$ and its corresponding induced subgraph $G[U]$ a \emph{balanced set} or \emph{balanced subgraph} (with respect to $\veg$) if $\sum_{e \in E(G[U])}\veg(e) = 0$, i.e., if the row corresponding to $U$ in $B$ is not in $\supp(B\veg)$. Otherwise, we call the set \emph{imbalanced} (with respect to $\veg$). When $\veg$ is clear from the context, we will leave out the `with respect to' phrase and simply refer to a balanced or imbalanced set. 
We say that an edge $e$ is contained in a balanced (or imbalanced) set $U$ if $e \in E(G[U])$. Finally, we define a {\em balanced pair} $(e,f)$ with respect to $\veg$ as a pair of edges $e,f \in E$ such that there exists a balanced subgraph $G[U]$ with respect to $\veg$ with $\{e,f\} = E(G[U])\cap\supp(\veg)$. Note then that $\veg(e)=-\veg(f)$ for a balanced pair.

Circuits are easily shown to be unique with respect to $\supp(B\veg)$ \cite{f-14,g-75}, but do not have to be unique with respect to $\supp(\veg)$ itself, i.e., the corresponding edge set. Still, the consideration of vectors $\vey$ that satisfy $\supp(\vey)\subsetneq \supp(\veg)$ will prove to be a valuable tool. We can form such a vector $\vey$ through the \textit{drop} of a vector: 

\begin{definition}\label{def:drop2}
    Let $\veg$ be a vector. We say that {\bf a vector $\vez$ with $\supp(\vez)\subseteq\supp(\veg)$ is dropped} from $\veg$ when $\vez$ is subtracted from $\veg$ to obtain vector $\vey=\veg-\vez$ such that a new zero entry in $\vey$ is created, i.e., $\supp(\vey)\subsetneq\supp(\veg)$. 
\end{definition}

Due to $\supp(\vey)\subsetneq\supp(\veg)$, the drop of vector $\vez$ from $\veg$ is directly associated with the {\em drop} of an edge or edge set to arrive at $\vey$. We introduce some terminology to directly refer to and work with these edge sets, without a specification of all entries of vector $\vez$. 

\begin{definition}\label{def:drop1}
    Let $\veg$ be a vector.  We say that {\bf an edge $e$ or edge set $F$ is dropped} from $\veg$ (or $\supp(\veg)$) to obtain a vector $\vey$ when $\{e\} = \supp(\veg)\setminus\supp(\vey)$ or when $F = \supp(\veg)\setminus\supp(\vey)$, respectively.
\end{definition}

Note that the drop of a vector $\vez$ from $\veg$ to arrive at $\vey$ does not require $\supp(B\vey)\subseteq \supp(B\veg)$. Further, the drop of an edge $e$ or edge set $F$ from $\supp(\veg)$  does not imply that the edges that remain in the support have not changed values. More precisely, for $\vey$ in Definition \ref{def:drop1} and $\vey = \veg - \vez$, the entries of $\vez$ in $\supp(\vey)\cap \supp(\veg)$ may be nonzero, unless stated otherwise or clear from the context. That is, specification of a vector $\vez$ that is dropped uniquely determines the edge set $F$ that is dropped; however, for a given $F$, there are many vectors $\vez$ such that dropping $\vez$ corresponds to dropping $F$. In most of our arguments, specification of edge set $F$ will imply exactly what happens to the entries in $\supp(\veg)\backslash F$, or it will be clear from the context.

Next, we observe some special circumstances under which a single edge can be dropped from a mixed-sign vector $\veg$ to acquire a vector $\vey$ with $\supp(B\vey)\subsetneq\supp(B\veg)$. For a simple wording, we say that such a $\vey$ has \emph{reduced support}; likewise, we may refer to \emph{reducing the support} to mean modifying a vector such that its support with respect to $B$ is reduced. We begin by connecting the concept of dropping a standard unit vector (or the corresponding edge) with the concept of balanced sets. 

\begin{lemma}\label{lem:singleedgedrop}
Let $\veg$ be a mixed-sign vector and let $e\in \supp(\veg)$. If the (scaled) standard unit vector $\lambda \veg_e$ corresponding to edge $e$ (and some $\lambda \geq 0$) 
can be dropped from $\veg$ to arrive at a vector $\vey = \veg - \lambda\veg_e$ with $\supp(B\vey)\subsetneq \supp(B\veg)$, then the edge $e$ is not in a balanced set. Conversely, if an edge $e$ is not in a balanced set, it can be dropped to reduce the support.
\end{lemma}

\begin{proof}

    Let $\veg$ and $\vey$ be as stated, and let $e_1$ take the role of $e$ in the statement. 
    First, suppose for the sake of a contradiction that there exists a balanced set $U \subseteq V$ with respect to $\veg$ which contains $e_1$.
    Then we have $\sum_{e \in E(G[U])}\veg(e) = 0$, but $\sum_{e \in E(G[U])}\vey(e) \neq 0$. Thus, $\supp(B\vey)$ is not a subset of $\supp(B\veg)$, a contradiction. This proves the first claim.
    
    Now, suppose that $e_1$ is not in any balanced sets with respect to $\veg$. For all balanced sets $U \subseteq V$, $\sum_{e \in E(G[U])}\veg(e) = \sum_{e \in E(G[U])}\vey(e) = 0$, as they do not contain $e_1$. Thus, $\supp(B\vey) \subseteq \supp(B\veg)$ and, in fact, $\supp(B\vey) \subsetneq \supp(B\veg)$ because $\supp(\vey) \subsetneq \supp(\veg)$.
\end{proof}

For a mixed-sign circuit $\veg$, Lemma \ref{lem:singleedgedrop} in particular implies that all edges in its support lie in \emph{some} balanced set. To set proper expectations, note that the number of vertex subsets $U$ of $V$ is exponential in $|V|$, and thus there are many `options' for this set. When restricted to 0/1 vectors, this suggests that the property of an edge \emph{not} being in a balanced set for some $\veg$ is quite restrictive. This can be validated already by considering vertex subsets $U$ of sizes up to $|U|=4$, as follows:

Let $\veg$ be a mixed-sign vector. There exists a \emph{$+$-edge $e$} (i.e., an edge  with $\veg(e)>0$), and there exists a \emph{$-$-edge f} (i.e., an edge  with $\veg(f)<0$). If $\veg$ is a 0/1 vector, this implies $\veg(e)=1$ and $\veg(f)=-1$. Now, consider the subgraph $G[U]$ induced by the vertices of $e$ and $f$. Suppose that $U$ is an imbalanced set. If $e$ and $f$ are incident, then $E(G[U])\cap\supp(\veg)$ must be a triangle; if $e$ and $f$ are not incident, then in $E(G[U])\cap\supp(\veg)$ there must be an edge between the endpoints of $e$ and $f$. For an edge $e$ to not lie in {\em any} balanced set, it must satisfy this (triangle or connectedness) condition for {\em all} other edges $f$ of opposite sign; this gives a restrictive property overall.

The above example on the importance of incidence and connectedness of edge pairs can serve as a starting point for studying a more general question. Recall that 0/1 mixed-sign vectors which contain one positive entry and one negative entry are circuits \cite{t-84}. This type of circuit corresponds to swapping a pair of edges when moving between two extreme-points of (\ref{MWF}). One may then ask how many such swaps can be performed in a single circuit step. First, let us consider the consequence of combining two circuits whose supports are disconnected from each other.

As a motivating example, let $\veg$ be a 0/1 circuit such that $\veg(e_1)  = 1$, $\veg(e_2) = -1$, and $\veg(e) = 0$ for all $e \in E \setminus\{e_1, e_2\}$. If $\vex = \veg + \veg_{e_3}$ for $e_3 \in E$ such that $e_3$ does not share vertices with $e_1$ and $e_2$
then $\vex$ is a circuit, as dropping (only) $e_3$ causes the balanced pair $(e_2, e_3)$ to become imbalanced. To retain the balanced pair $(e_2,e_3)$, $e_2$ must also be dropped, but this in turn forces $(e_1, e_2)$ to become imbalanced. Thus, a drop of any edge would force the drop of all edges if the result needs to have reduced support (and in particular, retain all balanced pairs).

In the following proofs, we will rely on Lemma \ref{lem:uniquewithinsupport2} whose proof is delayed until Section \ref{sec:mixedsign}.  It says that if a vector $\veg\neq 0$ is not a circuit of (\ref{Rank}), then not only does there exist a circuit $\vey$ with $\supp(B\vey)\subsetneq\supp(B\veg)$, but in fact there exists such a circuit $\vey$ with the additional property that $\supp(\vey)\subsetneq\supp(\veg)$. 

\begin{lemma}\label{lem:discon_circs}
Let $\veg$ and $\veh$ be nonzero 0/1 vectors such that $V[\supp(\veg)]\\ \cap V[\supp(\veh)] = \emptyset$ (where $V[\supp(\cdot)]$ corresponds to the vertex set induced by the edge set of the support), and such that one of the following holds: 
\begin{enumerate}
    \item $\veg$ is a uniform-sign vector, and $\veh$ is either a mixed-sign circuit of (\ref{Rank}) or a uniform-sign vector that is the opposite sign of $\veg$ or,
    \item $\veg$ and $\veh$ are both  mixed-sign circuits of (\ref{Rank}).
\end{enumerate}
Then $\vez = \veg + \veh$  is also a circuit of (\ref{Rank}).
\end{lemma}

\begin{proof}
    Suppose for the sake of a contradiction that there exists a circuit $\vey\neq\veo$ with $\supp(B\vey)\subsetneq\supp(B\vez)$.  This implies that if $(e,f)$ is a balanced pair with respect to $\vez$, then we have that $\vey(e)+\vey(f)=0$.  We will leverage this to show that $\vey$ is a scalar multiple of $\vez$, contradicting either that $\vey\neq\veo$ or $\supp(B\vey)\subsetneq\supp(B\vez)$.   By the later Lemma \ref{lem:uniquewithinsupport2}, we may assume further that $\supp(\vey)\subsetneq\supp(\vez)$. 

    \textbf{Case 1:} First, suppose $\veg$ and $\veh$ are both uniform-sign vectors of opposite sign.  Then for any edge $ab\in\supp(\veg)$ and edge $cd\in\supp(\veh)$, $\supp(\vez)\cap E[\{a,b,c,d\}]=\{ab,cd\}$, and since $\vez(ab)=-\vez(cd)$,we have that $(ab,cd)$ is a balanced pair with respect to $\vez$.  That is, for every $+$-edge $e$ and $-$-edge $f$ in $\supp(\vez)$, we have $\vey(e) = -\vey(f)$.  Since, by definition, $\vez(e)=-\vez(f)$ for all such pairs of edges, we have that $\vey$ is a scalar multiple of $\vez$, as desired.

    \textbf{Case 2:} Now suppose $\veg$ is a uniform-sign vector and $\veh$ is a mixed-sign circuit.  Assume without loss of generality that $\veg$ has only $-$-edges.  Then for all edges $e \in \supp(\veg)$ and all $+$-edges $f \in \supp(\veh)$, $(e,f)$ is a balanced pair in $\vez$, and thus $\vey(e)=-\vey(f)$.  

    First, suppose for the sake of a contradiction that $\vey(e_1)=0$ for some $e_1\in\supp(\veg)$. As above, this implies that for all $+$-edges $f\in\supp(\veh)$, $\vey(f)=0$, and therefore also that for all edges $e' \neq e_1\in\supp(\veg)$, $\vey(e')=0$. Thus, $\vey(e) = 0$ for all $e \in \supp(\veg)$, and $\vey$ can only contain nonzero entries which correspond to $-$-edges in $\supp(\veh)$. This combined with the fact that $\supp(B\vey)\subsetneq \supp(B\vez)$ implies that in fact $\supp(B\vey)\subsetneq\supp(B\veh)$, contradicting that $\veh$ is a circuit. 
    
    Let now $\vey(e) < 0$ for all $e \in \supp(\veg)$. Since $\veg$ and $\veh$ are both 0/1 vectors and $\veh$ is mixed-sign, and since all edges $e \in \supp(\veg)$ and all $+$-edges $f \in \supp(\veh)$ form a balanced pair $(e,f)$ in $\vez$, all nonzero entries of $\vey$ must have the same absolute value. Thus, without loss of generality, we may assume $\vey(e)=-1$ for all $e\in\supp(\veg)$, and therefore $\vey(f)=1$ for all $+$-edges $f\in\supp(\veh)$. Then, since $\supp(\vey)\subsetneq\supp(\vez)$, we have that $\vey(e)=0$ for some $-$-edge $e \in \supp(\veh)$. Consider the vector $\vey'$ obtained from $\vey$ by restricting to the entries of $\supp(\veh)$.  Since $V[\supp(\veg)]\cap V[\supp(\veh)]=\emptyset$, we have that for all sets $S\subseteq \supp(\veh)$ which are balanced with respect to $\veh$, $S$ is also balanced with respect to $\vez$.  Since $\supp(B\vey)\subsetneq\supp(B\vez)$, we have that any such set is balanced with respect to $\vey$, and therefore also $\vey'$.  Moreover, $\supp(\vey')\subsetneq\supp(\veh)$.  Together, these observations imply that $\supp(B\vey')\subsetneq\supp(B\veh)$, contradicting that $\veh$ is a circuit.

    \textbf{Case 3:} Now suppose $\veg$ and $\veh$ are both mixed-sign circuits.  
    Since $\supp(\vey)\subsetneq\supp(\vez)$, we may assume without loss of generality that there exists an edge in $\supp(\veh)\backslash\supp(\vey)$.   As in Case 2, 
    the vector $\vey'$ obtained by restricting $\vey$ to $\supp(\veh)$ provides a contradiction to the fact that $\veh$ is a circuit.
\end{proof}

In particular, Lemma \ref{lem:discon_circs} implies that combining 0/1 circuits whose edge sets do not share vertices will create a new circuit, provided their combination does not create a uniform-sign vector. Because of this, it becomes interesting to identify simple, {\em connected} edge sets that form circuits. We do so by looking at an extension of the mixed-sign circuit in \cite{bb-many-24,t-84}, i.e., a single swap of adjacent edges. We show that these adjacent swaps can be extended to alternating paths and even length alternating cycles (i.e., alternating between $+$-edges and $-$-edges) of length at least 4.

\begin{lemma}\label{lem:alts}
Let $\veg$ be a mixed-sign 0/1 vector. If $\supp(\veg)$ corresponds to an alternating path or an even length alternating cycle of length at least 4, then $\veg$ is a circuit. 
\end{lemma}

\begin{proof}
Suppose, to a contradiction, that $\veg$ is a mixed-sign vector as described in the statement and is not a circuit. Then there exists a nonzero $\vey \in \R^{|E|}$ such that $\supp(B\vey) \subsetneq \supp(B\veg)$. This implies that if $(e,f)$ is a balanced pair with respect to $\veg$, then we have that $\vey(e) + \vey(f) = 0$. 
Due to the alternating pattern of the edges in $\supp(\veg)$, every $+$-edge $e \in \supp(\veg)$ is in a balanced pair with every $-$-edge $f \in \supp(\veg)$. Thus, from $\vey(e) + \vey(f) = 0$, we have that $\vey(e) = -\vey(f)$ for all opposite sign edges. Therefore, $\vey$ is a scalar multiple of $\veg$, a contradiction.
\end{proof}

In the proof of Lemma \ref{lem:alts}, we use that every $+$-edge $e$ is contained in a balanced pair $(e,f)$ with all $-$-edges $f$. This is rather restrictive as it requires all edges of different signs to either be incident to each other or to have at least two edges in the support separating their vertices. Relaxing this restriction, i.e., having that there possibly exists at least one $-$-edge $f$ that is not in a balanced pair with a $+$-edge $e$, we expand the known types of 0/1 circuits to different alternating structures. The first of these is what we call {\em rooted alternating cycles}, i.e., odd length alternating cycles. 

\begin{definition}[Rooted Alternating Cycle]\label{def:rooted}
    Given a graph $G$ and two forests $F_1$ and $F_2$ in $G$, $\veg = \ve{X}(F_1)-\ve{X}(F_2)$ corresponds to a {\bf rooted alternating cycle} if $\supp(\veg)$ corresponds to an alternating cycle of odd length in $G$, i.e., there is a single node along the cycle, called the root node, which is incident to two edges from the same forest.
\end{definition}

Figure \ref{fig:rooted_cyc_ex} provides an example of a rooted alternating cycle. We show that these cycles are circuits, provided they have a length of at least 5 to avoid the triangle case discussed earlier.

\begin{lemma}\label{lem:root_cyc}
Let $\veg$ be a mixed-sign 0/1 vector. If $\supp(\veg)$  corresponds to a rooted (odd length) alternating cycle of length at least 5, then $\veg$ is a circuit.
\end{lemma}

\begin{proof}
Suppose, to a contradiction, that $\veg$ is a mixed-sign vector as in the statement and not a circuit. Then there exists a nonzero $\vey \in \R ^{|E|}$ such that $\supp(B\vey) \subsetneq \supp(B\veg)$. This implies that if $(e,f)$ is a balanced pair with respect to $\veg$, then we have that $\vey(e) + \vey(f) = 0$. We aim to show that for {\em any} $+$-edge $e$ and $-$-edge $f$, $\vey(e) = -\vey(f)$ (even if $e$ and $f$ are not in a balanced pair), i.e., $\vey$ is a scalar multiple of $\veg$ and $\supp(B\vey) = \supp(B\veg)$. 

Without loss of generality, assume the edges $e_0,e_1 \in \supp(\veg)$ connected to the root node are $+$-edges.
Note that every $+$-edge $e \neq e_0,e_1$ forms a balanced pair with every $-$-edge $f$. Therefore, we need only show that $e_0, e_1$ satisfy $\vey(e_0) = \vey(e_1) = -\vey(f)$. Due to the lower bound of $5$ on the length of the cycle, $e_0$ and $e_1$ are contained in at least one balanced pair with the respective $-$-edges, $f_0$ and $f_1$ which are incident to them, i.e., $(e_0,f_0)$ and $(e_1,f_1)$ are balanced pairs. 
Then if we let $e'$ be a $+$-edge distinct from $e_0$ and $e_1$, we have that $e'$ is in a balanced pair with both $f_0$ and $f_1$.
Thus, we have $\vey(e_0) = -\vey(f_0) = \vey(e') = -\vey(f_1) = \vey(e_1)$.
Therefore, $\vey$ is a scalar multiple of $\veg$, a contradiction. 
\end{proof}

\begin{figure}[htb]
        \centering
        \begin{tikzpicture}
            \tikzset{vertex/.style = {shape=circle,draw,minimum size=2em}}
            \node[vertex] (4) at (1,-0.5) {4};
            \node[vertex] (3) at (3,-0.5) {3};
            \node[vertex] (5) at (0,0.5) {5};
            \node[vertex] (1) at (2,0.5) {1};
            \node[vertex] (2) at (4,0.5) {2};
            \draw[thick,dashed] (5)--(1) {};
            \draw[thick,dashed] (1)--(2) {};
            \draw[thick] (2)--(3) {};
            \draw[thick,dashed] (3)--(4) {};
            \draw[thick] (4)--(5) {};
            \end{tikzpicture}
            \caption{A rooted alternating cycle with root node 1. Edges of the same line pattern belong to the same forest.}
            \label{fig:rooted_cyc_ex}
    \end{figure}
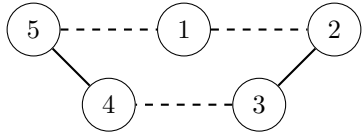

It is possible to further extend the arguments to what we call {\em pseudo-alternating paths} and {\em pseudo-alternating cycles} wherein single edges of a given sign are replaced with paths of that sign. As we do not require this extension as a tool for the later results, we provide this discussion in Appendix \ref{app:pseudo}.

\subsection{General Mixed-Sign Vectors}\label{sec:mixedsign}

Next, we establish more general properties distinguishing 0/1 circuits from non-circuits. In Lemma \ref{lem:discon_circs}, we saw that in many situations, vectors whose supports induce disconnected graphs are circuits. In turn, non-circuits are found more easily when examining vectors whose supports induce connected graphs. As we will see in Theorem \ref{thm:notacircuit}, we are able to significantly add to this observation: either there exists a standard unit vector that can be used to disprove that a given 0/1 vector is a circuit, or the vector induces an edge set where either all $+$-edges or all $-$-edges must lie in a unique connected component with low diameter. 

We can simplify our discussion by proving that non-circuits $\veg$ can be certified to not be circuits by comparing to vectors $\vey$ with $\supp(\vey) \subsetneq \supp(\veg)$. 

\begin{lemma}\label{lem:uniquewithinsupport2}
Let $\veg \neq \ve0$ be a non-circuit of (\ref{Rank}). Then there exists a circuit $\vey$ of (\ref{Rank}) with $\supp(B\vey)\subsetneq\supp(B\veg)$ and $\supp(\vey) \subsetneq \supp(\veg)$.
\end{lemma}

\begin{proof}
    Let $\veg \neq \ve0$ be a non-circuit of (\ref{Rank}). As $\veg$ is a non-circuit, there exists a circuit $\vey$ of (\ref{Rank}) with $\supp(B\vey)\subsetneq\supp(B\veg)$. If $\vey$ satisfies $\supp(\vey)\subsetneq\supp(\veg)$, we are done. Thus, let us assume that $\vey$ satisfies $\supp(\vey)=\supp(\veg)$.

Let $B_0$ denote the row submatrix of $B$ consisting of exactly those rows $i$ with $(B\veg)_i=0$, i.e., the rows for balanced sets with respect to $\veg$. As $\supp(B\vey)\subsetneq\supp(B\veg)$, we have $\lambda \vey \neq \veg$ and we know that 
$B_0(\veg + \lambda \vey) = 0$ for all $\lambda \in \mathbb{R}$. As $\vey\neq \veg$, there exists an edge $e \in \supp(\vey)=\supp(\veg)$ such that $\veg(e) = \lambda' \vey(e) \neq 0$ for some nonzero $\lambda'\neq 1$. Since $\vey\neq -\veg$, the vector $\vez=\veg-\lambda'\vey$ satisfies $\vez\neq \ve0$, and by construction it satisfies $B_0 \vez = 0$ and $\vez(e) = 0\neq\vey(e)$, and thus $\supp(\vez)\subsetneq\supp(\vey)=\supp(\veg)$ and $\supp(B\vez)\subsetneq\supp(B\veg)$. 
If $\vez$ is a circuit, we are done. If not, then there exists a circuit $\vey'$ with $\supp(B\vey')\subsetneq\supp(B\vez)\subsetneq\supp(B\veg)$, in which case $\supp(\vey')\subseteq\supp(\vez)\subsetneq\supp(\veg)$, as desired.
\end{proof}

Recall that, for a vector $\veg$, a vector $\vey$ with $\supp(\vey) \subsetneq \supp(\veg)$ can be constructed through the drop of a vector or edge set; c.f., Definitions \ref{def:drop2} and \ref{def:drop1}. First, let us take a closer look at a situation where there exists a uniform-sign vector $\vez$ that can be dropped from a vector $\veg$ to obtain a vector $\vey = \veg - \vez$ of reduced support, i.e., $\supp(B\vey) \subsetneq \supp(B\veg)$. 
We show that if $\vey$ has reduced support and $\vez$ is a uniform-sign vector then there exists a standard unit vector $\vez'$ such that $\vey' =\veg-\vez'$ satisfies $\supp(B\vey') \subsetneq \supp(B\veg)$.

\begin{lemma}\label{lem:singleedge}
Let $\veg$ be a (mixed-sign) vector. If a uniform-sign vector $\vez$ can be dropped from $\veg$ to give a vector $\vey=\veg-\vez$ with $\supp(B\vey)\subsetneq \supp(B\veg)$, then there exists an edge that can be dropped from $\veg$ to reduce the support.
\end{lemma}

\begin{proof}
    We consider the vector $\vey= \veg-\vez$ resulting from a drop of $\vez$ from $\veg$. Recall that this drop guarantees that there exists an edge $f\in\supp(\veg)\setminus\supp(\vey)$.
    Suppose that there exists a balanced set $U \subseteq V$ with respect to $\veg$ with $f \in E(G[U])$.
    Then we have $\sum_{e \in E(G[U])}\veg(e) = 0$, but since $\vez$ is a uniform-sign vector, $\sum_{e \in E(G[U])}\vey(e) \neq 0$. Thus, $\supp(B\vey)$ is not a subset of $\supp(B\veg)$, a contradiction. This implies that all sets $U'\subseteq V$ containing $f$ are imbalanced with respect to $\veg$. By Lemma \ref{lem:singleedgedrop}, $f$ can be dropped  from $\veg$ to arrive at a reduced support. This proves the claim.
\end{proof}

Lemma \ref{lem:singleedge} means that, if it is possible to reduce the support $\supp(B\veg)$ of vector $\veg$ by dropping a uniform-sign vector, then there exists just a single edge $f$ which is not in any balanced set, and it suffices to drop $f$ to reduce the support. More precisely, it suffices to drop a scaling of the standard unit vector $z_f$ to arrive at $\vey = \veg - \lambda\vez_f$, as doing so gives $\sum_{e \in E(G[U])}\veg(e) = \sum_{e \in E(G[U])}\vey(e)=0$ for all balanced sets with respect to $\veg$.

Let us exchange the roles of $\vey$ and $\vez_f$. Let $\veg$ be a (mixed-sign) vector and suppose dropping $\lambda \vez_f$ from $\veg$ gives $\vey=\veg-\lambda \vez_f$ with $\supp(B\vey)\subsetneq \supp(B\veg)$. Observe that $\lambda \vez_f=\veg-\vey$, i.e., dropping the (mixed-sign) vector $\vey$ from $\veg$ gives a (scaled) standard unit vector $\vez_f$. We obtain $\supp(B\vez_f)\subsetneq \supp(B\veg)$, as clearly $\supp(\vez_f)\subsetneq \supp(\veg)$ and any balanced sets with respect to $\veg$ do not contain $f$. This proves the following corollary to Lemma \ref{lem:singleedgedrop}. 

\begin{corollary}\label{cor:unitvectorcontainment}
Let $\veg$ be a (mixed-sign) vector. The (scaled) standard unit vector $\lambda \veg_e$ corresponding to an edge $e$ (and some $\lambda \geq 0$) can be dropped from $\veg$ to reduce the support if and only if the standard unit vector $\veg_e$ satisfies $\supp(B\veg_e)\subsetneq \supp(B\veg)$.
\end{corollary}

Recall that standard unit vectors are circuits; c.f., Lemma \ref{lem:unitvector}. Thus $\veg_e$ in the above statement is a circuit. Combining observations so far, we can characterize those situations in which there exists a uniform-sign vector $\vez$ that can be dropped from a vector $\veg$ to obtain a vector with reduced support as being exactly the same as those situations in which there exists a unit-vector circuit $\veg_e$ with inclusion-wise smaller support. In other words, consideration of unit-vector circuits suffices to disprove that such vectors are circuits. 

For the remainder of the discussion, we consider dropping a mixed-sign vector $\vez$ from a mixed-sign vector $\veg$ to arrive at a vector $\vey$. If $\vey$ is a uniform-sign vector, then by Lemma \ref{lem:unitvector} it is a circuit if and only if it is a standard unit vector. In turn, Corollary \ref{cor:unitvectorcontainment} tells us that it suffices to drop a (scaled) standard unit vector $\vez'$ from $\veg$ to reduce the support. Thus, we may restrict to the case where $\vey$ is a mixed-sign vector as well.

We are interested in understanding the structure of 0/1 vectors in this remaining setting. Recall the discussion after Lemma \ref{lem:singleedgedrop}, wherein we observed a surprisingly restrictive property for 0/1 vectors to not be circuits: circuits cannot have an edge $e$ that only lies in imbalanced sets; otherwise, $e$ could be dropped to reduce the support. Our goal is to generalize this property: instead of a drop of a single edge $e$, we consider the drop of a set of edges.

First, let $\veg$ be a mixed-sign 0/1 vector and let $S\subseteq E$ such that $\veg$ has both $+$-edges and $-$-edges in both $S$ and $E\backslash S$; edge sets with this property will be called \textit{mixed-sign sets} in the following. We consider a drop of $S$ to reduce the support. We show that this is only possible if $\supp(\veg)$ corresponds to the union of two low-diameter graphs. 

\begin{theorem}\label{thm:diameter}
Let $\veg$ be a mixed-sign 0/1 vector and let $S\subseteq E$ be such that $\veg$ has both $+$-edges and $-$-edges in both $S$ and $E\backslash S$. Let $\vey$ be obtained from $\veg$ by dropping the edge set $S$, and suppose that $\vey$ is mixed-sign and $\supp(B\vey)\subsetneq \supp(B\veg)$. Then $\supp(\veg)$ is the disjoint union of two edge sets which each induce a connected graph of diameter at most $5$.
\end{theorem}

Before we prove this theorem, we illustrate an example in Figure \ref{fig:mixedsignexample}. The graph consists of two partitions, $S = \{(1,2), (4,5), (5,6), (1,7)\}$ and $E\setminus S = \{(2,3), (3,4),$ $(3,6), (3,7)\}$, that have 0/1 edges of mixed signs. 
Observe that all edges lie in some balanced set due to having a neighbor of the opposite sign and that the structure is triangle-free. The edges which correspond to $S$ can be dropped to reduce the support (and, in fact, obtain a 0/1 circuit). Dropping all edges \textit{except} these 4 edges also gives a 0/1 circuit.

  \begin{figure}[ht]
\centering        
\begin{subfigure}{0.45\textwidth}
\centering
\begin{tikzpicture}
\tikzset{vertex/.style = {shape=circle,draw,minimum size=2em}}
            \node[vertex] (1) at (-2,0) {1};
            \node[vertex] (3) at (0,0) {3};
            \node[vertex] (4) at (1,1) {4};
            \node[vertex] (6) at (1,-1) {6};
            \node[vertex] (7) at (-1,-1) {7};
            \node[vertex] (2) at (-1,1) {2};
             \node[vertex] (5) at (2,0) {5};

            \draw[thick,dashed] (1)--(2) {};
            \draw[thick,dashed] (2)--(3) {};
            \draw[thick,dashed] (3)--(4) {};
            \draw[thick,dashed] (4)--(5) {};
            \draw[thick] (5)--(6) {};
            \draw[thick] (6)--(3) {};
            \draw[thick] (3)--(7) {};
            \draw[thick] (7)--(1) {};
\end{tikzpicture}
 \caption{A graph corresponding to a mixed-sign 0/1 vector.}
     \end{subfigure}
\hfill 
\begin{subfigure}{0.45\textwidth}
\centering
\begin{tikzpicture}
\tikzset{vertex/.style = {shape=circle,draw,minimum size=2em}}
            \node[vertex] (1) at (-2,0) {1};
            \node[vertex] (3) at (0,0) {3};
            \node[vertex] (4) at (1,1) {4};
            \node[vertex] (6) at (1,-1) {6};
            \node[vertex] (7) at (-1,-1) {7};
            \node[vertex] (2) at (-1,1) {2};
             \node[vertex] (5) at (2,0) {5};

            \draw[thick,dashed] (2)--(3) {};
            \draw[thick,dashed] (3)--(4) {};
            \draw[thick] (6)--(3) {};
            \draw[thick] (3)--(7) {};

\end{tikzpicture}
 \caption{A graph corresponding to a mixed-sign 0/1 vector of reduced support.}
     \end{subfigure}

\caption{An example for the drop of a mixed-sign edge set from a mixed-sign 0/1 vector
to reduce the support. Edges of the same line pattern are the same sign.}
\label{fig:mixedsignexample}
\end{figure}

\begin{proof}[Proof of Theorem \ref{thm:diameter}]
Let $\veg$, $S$, and $\vey$ be defined as in the statement. In this proof, for the sake of simplicity we will refer to distance between objects in $G[\supp(\veg)]$ as simply the distance.
Note that since $\supp(B\vey)\subsetneq \supp(B\veg)$, there does not exist any set $U$ which is balanced with respect to $\veg$ but imbalanced with respect to $\vey$. 
Consider any pair of edges $e=uv$ in $S$ and $f=rs$ in $E\setminus S$ such that $\veg(e)$ and $\veg(f)$ are nonzero and of opposite sign. Let $U=\{u,v\}\cup\{r,s\}$.  Note that it is not necessarily the case that $\{u,v\}\cap\{r,s\}=\emptyset$.
If $E[U]=\{e,f\}$, then $U$ is a balanced pair with respect to $\veg$ but not $\vey$, a contradiction. 
Thus there exists at least one additional edge $h\in E[U]$. If $e$ and $f$ are incident, $U$ induces a triangle; if they are not incident, then they are connected via at least one edge.
Thus, for any such pair of edges $e$ and $f$, their vertices are at distance at most $3$ from each other in $G$.

Now let $e=uv$ and $f=rs$ be such that either both are in $S$ or both are in $E\setminus S$, and such that $\veg(e)$ and $\veg(f)$ are nonzero and of the same sign.  Without loss of generality, suppose they are both in $S$ and both positive.
Since $E\setminus S$ contains both $+$-edges and $-$-edges, there is a $-$-edge $h=tw$ in $E\setminus S$, and per the above argument, the largest possible distance between any of $u,v,r,$ or $s$ and either of $t$ or $w$ is 3. It follows that the largest possible distance between any pair of $u,v,r,$ or $s$ is 5.

To conclude the proof, let $R_1$ be the set of $-$-edges of $S$ and $+$-edges of $E\setminus S$, and let $R_2$ be the set of $+$-edges of $S$ and $-$edges of $E\setminus S$.  Then $\supp(\veg)=R_1\cup R_2$, and as above, each of $G[R_1]$ and $G[R_2]$ has graph diameter at most 5, as desired.
\end{proof}

To summarize, Lemma \ref{lem:singleedge} shows that if a 0/1 vector $\veg$ can be certified as a non-circuit by dropping a uniform-sign vector, then it can be certified as a non-circuit by dropping an edge.  Theorem \ref{thm:diameter} covers the case where the certification that $\veg$ is not a circuit requires dropping a mixed-sign vector, under the technical condition that both the set of dropped edges and the set remaining in the support are mixed-sign with respect to $\veg$. In such a case, $\supp(\veg)$ is the disjoint union of two edges sets, both inducing connected graphs of low diameter. 

We conclude our discussion of non-circuits by covering a case not captured by the technical condition in Theorem \ref{thm:diameter}, and by showing that said condition, in fact, implies connectedness of $\supp(\veg)$. The result is an exhaustive description of three situations that can arise when a 0/1 vector $\veg$ is not a circuit. One is a very simple case, and the other two are highly restricted cases.

\begin{theorem}\label{thm:notacircuit}
Let $\veg \neq 0$ be a 0/1 vector which is not a circuit of (\ref{Rank}).
Then either
\begin{enumerate}
    \item there exists $e\in\supp(\veg)$ such that $\supp(B\veg_e)\subsetneq\supp(B\veg)$,
    \item $\supp(\veg)$ induces a connected graph, and is the disjoint union of two edge sets which each induce a connected graph of diameter at most $5$, or
    \item $\supp(\veg)$ is the disjoint union of 
    a mixed-sign edge set and a uniform-sign edge set, with the mixed-sign edge set corresponding to a non-circuit and inducing a connected graph of diameter at most $5$.
\end{enumerate}
\end{theorem}

\begin{proof}
In this proof, for the sake of simplicity we will refer to distance between objects in $G[\supp(\veg)]$ as simply the distance.  Let $\veg$ be a 0/1 non-circuit and let $\vey$ be a vector of reduced support arising from a drop of vector $\vez$, i.e., let $\vey = \veg - \vez$ and $\supp(B\vey)\subsetneq \supp(B\veg)$.

 Case $1$ arises in three situations. First, it arises when $\veg$ is a uniform-sign vector but not a standard unit vector, as in Lemma \ref{lem:unitvector}. Second, it covers the situation where $\vey$ is a uniform-sign vector, again by Lemma \ref{lem:unitvector}. Third, it captures situations in which the difference between $\veg$ and $\vey$ is a uniform-sign vector, by Lemma \ref{lem:singleedge} and Corollary \ref{cor:unitvectorcontainment}. Note that together, these cover all situations in which at least one of $\veg$, $\vez$, or $\vey$ is uniform-sign.

We have to show that cases $2$ and $3$ capture the remaining situations where case $1$ does not apply. Thus, let $\veg$ be a mixed-sign vector and let $\vez$ be a mixed-sign vector that is dropped to obtain a mixed-sign vector $\vey=\veg-\vez$ of reduced support. Let $S$ denote the corresponding edge set that is dropped. Note that we do not immediately satisfy the prerequisites of Theorem \ref{thm:diameter}: It is possible that $S$ or $E\backslash S$ is uniform-sign with respect to $\veg$. As case $1$ does not apply, by Lemma \ref{lem:singleedgedrop} each edge lies in a balanced set with respect to $\veg$. This implies that at least one of $S$ or $E\backslash S$ must be mixed-sign; otherwise some balanced set with respect to $\veg$ is imbalanced with respect to $\vey$.

First, let us assume that $S$ is mixed-sign and $E\backslash S$, without loss of generality, contains only $-$-edges. In the proof of Theorem \ref{thm:diameter}, we showed that the union of $-$-edges in $E\backslash S$ and $+$-edges in $S$ induce a connected graph of diameter at most $5$. The only edges not contained in this component are the $-$-edges in $S$. As the two edge sets are disconnected, by Lemma \ref{lem:discon_circs} we know that the mixed-sign component corresponds to a non-circuit. This situation is captured in case $3$. The argument holds analogously if one exchanges the roles of $S$ and $E\backslash S$ or the roles of $+$-edges and $-$-edges. 

Finally, let both $S$ and $E\backslash S$ be mixed-sign. By Theorem \ref{thm:diameter}, $\supp(\veg)$ is known to be the disjoint union of two edge sets which each induce a connected graph of diameter at most $5$. It remains to prove that $H = G[\supp(\veg)]$ is connected, confirming that case $2$ captures this final situation.

 For the purpose of contradiction, assume $H$ has several components. As $\supp(\veg)$ is the union of two edge sets, each inducing connected graphs, $H$ has two components $C_1,C_2$. In the following, we consider the relation between the set $S\subseteq \supp(\veg)$ and the components $C_1,C_2$. 

First, assume that one of the components is uniform-sign with respect to $\veg$. Without loss of generality, let $C_2$ have only $-$-edges. Note that this implies that $C_1$ has at least one $+$-edge.  
The whole component $C_2$ is then contained in either $S$ or $E\backslash S$ along with all $+$-edges from $C_1$; otherwise a balanced pair of a $+$-edge from $C_1$ and a $-$-edge from $C_2$ with respect to $\veg$ is not one for $\vey$. However, this implies that the other set from $S$ or $E\backslash S$ is uniform-sign with respect to $\veg$, as it can only contain the $-$-edges from $C_1$; 
a contradiction.

Thus, both components $C_1,C_2$ are mixed-sign with respect to $\veg$. Next, observe that if either of $S$ or $E\setminus S$ is contained entirely in one of $C_1$ or $C_2$, some balanced pair with respect to $\veg$ is not one with respect to $\vey$: Without loss of generality, let $S \subseteq C_1$. Then $C_2\subseteq E\backslash S$, and any balanced pair $(e,f)$ with edge $e\in S$ and $f\in C_2$ can be used to validate the claim. This shows that both components $C_1$ and $C_2$ contain edges from both of $S$ and $E\setminus S$.
Let $C^S_1=C_1\cap S, C^{E\backslash S}_1=C_1\cap (E\backslash S)$,  $C^S_2=C_2\cap S,$ and $ C^{E\backslash S}_2=C_2\cap (E\backslash S)$. By the above, all of these sets are non-empty.

Next we argue that each of these four edge sets must be uniform-sign. Without loss of generality, assume $C^S_1 \subseteq S$ is mixed-sign. As $C_2$ is mixed-sign and disconnected from $C_1$, there exist balanced pairs 
$(e,f)$ with edge $e\in C^S_1$ and $f\in C^{E\backslash S}_2$ with respect to $\veg$ that are not balanced pairs with respect to $\vey$. 
For the same reason, it also holds that the edges of $C^S_1$ and $C^{E\backslash S}_2$ share the same sign (w.l.o.g. $+$-edges), and the edges of $C^S_2$ and $C^{E\backslash S}_1$ share the same, other sign (w.l.o.g. $-$-edges). 

Recall that each component $C_1$ and $C_2$ is mixed-sign and thus has pairs of incident $\pm$ edges. These pairs $(e,f)$ of incident edges are split between $+$-edges in $C^S_1$ and $-$-edges in $C^{E\backslash S}_1$ for $C_1$, and $-$-edges in $C^S_2$ and $+$-edges in $C^{E\backslash S}_2$ for $C_2$. The set $U$ of three vertices of such a pair $(e,f)$ induce a triangle in $H$; otherwise $(e,f)$ is a balanced pair with respect to $\veg$ but not $\vey$, a contradiction. Note that the third edge of the triangle lies in the same of the four parts as the other edge in that triangle which shares its sign.
 
 Consider a pair $E[U_1], E[U_2]$ of such mixed-sign triangles induced by vertex sets $U_1$ and $U_2$, one from each component $C_1$ and $C_2$, respectively. It is readily verified that one can always select $5$ vertices $U$ from the $6$ vertices $U_1 \cup U_2$, such that one has a balanced set with respect to $\veg$: $U$ is constructed to induce two $+$- and two $-$edges; the set $U$ can be built by choosing the vertices $U_i$ of one of the triangles and adding the two vertices of an edge $f_U$ of the ``missing'' sign from the other.

We claim that the balanced set $U$ or one of its balanced subsets with respect to $\veg$ is imbalanced with respect to $\vey$. Without loss of generality, let $U_i=U_1$. Note that only one or three of the four edges in $E(U)$ lie in $E\backslash S$:  If $G[U_1]$ has two $+$-edges in $C_1^S$ and one $-$-edge in $C_1^{E\backslash S}$, then the additional $-$-edge $f_U$ lies in $C_2^S$; 
if $G[U_1]$ has one $+$-edge in $C_1^S$ and two $-$-edges in $C_1^{E\backslash S}$, then the additional $+$-edge $f_U$ lies in $C_2^{E\backslash S}$. 
If only one edge from $E(U)$ is in $E\backslash S$, then clearly $U$ is imbalanced with respect to $\vey$. If three edges lie in $E\backslash S$, then either $U$ or some balanced pair $(e,f)$ contained in $E[U]$ with $e \in U_1$ and $f \in U_2$ is imbalanced with respect to $\vey$.
This shows our final contradiction, proving that $\supp(\veg)$ must be connected. \end{proof}

Note that the cases are not exclusive from each other. For example, a uniform-sign vector with two nonzero entries would satisfy both cases 1 and 2. Theorem \ref{thm:notacircuit} tells us that either a standard unit vector can be used to disprove that a non-circuit 0/1 vector is, in fact, not a circuit---a situation that is quite rare, because it would require the existence of an edge that is not in \textit{any} balanced set (recall the discussion after Lemma \ref{lem:singleedgedrop})---or the support of the non-circuit 0/1 vector either is a single mixed-sign component of low diameter (at most $10$) or it contains a unique mixed-sign component of low diameter (at most $5$) which itself corresponds to a non-circuit. 
In particular, this implies that 0/1 vectors where the diameter of this component exceeds $10$ (and that do not have an edge that is not in a balanced set) must be circuits.

\subsection{Constant Circuit Diameters}\label{sec:rankddiameters}

We prove a constant upper and lower bound on the circuit diameter of the polyhedron given by (\ref{MWF}). The arguments build on the insight gained in Sections \ref{sec:simplerank} and \ref{sec:mixedsign} on the relation of circuits, non-circuits, alternating structures, and connectedness. We begin with a general upper bound of $9$ on the circuit diameter. Later, we will reduce this bound for special cases.

\begin{theorem}\label{thm:upperbound9}
Let $G=(V,E)$ be a graph and let $P$ be the polyhedron defined by (\ref{MWF}). Then $P$ has circuit diameter at most $9$.
\end{theorem}

\begin{proof}
If $|V|\leq 4$, then $P$ is known to have circuit diameter at most $6$ \cite{bb-many-24,t-92}. Thus, we may assume $|V|\geq 5$. Further, note that all feasible 0/1 vectors are extreme-points. Our proof strategy is as follows. We show that any extreme-point $\vex$ of $P$ is at most $5$ circuit steps (using 0/1 circuits and step lengths $1$) from $\veo$, and vice versa.  
This gives an upper bound of $10$ on the circuit diameter of $P$. We then explain a refinement that improves the bound from $10$ to $9$.

Let $\vex=\ve{X}(F)$ be an extreme-point of $P$ where $F$ is a forest containing edge $e=uv$ and where $u$ is a leaf. Let $F_{v}:=\{vw \in F: w \neq u\}$, and recall that $\veg_{e}$ denotes the standard unit vector circuit corresponding to $e$. Starting from $\veo$, consider the following sequence of 5 circuit steps:
\[
\begin{array}{c}
     (1)\; +\veg_{e} \quad\quad (2)\; -\veg_{e}+\ve{X}(F_{v}) \quad\quad (3)\; +\veg_{e}\\ 
     (4)\; -\veg_{e}+\ve{X}(F\backslash F_{v}) \quad\quad (5)\; +\veg_{e}.
\end{array}
\]
 Steps 1, 3, and 5 are along the standard unit vector circuit corresponding to edge $e$. Steps 2 and 4 remove edge $e$ again, but add all edges incident or non-incident to $v$ from $F\backslash\{e\}$, respectively. If $F\backslash F_{v}$ or $F_{v}$ are empty, we reduce this sequence to steps 1,2,3 or steps 1,4,5, respectively.

Each intermediate solution has support contained in $F$, and thus this is a feasible walk. The vector $-\veg_{e}+\ve{X}(F_{v})$ in step $2$ is a circuit, as there is only one $-$-edge $e$ and all $+$-edges are incident to the same vertex $v$; each pairing of $e$ and any $+$-edge corresponds to a balanced pair. The vector $-\veg_{e}+\ve{X}(F\backslash F_{v})$ in step 4 is a circuit by point 1 of Lemma \ref{lem:discon_circs}.

Thus, we have verified that there is a circuit walk of at most 5 steps from $\veo$ to $\ve{X}(F)$. The construction is reversible, as all intermediate points are extreme-points; for the reversed walk from $\ve{X}(F)$ to $\veo$, one performs the reversed sequence of opposite steps (negated in sign). By the arbitrary choice of $F$, we obtain a bound of $10$ on the circuit diameter of $P$.

It remains to explain a refinement of this upper bound to $9$. Let $F_1 \neq F_2$ be two nonempty forests. A circuit walk of at most $10$ steps from $F_1$ to $F_2$ as explained above would perform the reversed walk from $\ve{X}(F_1)$ to $\veo$, and then perform the stated walk from $\veo$ to $\ve{X}(F_2)$. Let $e_1 \in F_1$ and $e_2 \in F_2$ be the leaf edges used in the construction.
Since the vector $\veg_{e_2}-\veg_{e_1}$ is a circuit, the last step of the walk from $\veX(F_1)$ to $\veo$ and the first step of the walk from $\veo$ to $\veX(F_2)$ can be replaced by a single step. This gives a total of $9$ steps to walk from $F_1$ to $F_2$, as desired.   
\end{proof}

Next, we show that if the underlying graph is complete, the strategy can be improved to obtain an upper bound of $7$. 

\begin{theorem}\label{thm:upperbound7}
Let $G=(V,E)$ be a complete graph and let $P$ be the polyhedron defined by (\ref{MWF}). Then $P$ has circuit diameter at most $7$.
\end{theorem}

\begin{proof}
As in the proof of Theorem \ref{thm:upperbound9}, recall that all feasible 0/1 vectors of $P$ are extreme-points and that we may assume $|V|\geq 5$. Let $\vex$ and $\vey$ be two extreme points of $P$.  Our proof strategy is as follows: We show that  $\vex$  ($\vey$, respectively) is only at most $2$ circuit steps from an extreme-point $\vex'$ ($\vey'$, respectively) whose support is a path of length $3$. As standard unit vectors $\veg_e$ and differences of standard unit vectors $\veg_e-\veg_f$ are circuits, $\vex'$ and $\vey'$ are connected by a circuit walk of at most $3$ steps. This gives a circuit walk between $\vex$ and $\vey$ of at most $7$ steps. Note that if the support of $\vex$ ($\vey$, respectively) itself has at most $3$ edges, then part of the construction can be skipped, as such an extreme-point is at most $3$ circuit steps from $\vey'$ ($\vex'$, respectively). 

Let $\vex=\ve{X}(F)$ be an extreme-point of $P$ where $F$ is a forest with more than $3$ edges. Any nonempty forest $F$ has at least two leaves $u$ and $v$. Let $f$ and $h$ be the edges of $F$ containing $u$ and $v$, respectively, and let $e=uv$.  Note that $e\in E$ exists since $G$ is a complete graph. Note further that possibly $e\in F$; in this case, we have $e=f=h$. It suffices to exhibit that $\vex$ is at most 3 circuit steps (using 0/1 circuits and step lengths $1$) away from an extreme-point $\vex'$ whose support is a path of length $3$, and vice versa. As in the proof of Theorem \ref{thm:upperbound9}, this sequence of steps only visits extreme-points of $P$ and thus is reversible. 

First, we resolve the case $e \notin F$. Let $\vex'' = \veX(\{e,f,h\})$ and note that $e \notin F$ implies $e \neq f,h$ and $f\neq h$. If $\vex''$ does not correspond to a triangle, then it corresponds to a path of length $3$ and we set $\vex'=\vex''$. This $\vex'$ is an extreme-point of $P$ and $\vex'-\vex$ is a circuit since it corresponds to an edge set of two uniform-sign components of opposite signs, by Lemma \ref{lem:discon_circs}. Thus, one can walk from $\vex$ to $\vex'$ in one step, and vice versa. 

If $\vex''$ instead does correspond to a triangle, then $\vex'' \notin P$ and we construct $\vex'= \veX(\{e,f',h\})$ from $\vex''$ by replacing the edge $f \in F$ incident to leaf $u$ by an edge $f' \notin F$ incident to $u$ and $f'\neq e$. Note that such an edge exists since $|V|\geq 5$. This $\vex'$ corresponds to a path of length $3$ and is an extreme-point of $P$. Further, $\tilde\vex = \veX(F\backslash \{f\} \cup \{f'\})=\vex -\veg_{f}+\veg_{f'}$ is an extreme-point, and one circuit step $-\veg_{f}+\veg_{f'}$ away from $\vex$. Finally, $\vex'-\tilde\vex$ is a circuit since it corresponds to an edge set of two uniform-sign components of opposite sign, again by Lemma \ref{lem:discon_circs}. Thus, one can walk from $\vex$ to $\vex'$ in at most $2$ steps, and vice versa.

Next, we resolve the case $e \in F$, which implies that $e=f=h$ and $e$ constitutes its own component of $F$. We construct $\tilde\vex = \veX(F\backslash \{e\} \cup \{f',h'\})$ from $\vex$ by removing $e$ and adding two other edges $f',h' \notin F$ incident to $u$ and $v$, respectively, and not incident to each other. Again, such edges exist since $|V|\geq 5$. By construction, $\tilde\vex$ corresponds to a forest and thus is an extreme-point of $P$. Further, $\tilde\vex-\vex$ is a circuit by Lemma \ref{lem:alts}, as it corresponds to an alternating path of length $3$.  Let now $\vex'= \veX(\{e,f',h'\})$, and note that it corresponds to a path of length $3$ and is an extreme-point. Then $\vex'-\tilde\vex$ is a circuit by the same argument as above. Thus, in all cases, $\vex$ is at most 2 steps from some extreme-point $\vex'$ whose support is a path of 3 edges. By the discussion preceding this argument, this means that between any two extreme-points $\vex$ and $\vey$ there is a circuit walk of at most 7, as desired.        
\end{proof}

Finally, we devise a lower bound of $3$ on the circuit diameter.  

\begin{theorem}\label{thm:low_diam_bnd}
    Let $G = (V,E)$ be a graph on $|V| \geq 4$ vertices with a spanning forest $F\subseteq E$ of at least 3 edges, and let $P$ be the polyhedron defined by (\ref{MWF}). Then $P$ has circuit diameter at least $3$.
\end{theorem}

\begin{proof}
    The extreme-points of (\ref{MWF}) are feasible 0/1 vectors and correspond to the characteristic vectors of forests in $G$. We consider a walk from $\veo$ to $\ve{X}(F)$, where $F$ is a spanning forest of at least $3$ edges as in the statement. 
    At $\veo$, all feasible circuits are standard unit vectors. Assume the first circuit step taken is $\veg_{e_1}$ for some $e_1 \in E$. 
    The next feasible circuit step taken can either be along another standard unit vector or a mixed-sign vector. Since $F$ has at least 3 edges, one cannot do better than 3 steps if the second step is another standard unit vector. 
    
    If the second step is a mixed-sign vector $\veg$, it has $\veg(e_1) < 0$ and $\veg(e) > 0$ for some number of other edges $e$. Due to $\veX(F)$ being an extreme-point, we may assume that the step along $\veg$ arrives at a new extreme-point as otherwise (at least) a third step is required to arrive at $\veX(F)$. This implies that $\veg$ is a 0/1 vector and the step length is $1$. We will conclude by showing that, in all cases, there exists some edge in $F$ that still has to be added after the second step along $\veg$.
    
    If $e_1 \in F$, at least one more circuit step is required to move to $\veX(F)$ because $e_1$ has to be added; recall it was dropped in the second step. If $e_1\notin F$, then since $F$ is spanning, the vertices of $e_1$ are connected by a path $\mathcal{P}$ in $F$ of at least two edges. Regardless of the length of $\mathcal{P}$, we will show that at least one edge in $\mathcal{P}$ was not added in the second circuit step. If $\mathcal{P}$ has two edges, then it forms a triangle with $e_1$ and it thus is not possible to add both edges in $\mathcal{P}$ by a second circuit step. If $\mathcal{P}$ has at least $3$ edges, then it contains an edge $f$ incident to $e_1$ and an edge $h$ incident to $f$ but not $e_1$. In this case, the second circuit step cannot remove $e_1$ and add both $f$ and $h$: with $\veg(e_1)=-1$ for the only $-$-edge $e_1$ and $\veg(f)=\veg(h)=1$, $h$ is not in any balanced subgraphs with respect to $\veg$; a contradiction to $\veg$ being a circuit by Lemma \ref{lem:singleedgedrop}. This again means that at least one more circuit step is required. Thus, in all cases we have a circuit walk of at least 3 steps. This proves our claim.
\end{proof}

We note that quite restrictive properties are required to make the bound of $3$ in Theorem \ref{thm:low_diam_bnd} tight. In fact, in most situations a circuit walk from $\veo$ to $\ve{X}(F)$ for a spanning forest $F$ requires at least $4$ steps.  Recall that the eccentricity of a vertex is the maximum distance from that vertex to any other vertex and the radius of a connected graph is the minimum eccentricity over all vertices.  To achieve a walk of exactly 3 steps, one would need either at least one component with radius 1 where the spanning tree of that component utilizes all the edges of the vertex whose eccentricity in that component is 1,  or to have a component which is a path of length less than or equal to 3. Assuming $e_1 \in F$ is in one such component, there are no triangles formed or edges in $F$ far enough from $e_1$ which could not be added in step 2. When the spanning forest is built using all edges from a vertex with eccentricity 1 in its component, it means the third and final circuit step is the standard unit vector $e_1$. A vertex with eccentricity 1 in its component is adjacent to all other vertices in the same component, so the second circuit step can add all edges in the forest except the edge which was removed. The same situation plays out when the graph contains a path of length 3 or less as a component since $e_1$ can correspond to the middle edge in the path of length 3. In all other cases, a lower bound on the number of circuit steps is 4 since there would exist a triangle or path/cycle of length at least 4 which contains at least one edge which cannot be added when $e_1$ is removed. 

\section{Conclusion and Final Remarks} 
\label{sec:concl}

In this work, we studied circuits and circuit walks with an interest in their interpretability relative to the underlying application. While the inclusion-minimal support of circuits makes them prime candidates to this end, a high circuit imbalance works directly against interpretability. Even for the highly structured matrices arising in combinatorial optimization, this poses a significant challenge. We exhibited a number of constructions of simple, polynomially sized matrices with rows corresponding to characteristic vectors of ``simple" edge sets, all of which lead to an exponential imbalance. This holds even for very restrictive edges sets, such as all those of paths of length $3$ or cycles of length at most $4$. Similarly simple collections of constraints appear commonly as part of formulations of graph-theoretic problems, and so it becomes natural to restrict the entries of circuits for the purpose of recovering interpretability. We followed such an idea for the study of polytopes coming from two classic problems: vertex coloring and (max weight) forest problems (i.e., graphic matroids). We studied their 0/1 circuits and showed that they suffice for the construction of short walks between the meaningful extreme-points of the underlying application; this point of view closely relates to the field of combinatorial reconfiguration. In fact, for the coloring problem, we saw that the 0/1 circuits generalize classic Kempe swaps and guarantee the ability to transform any coloring into any other---a property that Kempe swaps alone do not have.

Our work leads quite naturally to more general questions. As discussed in Section \ref{sec:intro} and as we have seen in some of this work, the set of just the 0/1 circuits often suffices to find short circuit walks between extreme-points of polyhedra coming from combinatorial optimization.
A natural direction for further investigation is to determine conditions under which this is the case. In general, a polyhedron may not have any 0/1 circuits, so the existence of such circuits is of course necessary and an interesting question in its own right. Without additional knowledge, checking whether a rational polyhedron has a 0/1 circuit is NP-complete. Membership in NP comes from an efficient check whether a given vector is a circuit \cite{bgkls-24}, whereas NP-hardness can be seen through a reduction to the subset sum problem. Consider an instance of subset sum with integers $a_1,a_2,\dots,a_n$ and desired sum $0$, and consider the polyhedron $$P:=\{\vex\in \mathbb{R}^n: \vex{(i)}-\vex{(j)} \geq -1 \;\; \forall i \neq j, \;\; \vex \geq \veo \}.$$ The circuits of $P$ are precisely the vectors in $\{0,1\}^n \cup \{0,-1\}^n$. Now, intersect $P$ with a hyperplane orthogonal to $(a_1,\dots,a_n)$. The resulting polyhedron has a 0/1 circuit if and only if there is a nonzero 0/1 vector in $\ker(a_1,\dots,a_n)$; this is precisely the condition of the subset sum problem.

With that in mind, we now let $P$ refer to a polyhedron which is known to have 0/1 circuits. 
We say that $P$ satisfies \textit{0/1 circuit reachability} if for all ordered pairs of extreme-points $(\veu,\vev)$ of $P$, there exists a circuit walk from $\veu$ to $\vev$ in $P$ using only 0/1 circuits.  This too is necessary, but it does not suffice in order to conclude that there exist \textit{short} circuit walks using only 0/1 circuits, even when there exist short circuit walks between all pairs of vertices. An example is given in Figure \ref{fig:no_short_path}, which depicts a polytope $P$ which is the convex hull of the points $(0,0), (\varepsilon,0), (M+1+\varepsilon,M),$ and $(M+1+\varepsilon,M+\varepsilon)$ for some $\varepsilon$ and $M$ with $\varepsilon \ll M$.  Its only 0/1 circuits are $\pm(1,0)$ and $\pm(0,1)$,  and so any circuit walk from $(0,0)$ to $(M+1+\varepsilon,M+\varepsilon)$ using only 0/1 circuits requires a number of circuit steps at least on the order of $2\frac{M}{\varepsilon}$.  On the other hand, the circuit diameter of this polygon is two. 

This example demonstrates that the guaranteed existence of short paths using only 0/1 circuits is not only a property of the circuits themselves (which only depend on the constraint matrices of a given polyhedron, but not the right-hand-side vectors), but is also sensitive to the geometry of the polyhedron (i.e., also the right-hand-side vectors).  This should be unsurprising, but it does make explicitly clear the fact that, if one wants to characterize the conditions under which the 0/1 circuits of a polyhedron suffice to give short paths between all pairs of extreme-points, this characterization must involve a close interplay between the constraint matrices defining the polyhedron and the polyhedron itself. 

There is some additional and interesting nuance in this possible direction of research. Circuits inherently are a concept of linear programming, not integer programming (where the equivalent Graver bases would arise  \cite{g-75}). One commonly uses the circuits of the relaxation of an integer formulation of, say, a graph-theoretic problem. This allows one to traverse the relaxed polyhedron, which may have fractional extreme-points that do not appear in the integer hull of the original problem, as is the case with the fractional coloring polytope studied here.  Though we studied the circuits of the fractional coloring polytope, we restricted our attention to circuit walks that moved only between integral extreme points.  It is potentially interesting to characterize when it is possible to reach the fractional extreme-points of the LP relaxation of an IP, but still under a restriction to 0/1 circuits. 

Conversely, it may be more justified to restrict definitions like that of the proposed ``0/1 circuit reachability" to \textit{only} consider reachability and distance of 0/1 (or integral) extreme-points, and possibly further restricting the walks to only stop at integral points,  just as we did for the fractional coloring polytope. Note that for many problems, this could further serve the goal of maintaining interpretability, with these added restrictions enforcing interpretability not just of the circuits used, but of the intermediary solutions visited. Which of these regimes may be more interesting or useful could very well be context-dependent. This setup contrasts with the much stronger property satisfied by some polyhedra in which {\em all} circuit walks stop only at integral points \cite{bv-17}, such as polyhedra described with a totally-unimodular matrix.

Most of the above discussion on working only with the 0/1 circuits also holds for more general but still restricted subsets of the set of circuits---for example those adhering to fixed upper bounds on the absolutes of entries or on the circuit imbalance. In particular, it is easy to exhibit the zig-zagging behavior displayed in Figure \ref{fig:no_short_path} in more general settings. 
Further, the questions we pose are equally relevant to circuit augmentation. In \cite{dknv22,env-21}, the authors devised bounds on circuit diameters and the number of iterations of a circuit augmentation algorithm based on the circuit imbalance measure of the underlying set of circuits. These bounds become particularly strong when the circuit imbalance is low, like for totally-unimodular matrices. We see an exciting research direction in investigating under which circumstances these or similar bounds can be established when working with a restricted set of circuits of low imbalance (while the underlying problem still has a high imbalance). However, as with many results about the theoretical performance of circuit augmentation algorithms, the arguments in \cite{dknv22,env-21} crucially rely on using the \textit{whole} set of circuits, and therefore must parameterize by the imbalance of the whole set.  As such, we expect that new tools will have to be developed in order to explore this further.

\begin{figure}[htb]
\centering
    \begin{tikzpicture}
            \tikzset{vertex/.style={draw=none,fill=none,inner sep=0pt,minimum size=0pt}}
            \node[vertex, label=left:{$(0,0)$}] (1) at (0,0) {};
            \node[vertex, label={[xshift=.7cm,yshift=-.3cm]$(\varepsilon,0)$}] (2) at (1,0) {};
            \node[vertex, label=right:{$(M+1+\varepsilon,M)$}] (3) at (4,3) {};
            \node[vertex, label=right:{$(M+1+\varepsilon,M+\varepsilon)$}] (4) at (4,4) {};
            \draw[thick] (1)--(2)--(3)--(4)--(1) {};
            \draw[dashed] (1,0)--(1,1)--(2,1)--(2,2)--(3,2)--(3,3)--(4,3){};
            \end{tikzpicture}
    \caption{A polytope satisfying that 1) each extreme-point can be reached from each other extreme-point via a circuit walk using only 0/1 circuits, and 2) any circuit walk from $(0,0)$ to $(M+1+\varepsilon,M+\varepsilon)$ using only 0/1 circuits requires a number of steps at least on the order of $2\frac{M}{\varepsilon}$.  The dashed lines give an example of such a circuit walk.}
         \label{fig:no_short_path}
    \end{figure}

\vspace*{0.25cm}

\noindent\textbf{Acknowledgments.} 
We thank Alexander Black for insightful discussions on the topics of Section \ref{sec:col} and for private communication about NP-hardness of checking whether a polyhedron has a 0/1 circuit. 
\\\\
\noindent\textbf{Statement on the Use of Generative AI.}
No generative AI or AI-assisted technologies were used in the preparation of this manuscript.

\bibliographystyle{plain}

\begin{thebibliography}{10}

\bibitem{ach-21}
O.~Aichholzer, J.~Cardinal, T.~Huynh, K.~Knauer, T.~M{\"u}tze, R.~Steiner, and B.~Vogtenhuber.
\newblock Flip distances between graph orientations.
\newblock {\em Algorithmica}, 83(1):116--143, 2021.

\bibitem{afjkswy-25}
M.~Aprile, S.~Fiorini, G.~Joret, S.~Kober, M.~T. Seweryn, S.~Weltge, and Y.~Yuditsky.
\newblock Integer programs with nearly totally unimodular matrices: the cographic case.
\newblock In {\em Proceedings of the 2025 Annual ACM-SIAM Symposium on Discrete Algorithms (SODA)}, pages 2301--2312, 2025.

\bibitem{bb-many-24}
M.~M. Bayer, S.~Borgwardt, T.~Chambers, S.~Daugherty, A.~Dawkins, D.~Deligeorgaki, H.-C. Liao, T.~McAllister, A.~Morrison, G.~Nelson, and A.~R. Vindas-Meléndez.
\newblock Combinatorics of {G}eneralized {P}arking-{F}unction {P}olytopes.
\newblock {\em Discrete \& Computational Geometry}, doi 10.1007/s00454-025-00770-1, 2025.

\bibitem{bns-25}
A.~Black, C.~N\"obel, and R.~Steiner.
\newblock Short circuit walks in fixed dimension.
\newblock In {\em Proceedings of the 2026 Annual ACM-SIAM Symposium on Discrete Algorithms (SODA)}, pages 563--573, 2026.

\bibitem{b-19}
M.~Bonamy, N.~Bousquet, C.~Feghali, and M.~Johnson.
\newblock On a conjecture of {M}ohar concerning {K}empe equivalence of regular graphs.
\newblock {\em Journal of Combinatorial Theory, Series B}, 135:179--199, 2019.

\bibitem{bhi-20}
M.~Bonamy, M.~Heinrich, T.~Ito, Y.~Kobayashi, H.~Mizuta, M.~Mühlenthaler, A.~Suzuki, and K.~Wasa.
\newblock Diameter of colorings under {K}empe changes.
\newblock {\em Theoretical Computer Science}, 838:45--57, 2020.

\bibitem{bdfm-18}
S.~Borgwardt, J.~A. De~Loera, E.~Finhold, and J.~Miller.
\newblock The hierarchy of circuit diameters and transportation polytopes.
\newblock {\em Discrete Applied Mathematics}, 240:8--24, 2018.

\bibitem{bfh-14}
S.~Borgwardt, E.~Finhold, and R.~Hemmecke.
\newblock On the circuit diameter of dual transportation polyhedra.
\newblock {\em S{I}{A}{M} Journal on Discrete Mathematics}, 29(1):113--121, 2016.

\bibitem{bgkls-24}
S.~Borgwardt, W.~Grewe, S.~Kafer, J.~Lee, and L.~Sanit\`a.
\newblock On the hardness of short and sign-compatible circuit walks.
\newblock {\em Discrete Applied Mathematics}, 367:129--149, 2025.

\bibitem{bgl-23}
S.~Borgwardt, W.~Grewe, and J.~Lee.
\newblock On the combinatorial diameters of parallel and series connections.
\newblock {\em SIAM Journal on Discrete Mathematics}, 38(1):485--503, 2024.

\bibitem{bgl-24}
S.~Borgwardt, W.~Grewe, and J.~Lee.
\newblock On the diameter of a 2-sum of polyhedra.
\newblock {\em Optimization Letters}, 19:1053--1074, 2025.

\bibitem{bm-23}
S.~Borgwardt and A.~Morrison.
\newblock On combinatorial network flows algorithms and circuit augmentation for pseudoflows.
\newblock {\em Journal of Combinatorial Optimization}, 49:73, 2025.

\bibitem{bsy-18}
S.~Borgwardt, T.~Stephen, and T.~Yusun.
\newblock On the circuit diameter conjecture.
\newblock {\em Discrete {\&} Computational Geometry}, 60(3):558--587, 2018.

\bibitem{bv-19a}
S.~Borgwardt and C.~Viss.
\newblock Constructing clustering transformations.
\newblock {\em SIAM Journal on Discrete Mathematics}, 35(1):152--178, 2021.

\bibitem{bv-17}
S.~Borgwardt and C.~Viss.
\newblock Circuit walks in integral polyhedra.
\newblock {\em Discrete Optimization}, 44(1):100566, 2022.

\bibitem{bv-22}
S.~Borgwardt and C.~Viss.
\newblock A polyhedral model for enumeration and optimization over the set of circuits.
\newblock {\em Discrete Applied Mathematics}, 308:68--83, 2022.

\bibitem{c-64}
P.~Camion.
\newblock {\em Matrices Totalement Unimodulaires et Probl\`emes Combinatoires}.
\newblock PhD thesis, Universit\'e Libre de Bruxelles, 1964.

\bibitem{cs-23}
J.~Cardinal and R.~Steiner.
\newblock Inapproximability of shortest paths on perfect matching polytopes.
\newblock In {\em International Conference on Integer Programming and Combinatorial Optimization}, pages 72--86. Springer, 2023.

\bibitem{cmmz-02}
P.~Coll, J.~Marenco, I.~M\'endez D\'iaz, and P.~Zabala.
\newblock Facets of the graph coloring polytope.
\newblock {\em Annals of Operations Research}, 116:79--90, 2002.

\bibitem{ccps-97}
W.~Cook, W.~Cunningham, W.~Pulleyblank, and A.~Schrijver.
\newblock {\em Combinatorial Optimization}.
\newblock Wiley, 1997.

\bibitem{c-23}
D.~Cranston and C.~Feghali.
\newblock Kempe classes and almost bipartite graphs.
\newblock {\em Discrete Applied Mathematics}, 357:94--98, 2023.

\bibitem{deprs-26}
D.~Dadush, Eisenbrand F., Pinchasi R., Rothvoss T., and Singer N.
\newblock Excluding a line minor via design matrices and column number bounds for the circuit imbalance measure.
\newblock In {\em Proceedings of the 2026 Annual ACM-SIAM Symposium on Discrete Algorithms (SODA)}, pages 871--879, 2026.

\bibitem{dhnv-23}
D.~Dadush, S.~Huiberts, B.~Natura, and L.A. Végh.
\newblock A scaling-invariant algorithm for linear programming whose running time depends only on the constraint matrix.
\newblock In {\em STOC 2020: 52nd Annual ACM Symposium on Theory of Computing}, pages 761--–774, New York, NY, USA, 2023. Association for Computing Machinery.

\bibitem{dkno-24}
D.~Dadush, Z.~K. Koh, B.~Natura, N.~Olver, and L.~A. V\'{e}gh.
\newblock A strongly polynomial algorithm for linear programs with at most two nonzero entries per row or column.
\newblock In {\em STOC 2024: 56th Annual ACM Symposium on Theory of Computing}, pages 1561--1572, New York, NY, USA, 2024. Association for Computing Machinery.

\bibitem{dknv22}
D.~Dadush, Z.K. Koh, B.~Natura, and L.A. Végh.
\newblock On circuit diameter bounds via circuit imbalances.
\newblock In {\em IPCO 2022: Integer Programming and Combinatorial Optimization, 23rd International Conference}, pages 140--153. Springer International Publishing, 2022.

\bibitem{dhl-15}
J.~A. De~Loera, R.~Hemmecke, and J.~Lee.
\newblock On augmentation algorithms for linear and integer-linear programming: from {E}dmonds-{K}arp to {B}land and beyond.
\newblock {\em SIAM Journal on Optimization}, 25(4):2494--2511, 2015.

\bibitem{dks-19}
J.~A. De~Loera, S.~Kafer, and L.~Sanit\`{a}.
\newblock Pivot rules for circuit-augmentation algorithms in linear optimization.
\newblock {\em SIAM Journal on Optimization}, 32(3):2156--2179, 2022.

\bibitem{env-21}
F.~Ekbatani, B.~Natura, and L.A. Végh.
\newblock Circuit imbalance measures and linear programming, 2022.
\newblock Pages 64--114, In: Surveys in Combinatorics 2022, London Mathematical Society Lecture Note Series, Cambridge University Press, A. Nixon and S. Prendiville, eds.

\bibitem{f-17}
C.~Feghali, M.~Johnson, and D.~Paulusma.
\newblock Kempe equivalence of colourings of cubic graphs.
\newblock {\em European Journal of Combinatorics}, 59:1--10, 2017.

\bibitem{f-14}
E.~Finhold.
\newblock {\em Circuit diameters and their application to transportation problems}.
\newblock PhD thesis, Technische Universit\"{a}t M\"{u}nchen, 2014.

\bibitem{gdl-14}
J.~B. Gauthier, J.~Desrosiers, and M.~L\"ubbecke.
\newblock Decomposition theorems for linear programs.
\newblock {\em Operations Research Letters}, 42(8):553--557, 2014.

\bibitem{g-75}
J.~E. Graver.
\newblock On the foundation of linear and integer programming {I}.
\newblock {\em Mathematical Programming}, 9:207--226, 1975.

\bibitem{h-21}
R.~Hlad\'ik.
\newblock Combinatorial algorithms for flow problems.
\newblock Bachelor thesis, Charles University, 2021.

\bibitem{ikkko-22}
T.~Ito, N.~Kakimura, N.~Kamiyama, Y.~Kobayashi, and Y.~Okamoto.
\newblock Shortest reconfiguration of perfect matchings via alternating cycles.
\newblock {\em SIAM Journal on Discrete Mathematics}, 37(2), 2022.

\bibitem{k-22}
S.~Kafer.
\newblock {\em Polyhedral Diameters and Applications to Optimization}.
\newblock PhD thesis, University of Waterloo, 2022.

\bibitem{kps-19}
S.~Kafer, K.~Pashkovich, and L.~Sanit\`{a}.
\newblock On the circuit diameter of some combinatorial polytopes.
\newblock {\em SIAM Journal on Discrete Mathematics}, 33(1):1--25, 2019.

\bibitem{k-72}
R.~M. Karp.
\newblock Reducibility among combinatorial problems.
\newblock In R.~E. Miller and J.~W. Thatcher, editors, {\em Complexity of Computer Computations}, pages 85--103. Plenum Press, 1972.

\bibitem{v-81}
M.~Las~Vergnas and H.~Meyniel.
\newblock Kempe classes and the {H}adwiger conjecture.
\newblock {\em Journal of Combinatorial Theory, Series B}, 31(1):95--104, 1981.

\bibitem{l-16}
R.~M.~R. Lewis.
\newblock {\em A Guide to Graph Colouring: Algorithms and Applications}.
\newblock Springer International Publishing, 2016.

\bibitem{ns-25}
C.~N\"obel and R.~Steiner.
\newblock Complexity of polytope diameters via perfect matchings.
\newblock In {\em Proceedings of the 2025 Annual ACM-SIAM Symposium on Discrete Algorithms (SODA)}, pages 2234--2251, 2025.

\bibitem{o-10}
S.~Onn.
\newblock {\em Nonlinear Discrete Optimization}.
\newblock Zurich Lectures in Advanced Mathematics. European Mathematical Society, 2010.

\bibitem{oxley-06}
J.~Oxley.
\newblock {\em Matroid Theory}, volume~21 of {\em Oxford Graduate Texts in Mathematics}.
\newblock Oxford University Press, Oxford, second edition, 2011.

\bibitem{r-69}
R.~T. Rockafellar.
\newblock The elementary vectors of a subspace of {$\mathbb{R}^{N}$}.
\newblock In {\em Combinatorial {M}athematics and its {A}pplications}, pages 104--127. University of North Carolina Press, 1969.

\bibitem{t-84}
D.~M. Topkis.
\newblock Adjacency on polymatroids.
\newblock {\em Mathematical Programming}, 30(2):229--237, 1984.

\bibitem{t-92}
D.~M. Topkis.
\newblock Paths on polymatroids.
\newblock {\em Mathematical Programming}, 54:335--351, 1992.

\bibitem{t-65}
W.~T. Tutte.
\newblock Lectures on matroids.
\newblock {\em Journal of Research of the National Bureau of Standards-B. Mathematics and Mathematical Physics}, 69B(1), 1965.

\end{thebibliography}

\appendix
\section{Additional Proofs}
\label{app:imbal}

We provide proofs for Theorems \ref{thm:cycle_imbalance} and \ref{thm:star_imbalance}.
\cycleimbal*
\begin{proof}
    Consider the graph $G_k$ with vertex set $V=\{u_i,v_i:0\leq i\leq k\}$ and edge set
    \[    E=\{u_iv_i,u_iu_{i+1},u_iv_{i+1},v_iu_{i+1},v_iv_{i+1}:0\leq i\leq k-1\}\cup\{u_kv_k\}.
    \]
    Let $\veg$ be such that $\veg(u_iv_i)=(-2)^{-i}$ for all $0\leq i\leq k$, 
    \[
    \veg(u_{i-1}u_{i})=\veg(v_{i-1}u_{i})=(-2)^{-i}
    \] 
    for all $1\leq i\leq k$, and 
    \[
    \veg(u_{i-1}v_i)=\veg(v_{i-1}v_i)=0
    \]
    for all $1\leq i\leq k$. See Figure \ref{fig:cycle_imbalance} for such a construction of $\veg$.

    There exists a circuit $\veh$ with $\supp(B\veh)\subseteq\supp(B\veg)$.  We will show that for any such circuit, $|\veh(u_iv_i)|=2|\veh(u_{i+1}v_{i+1})|$ for all $0\leq i\leq k-1$. We have that for all $0\leq i\leq k-1$,
    \begin{align*}
    \veg(\{u_iv_i,v_iu_{i+1},u_iu_{i+1}\})=0,&\\    
    \veg(\{u_iv_i,v_iu_{i+1},u_{i+1}v_{i+1},u_iv_{i+1}\})=0,&\text{ and}\\
    \veg(\{u_iv_i,v_iv_{i+1},v_{i+1}u_{i+1},u_iu_{i+1}\})=0.&
    \end{align*}

     Since $\supp(B\veh)\subseteq\supp(B\veg)$, we also have that for all $0\leq i\leq k-1$

     \begin{align*}
    \veh(\{u_iv_i,v_iu_{i+1},u_iu_{i+1}\})=0,&\\    
    \veh(\{u_iv_i,v_iu_{i+1},u_{i+1}v_{i+1},u_iv_{i+1}\})=0,&\text{ and}\\
    \veh(\{u_iv_i,v_iv_{i+1},v_{i+1}u_{i+1},u_iu_{i+1}\})=0.&
    \end{align*}
    
    If for all $0\leq i\leq k-1$ we let $a=-\veh(u_iv_i)$, the above implies that
    \begin{alignat*}{6}
    \veh(u_iu_{i+1}) &+& \veh(v_iu_{i+1}) & &              &= a,& \\
                 & & \veh(v_iu_{i+1}) &+& \veh(u_{i+1}v_{i+1}) &= a,& \text{ and}\\
    \veh(u_iu_{i+1}) & &              &+& \veh(u_{i+1}v_{i+1}) &= a.
    \end{alignat*}
    By (\ref{eq:all_half}), $\veh(u_iu_{i+1})=\veh(v_iu_{i+1})=\veh(u_{i+1}v_{i+1})=\frac{1}{2}a$.  It follows that $|\veh(u_iv_i)|=2|\veh(u_{i+1}v_{i+1})|$ for all $1\leq i\leq k-1$.  Thus, $\kappa(\mh)\geq\Omega(2^{|E|})$, as desired.
\end{proof}

\starimbal*
\begin{proof}
    Consider the graph $G_k$ with vertex set $V=\{u_i,v_i:0\leq i\leq k\}$ and edge set 
    \[
    E=\{u_iv_i,v_iu_{i+1}v_iv_{i+1}:0\leq i\leq k-1\}\cup \{u_kv_k\}
    \]
    Let $\veg$ be such that $\veg(u_iv_i)=2^{-i}$ and $\veg(v_{i-1}u_i)=\veg(v_{i-1}v_i)=-(2^{-i})$ for $1\leq i\leq k$. See Figure \ref{fig:star_imbalance} for such a construction of $\veg$.

     There exists a circuit $\veh$ with $\supp(B\veh)\subseteq\supp(B\veg)$.  We will show that for any such circuit, $\veh(u_iv_i)=2\veh(u_{i+1}v_{i+1})$ for all $0\leq i\leq k-1$. We have that for all $0\leq i\leq k-1$,
    \begin{align*}
    \veg(\{u_iv_i,v_iu_{i+1},v_iv_{i+1}\})=0,&\\    
    \veg(\{v_iu_{i+1},u_{i+1}v_{i+1}\})=0,&\text{ and}\\
    \veg(\{v_iv_{i+1},u_{i+1}v_{i+1}\})=0.&
    \end{align*}

    Since $\supp(B\veh)\subseteq\supp(B\veg)$, we also have that for all $0\leq i\leq k-1$

    \begin{align}
    \veh(\{u_iv_i,v_iu_{i+1},v_iv_{i+1}\})=0,&\label{f1}\\    
    \veh(\{v_iu_{i+1},u_{i+1}v_{i+1}\})=0,&\text{ and}\label{f2}\\
    \veh(\{v_iv_{i+1},u_{i+1}v_{i+1}\})=0.\label{f3}&
    \end{align}

    By equations (\ref{f2}) and (\ref{f3}) we have that for $0\leq i\leq k-1$, 
    \[
    -\veh(v_iu_{i+1}) = \veh(u_{i+1}v_{i+1}) = - \veh(v_iv_{i+1}),
    \]
    and so by equation (\ref{f1}) we have that for $0\leq i\leq k-1$
    \[
    \veh(u_iv_i) = - (\veh(v_iu_{i+1}) + \veh(v_iv_{i+1})) = 2\veh(u_{i+1}v_{i+1}).
    \]
    Thus, $\kappa(\mh)\in\Omega(2^{|E|})$, as desired.
\end{proof}

\section{Pseudo-Alternating Paths and Cycles} \label{app:pseudo}

We extend the characterization of (rooted) alternating paths and cycles as 0/1 circuits to {\em pseudo-alternating paths} and {\em pseudo-alternating cycles}, where single edges of a given sign are replaced with paths of that sign. We call these types of paths and cycles pseudo-alternating because they can be contracted to true alternating paths and cycles; Figure \ref{fig:pseudo_exs} shows an example.

\begin{definition}[Pseudo-alternating Structure]\label{def:pseudo-alt}
A pseudo-alternating structure on a path or cycle $H$ is a sequence of subpaths $(H_i)_{i \in \Z}$ of $H$ such that all $e \in E(H_i)$ are $+$-edges for all even $i$ and all $f \in E(H_i)$ are $-$-edges for all odd $i$. Each subpath can be contracted to a single $+$- or $-$-edge to form a true alternating structure (on the appropriate contraction of $H$).  
\end{definition}

\begin{figure}[htb]
\centering
\begin{subfigure}{0.49\textwidth}
    \centering
    \begin{tikzpicture}
            \tikzset{vertex/.style = {shape=circle,draw,minimum size=2em}}
            \node[vertex] (1) at (0,0) {1};
            \node[vertex] (2) at (1.5,0) {2};
            \node[vertex] (3) at (3,0) {3};
            \node[vertex] (4) at (4.5,0) {4};
            \node[vertex] (6) at (3.75,-1.5) {5};
            \node[vertex] (7) at (2.25,-1.5) {6};
            \node[vertex] (8) at (0.75,-1.5) {7};
            \draw[thick] (1)--(2) {};
            \draw[thick, dashed] (2)--(3) {};
            \draw[thick,dashed] (3)--(4) {};
            \draw[thick] (4)--(6) {};
            \draw[thick,dashed] (6)--(7) {};
            \draw[thick,dashed] (7)--(8) {};
            \draw[thick] (8)--(1) {};
            \end{tikzpicture}
    \caption{Pseudo-alternating cycle}
    \label{fig:pseudo-root}
\end{subfigure}
\hfill
\begin{subfigure}{0.49\textwidth}
    \centering
    \begin{tikzpicture}
            \tikzset{vertex/.style = {shape=circle,draw,minimum size=2em}}
            \node[vertex] (2) at (1.5,0) {2};
            \node[vertex] (4) at (4.5,0) {4};
            \node[vertex] (6) at (3.75,-1.5) {5};
            \node[vertex] (8) at (0.75,-1.5) {7};
            \draw[thick,dashed] (2)--(4) {};
            \draw[thick] (4)--(6) {};
            \draw[thick,dashed] (6)--(8) {};
            \draw[thick] (8)--(2) {};
            \end{tikzpicture}
    \caption{Alternating cycle after contractions.}
    \label{fig:cont-pseudo-root}
    \end{subfigure}
        \caption{A pseudo-alternating cycle and its underlying structures after contraction. Edges of the same line pattern belong to the same forest.}
            \label{fig:pseudo_exs}
    \end{figure}

\begin{lemma}\label{lem:pseudo}
Let $\veg$ be a mixed-sign 0/1 vector, and let the edges in $\supp(\veg)$ form one of the following structures in $G$:
\begin{enumerate}
    \item a pseudo-alternating path with at least 4 subpaths $H_i$ all having length at least 2, or
    \item a pseudo-alternating cycle with at least 4 subpaths $H_i$ where $|E(H_i)| \geq 2$ for even $i$ and $|E(H_i)| \geq 3$ for odd $i$. 
\end{enumerate}
Then $\veg$ is a circuit.
\end{lemma}

\begin{proof}
Suppose $\veg$ is not a circuit and let $\vey\neq\veo$ be a vector such that $\supp(B\vey)\subsetneq\supp(B\veg)$.  Note that this implies that if $(e,f)$ is a balanced pair with respect to $\veg$, then we have that $\vey(e)+\vey(f)=0$.  We will leverage this to show that $\vey$ is, in fact, just a scalar multiple of $\veg$, contradicting either that $\vey\neq\veo$ or that $\supp(B\vey)\subsetneq\supp(B\veg)$.
Specifically, we will show that for \textit{any} $-$-edge $e$ and $+$-edge $f$, $\vey(e)=-\vey(f)$ (even if they do not form a balanced pair). This will prove the claim.  Note that if $e=tu$ is a $-$-edge and $f=vw$ is a $+$ edge, then $(e,f)$ is a balanced pair unless $E[\{t,u,v,w\}]$ is a path of length 3, i.e., one of $t$ or $u$ is adjacent to one of $v$ or $w$.

Let $e_0\in \supp(\veg)$, and without loss of generality suppose it is a $-$-edge. Then there exist at most two $+$-edges $f_1,f_2 \in \supp(\veg)$ such that $(e_0, f_1)$ and $(e_0, f_2)$ are not balanced pairs, i.e., $(e_0, f)$ is a balanced pair for all $+$-edges $f$ such that $f \neq f_1, f_2$. In fact, all edges in $\supp(\veg)$ are contained in at least two balanced pairs due to the minimum length requirements for the pseudo-alternating structures in the statement.

First suppose only $f_1$ exists, i.e., there is only one $+$-edge such that $(e_0,f_1)$ is not a balanced pair. 
By the minimum length requirements, there exists a $-$-edge $e_1 \neq e_0 \in \supp(\veg)$ such that $(e_1,f_1)$ is a balanced pair and there exists a $+$-edge $f_3 \neq f_1$ such that $(e_1,f_3)$ is a balanced pair. Because $f_1$ is the only $+$-edge that does not form a balanced pair with $e_0$, $(e_0, f_3)$ is a balanced pair. 
Then we have that $\vey(f_1)=-\vey(e_1)=\vey(f_3)=-\vey(e_0)$.  Thus, for all $+$-edges $f\in\supp(\veg)$, we have that $\vey(e_0)=-\vey(f)$, as desired.

Now, suppose that both $f_1$ and $f_2$ exist (and see Figure \ref{fig:alternating_proof}). Then $f_1$ and $f_2$ are both distance 2 from $e_0$ in $\supp(\veg)$.  Since each subpath has length at least two, without loss of generality the edge $e_1$ between $f_1$ and $e_0$ in $\supp(\veg)$ is a $-$-edge.  Likewise, the other edge $f_3$ incident to $f_1$ in $\supp(\veg)$ is a $+$-edge.  We need to show that $\vey(f_1)=-\vey(e_0)=\vey(f_2)$.   

\textbf{Case 1: } Suppose $\supp(\veg)$ is a pseudo-alternating cycle, and let $e_0$ lie on a subpath $H_i$ of the cycle. Because $e_0$ lies on subpath $H_i$, $f_1$ and $f_2$ must lie on distinct subpaths $H_{i-1}$ and $H_{i+1}$, respectively.  Without loss of generality, let $f_1\in H_{i-1}$ and $f_2 \in H_{i+1}$.

If  $e_0 \in H_i$ for even $i$, then $f_1$ and $f_2$ are in subpaths where $i \pm 1$ is odd and there are at least 3 $+$-edges in each of those subpaths. As $|E(H_i)|\geq 2$ and by existence of $-$-edge $e_1$, we know that $f_1$ is the closest edge in subpath $H_{i-1}$ to $e_0$ . By $|E(H_{i+1})|\geq 3$, we further know that $f_2$ is not the furthest edge in subpath $H_{i+1}$ from $e_0$. Thus, there exists a $+$-edge $f_4$ which is adjacent to $f_2$ and not along the path connecting $e_0$ to $f_2$. We have that $(e_0, f_4)$, $(e_1, f_4)$, $(e_1,f_1)$, and $(e_1,f_2)$ are balanced pairs, and so $\vey(f_1) = \vey(f_2) = -\vey(e_0)$, as desired. If $e_0 \in H_i$ for odd $i$, then we know both edges incident to $e_0$ in $\supp(\veg)$ are $-$-edges and $|E(H_i)| = 3$, since $|E(H_i)| \geq 3$ and since both $f_1$ and $f_2$ are distance $2$ from $e_0$. Recall that the edge between $e_0$ and $f_1$ is denoted $e_1$, and let the edge between $e_0$ and $f_2$ be denoted $e_2$. Since all subpaths have length at least two, $f_2$ is incident to a $+$-edge $f_4$, and $f_1$ is incident to a $+$-edge $f_3$. We have that $(e_1,f_1)$, $(e_1,f_4)$, and $(e_0,f_4)$ are balanced pairs, and so $\vey(f_1)=-\vey(e_0)$.  Likewise we have that $(e_2,f_2)$, $(e_2,f_3)$, and $(e_0,f_3)$ are balanced pairs, and so $\vey(f_2)=-\vey(e_0)$, as desired.

\textbf{Case 2:} Suppose $\supp(\veg)$ is a pseudo-alternating path, and let $e_0$ lie on a subpath $H_i$ of the path. If the edge between $e_0$ and $f_2$ is a $-$-edge, then the argument for Case 1 works.  Thus, we now assume that the edge between $e_0$ and $f_2$ is a  $+$-edge $f_4$.  Since there are at least four subpaths, there exists a $-$-edge $e_2$ such that both $(e_2,f_1)$ and $(e_2,f_4)$ are balanced pairs.  Then we have that $(e_0,f_4)$, $(f_4,e_2)$, and $(e_2,f_1)$ are balanced pairs, and so $\vey(f_1)=-\vey(e_0)$.  Moreover, $(f_1,e_1)$ and $(e_1,f_2)$ are balanced pairs, and so $\vey(f_2)=\vey(f_1)=-\vey(e_0)$, as desired.
\end{proof}

\begin{figure}[htb]
        \centering
        \begin{tikzpicture}[scale = 0.75]
            \tikzset{vertex/.style = {shape=circle,draw,minimum size=1em}}
            \node[vertex] (1) at (0,0) {};
            \node[vertex] (2) at (2,0) {};
            \node[vertex] (3) at (4,0) {};
            \node[vertex] (4) at (6,0) {};

            \node[vertex] (a) at (-2,0) {};
            \node[vertex] (b) at (-4,0) {};

            \node[vertex] (c) at (8,0) {};
            \node[vertex] (d) at (10,0) {};

             \draw[thick,densely dashed] (1) to node[midway, fill=white] {$e_1$} (2) {};
             \draw[thick,densely dashed] (2) to node[midway, fill=white] {$e_0$} (3) {};
             \draw[thick,densely dotted] (3) to (4) {};

             \draw[thick] (1) to node[midway, fill=white] {$f_1$} (a) {};
             \draw[thick] (a) to node[midway, fill=white] {$f_3$} (b) {};

             \draw[thick] (4) to node[midway, fill=white] {$f_2$} (c) {};
             \draw[thick, densely dotted] (c) to (d) {};
            
            \end{tikzpicture}
            \caption{The case in the proof of Lemma \ref{lem:pseudo} wherein both $f_1$ and $f_2$ exist. A depiction of the edges of distance at most 3 from $e_0$ in $\supp(\veg)$. Solid edges are $+$-edges and dashed edges $e_i$ are $-$-edges. Dotted edges could be either $+$- or $-$-edges, but at least one of them is a $+$-edge.}
            \label{fig:alternating_proof}
    \end{figure}

Note that the length requirements in Lemma \ref{lem:pseudo} apply both to the number of subpaths and to the minimum length of each subpath. These technical conditions ensure not only that the graph is large enough for the balanced pairs used in the proof to exist, but also that there are few edges like $f_1$ which one needs to consider.
What matters is not just the diameter of the graph formed from $\supp(\veg)$, but also the distance between edges of opposite sign. This distance determines whether all edges of opposite sign are far enough apart to form balanced pairs, such that one can connect any $+$-edge $e$ and $-$edge $f$ via said balanced pairs and show $\vey(e) = -\vey(f)$.

\end{document}